\documentclass{article}

\usepackage[utf8]{inputenc}
\usepackage{a4wide}
\usepackage{amsmath, amsthm, amsfonts,amssymb}
\usepackage{bm}
\usepackage{fancyhdr}
\usepackage{authblk}
\usepackage{graphicx, color}
\usepackage{subfigure}
\usepackage{epstopdf}
\usepackage{float}
\usepackage{hyperref}
\hypersetup{colorlinks,linkcolor={blue},citecolor={blue},urlcolor={red}}  

\usepackage[margin=1in]{geometry}
\makeatletter

\usepackage{color}

\usepackage{xcolor}
\usepackage{mathrsfs}
\usepackage{stmaryrd} 
\theoremstyle{plain}
\newtheorem{thm}{Theorem}[section]
\newtheorem{cor}[thm]{Corollary}
\newtheorem{lem}[thm]{Lemma}

\newtheorem{prop}[thm]{Proposition}
\newtheorem{defn}[thm]{Definition}
\newtheorem{assum}[thm]{Assumption}
\newtheorem{rem}[thm]{Remark}

\newcommand{\dd}{\mathrm{d}}
\newcommand{\ii}{\mathrm{i}}
\newcommand{\rO}{\mathrm{O}}
\newcommand{\bv}{\mathbf{v}}
\newcommand{\bw}{\mathbf{w}}

\numberwithin{equation}{section}

\date{}
\begin{document}

\title{Global and local CLTs for linear spectral statistics of general sample covariance matrices when the dimension is much larger than the sample size with applications}

\author[1]{Xiucai Ding \thanks{E-mail: xcading@ucdavis.edu. }}
\author[1]{Zhenggang Wang \thanks{Email: zggwang@ucdavis.edu}}
\affil[1]{Department of Statistics, University of California, Davis}

\maketitle


\begin{abstract} 
In this paper, under the assumption that the dimension is much larger than the sample size, i.e., $p \asymp n^{\alpha}, \alpha>1,$ we consider the (unnormalized) sample covariance matrices  $Q = \Sigma^{1/2} XX^*\Sigma^{1/2}$, where $X=(x_{ij})$ is a $p \times n$ random matrix with centered i.i.d entries whose  variances are $(pn)^{-1/2}$, and $\Sigma$ is the deterministic population covariance matrix.  We establish two classes of central limit theorems (CLTs) for the linear spectral statistics (LSS) for $Q,$ the global CLTs on the macroscopic scales  and the local CLTs on the mesoscopic scales. We prove that the LSS converge to some Gaussian processes whose mean and covariance functions depending on $\Sigma$, the ratio $p/n$ and the test functions, can be identified explicitly on both macroscopic and mesoscopic scales. We also show that even though the global CLTs depend on the fourth cumulant of $x_{ij},$ the local CLTs do not.  Based on these results, we propose two classes of statistics for testing the structures of $\Sigma,$ the global statistics and the local statistics, and analyze their superior power under general local alternatives. To our best knowledge, the local LSS testing statistics which do not rely on the fourth moment of $x_{ij},$ is used for the first time in hypothesis testing while the literature mostly uses the global statistics and requires the prior knowledge of the fourth cumulant. Numerical simulations also confirm the accuracy and powerfulness of our proposed statistics and illustrate  better performance compared to the existing methods in the literature.

\end{abstract}


\section{Introduction}\label{sec_introduction}
Covariance matrices are fundamental objects in multivariate analysis \cite{anderson2009introduction} and high dimensional statistics \cite{yao2015large}. Many statistical methodologies and techniques rely on the knowledge of the structure of the covariance matrix, to
name but a few, Principal Component Analysis, Discriminant Analysis and
Cluster Analysis. In order to correctly apply these tools, people usually need to conduct hypothesis testings to understand the population covariance structures. An important problem of significant interest is to test
\begin{equation}\label{eq_generaltesting}
\mathbf{H}_0: \Sigma=\Sigma_0,
\end{equation}
where $\Sigma$ is the population covariance matrix and $\Sigma_0$ is some pre-given positive definite matrix. Even though many classic statistics including the likelihood ratio statistic and Nagao statistic \cite{anderson2009introduction,Nagao1973} are useful in the setting when the dimension $p$ is fixed and the sample size goes to infinity, they lose their validity when $p$ diverges with $n.$ To overcome the curse of high dimensionality, many modified or novel statistics have been proposed in the last decades, to name but a few \cite{CaiMa2013,ZhangHuBai2019,Chen2020,ZBY,Mao2017,Butucea2016,BaiJiangYaoZheng2009,Srivastava2005,JiangYang2013,LedoitWolf2002,QiuPreprint,Ahmad2015,Fisher2012,Srivastava2011,Chen2010}. For a comprehensive review, we refer the readers to \cite{Cai2017}. 

We point out that the aforementioned statistics can be roughly divided into two classes. The first classes of the statistics rely on some $U$-statistics like \cite{CaiMa2013,Ahmad2015,Chen2010,Butucea2016,Mao2017}. Even though $U$-statistics can be made dimension-free, there exist several disadvantages. First, it is computationally expensive. Second, it usually requires the prior knowledge of the fourth cumulant of the random variables. Third, it usually needs a relatively  strong alternative to make the statistics powerful. The second classes of the statistics utilize the linear spectral statistics (LSS) of some sample covariance matrices like \cite{Fisher2012, Chen2020, QiuPreprint,ZhangHuBai2019,Srivastava2011,JiangYang2013,ZBY,BaiJiangYaoZheng2009}. The statistics' straightforward constructions make it highly computationally efficient. However, all the existing literature requires the prior knowledge of the fourth cumulant of the random vectors and most of the existing literature focuses on a specific regime that $p$ and $n$ are comparably large and do not apply to the setting when $p$ is much larger than $n$.  

In this paper, we aim to propose some LSS based statistics to simultaneously  address several aforementioned challenges. First, the proposed statistics are accurate in the regime that $p$ is much large than $n$ in the sense that  
\begin{equation}\label{eq_ratioassumption}
p \asymp n^{\alpha}, \  \text{for some constant} \ 1<\alpha<\infty.  
\end{equation}
Second, the proposed statistics do not necessarily rely on the cumulants of the random vectors. Finally, the proposed statistics should be powerful under weak local alternatives.

\subsection{Some related results on LSS of sample covariance matrices}
 
In this section, we first pause to give a brief review of the literature on results related to linear spectral statistics especially in the context of sample covariance matrices. In the literature of random matrix theory, two types of CLTs for LSS have been established: the global (a.k.a. macroscopic) CLTs and the local (a.k.a  mesoscopic) CLTs. The global CLTs are computed on the global scales so that all the eigenvalues will be utilized, while the local CLTs are calculated on the local scales that only  a relatively small part of the eigenvalues will be counted. 

Most of the results are established assuming that $p$ and $n$ are comparably large.  On the global scales, the CLTs for linear eigenvalue statistics were first studied in \cite{JONSSON19821} for Wishart matrices. Later on, the CLTs for linear eigenvalues statistics with analytic test functions for general sample covariance matrices were studied in \cite{Bai2004}. Then the regularity conditions on the test functions were weakened under various settings in \cite{Bai2010,Lytova2009b,Shcherbina2011,Najim2016}. The ideas and results were also applied to study various hypothesis testing problems in \cite{Li2016,Zhang2022,ZhangHuBai2019,Chen2020,Li2014,Bodnar2019, Jiang2016,BaiJiangYaoZheng2009,JiangYang2013,Debapratimma2022,Li2018,Zheng2017,Zheng2012,ZBY,Hu2019,Wang2014,LiYinZheng2021,Zou2022}. We point out that  the ideas have also been borrowed to study other types of random matrices related to sample covariance  matrices, for example, separable sample covariance matrices \cite{Bai2019,LiYao2018}, sample correlation matrices \cite{Yin2023,ZhengZhu2019,Gao2017} and Kendall's tau correlation matrices \cite{LiWangLi2021}. Furthermore, in addition to the eigenvalue LSS, the eigenvector LSS of the sample covariance matrices have also been studied in \cite{BMP, DT, xia2015functional,Yang2020}. Much less work has been done on the local scales. To our best knowledge, even though lots of works have been done related Wigner matrices \cite{He2017,Landon2020,Li2021,LiXu2021}, fewer concern the sample covariance matrices except for \cite{Li2021,Yang2020}.

Under our concerned regime (\ref{eq_ratioassumption}), to our best awareness, the very limited existing works are all global CLTs \cite{BAObao, Chen2015,QiuPreprint}. Moreover, instead of analyzing the sample covariance matrices, they work with a normalized sample covariance matrices
\begin{equation}\label{eq)A} 
A=\mathsf{c}_1 X^* \Sigma X-\mathsf{c}_2\sqrt{p/n} I,
\end{equation}
for some constants $\mathsf{c}_1$ and $\mathsf{c}_2.$ On the local scales, there does not exist any work. Motivated by these, our current work aims to fill these gaps by establishing both global and local CLTs under the setup (\ref{eq_ratioassumption}) considering the following more standard (unnormalized) sample covariance matrix 
\begin{equation}\label{eq_samplecovaraincematrix}
Q=\Sigma^{1/2} XX^* \Sigma^{1/2},
\end{equation} 	
where $X=(x_{ij})$ is a $p \times n$ random matrix with centered i.i.d entries whose  variances are $(pn)^{-1/2}$, and $\Sigma$ is the deterministic population covariance matrix.

\subsection{An overview of our results and contributions}
In this section, we provide a rough overview of our results and contributions. One purpose of this paper is to generalize both the global and local CLTs for $p \asymp n$ to the ultra high dimensional setting (\ref{eq_generaltesting}). More concretely, we will study the LSS for the sample covariance matrix  $Q$ in (\ref{eq_samplecovaraincematrix}) on both the global scales and local scales. We first show that the limiting ESD (LSD) of $Q$ still follows a deformed Marchenco-Pastur (MP) law as in the spirit of \cite{LL} which only handles $\Sigma=I.$ Even though the edges of the LSD diverge, it is supported on a single interval and the length is bounded once $\Sigma$ is a bounded positive definite matrix; see Lemma \ref{lem_proofgloballawresult}.  This is quite different from the comparable setting $p \asymp n$ that  people usually need additional regularity conditions on $\Sigma$ and the LSD may support on several disjoint intervals \cite{AILL}.

For CLTs on the global scales, we prove that the joint distribution of LSS indexed by different test functions converge to some Gaussian processes whose mean and covariance functions can be identified explicitly and depend on the fourth cumulant of $x_{ij}$; see Theorem \ref{thm_mainone}. The literature mostly focuses on studying CLTs under the regime $p \asymp n$ with stronger assumptions on the test functions, for example see \cite{yao2015large} for a review. The only exceptions are \cite{BAObao,Chen2015,QiuPreprint} where the authors focus on a normalized sample covariance matrix in (\ref{eq)A}).  First, the LSD of $A$ follows Wigner's semicircle law instead of MP law. Second, to facilitate statistical applications, they usually require stronger assumptions on $\alpha$ or $\Sigma$. For example, \cite{BAObao,Chen2015} assumes $\Sigma=I$ and \cite{QiuPreprint} needs $\alpha \geq 2.$  For CLTs on the local scales, we prove that the joint distribution of the LSS is also Gaussian. However, the mean and covariance functions do not depend on the fourth cumulant of $x_{ij};$ see Theorem \ref{thm_maintwo}. On the local scales, the most relevant work in the literature is \cite{Li2021} which only handles a single test function and requires $p \asymp n.$ To our best knowledge, it is the first time that the local CLTs are established under the regime (\ref{eq_ratioassumption}) and for multiple functions.  
 
To ease applications and numerical calculations, we further simplify the formulas of the mean and covariance functions for some commonly used test functions in Corollaries \ref{cor_globalexamples} and \ref{cor_localexample}.  
The theoretical results are not only interesting and natural on their own, they are also highly motivated by the statistical inference problem (\ref{eq_generaltesting}). In the related literature, people usually use global CLTs to construct statistics to test (\ref{eq_generaltesting}), see \cite{yao2015large} for the $p \asymp n$ setting and \cite{QiuPreprint} for (\ref{eq_ratioassumption}) with $\alpha \geq 2.$ A disadvantage is that people need to know the fourth cumulant of $x_{ij}$ in order to apply these tests. Motivated by this challenge, we propose two classes of testing statistics, global statistics and local statistics, in Section \ref{sec_stattheory} where the local testing statistics are moment free in the sense that it does not depend on fourth cumulant of $x_{ij}$. These statistics are not only accurate as in Corollary \ref{coro_statisticalapplications} but also powerful as in Corollary \ref{cor_poweranalysis} under weak local alternatives.

Finally, our proof strategies rely on the devices of cumulant expansions, Helffer-Sj{\" o}strand formula and Wick's probability theorem which follow and generalize those of \cite{Li2021, Yang2020}. We refer the readers to Section \ref{sec_examproofsrategy} for more details.  We point out that the key ingredients are the local laws (c.f. Theorem \ref{thm_locallaw}) for $Q$ under (\ref{eq_ratioassumption}) which is of self-interest and can be used to study other problems like spiked sample covariance matrices in the regime (\ref{eq_ratioassumption}).

The rest  of the paper is organized as follows. In Section \ref{sec_modelandlocalaw}, we introduce the model and the asymptotic global laws. In Section \ref{sec_mainresultandlocalaw}, we present our main results and outline the strategies and for the proofs.  Section \ref{sec_statapplication}  discusses the statistical applications. The remainder of the paper is devoted to the details of the proof. In Section \ref{sec_preliminary}, we summarize and provide some of the basic tools which will be used in the proofs. In Section \ref{sec_CLTresolvent}, we prove the CLTs for the resolvents which is an intermediate result for proving our final CLTs. In Section \ref{sec_proofofmainresult}, we prove the main theorems and in Section \ref{sec_subcorollaryproof}, we prove the corollaries. Finally, Appendix \ref{appendix_proof54} is devoted to the proof of the local laws and Appendix \ref{sec_appedixb} proves other auxiliary lemmas. \vspace*{0.1in}

\noindent {\bf Conventions and notations.} Throughout the paper, we always use $n$ as the fundamental large parameter, and write $p \equiv p(n)$. For an $n \times n$ symmetric matrix $H$, its empirical spectral distribution (ESD) is defined as $\mu_H:=\frac{1}{n} \sum_{i=1}^n \delta_{\lambda_i(H)}, $ where $\delta$ is the Dirac's delta function and $\{\lambda_i(H)\}$ are the eigenvalues of $H.$ For any probability measure $\nu$ defined on $\mathbb{R},$ its Stieltjes transform is defined as 
\begin{equation*}
m_{\nu}(z)=\int \frac{1}{x-z} \mathrm{d} \nu (x), 
\end{equation*}
where $z \in \mathbb{C}_+:=\{E+\mathrm{i} \eta: E \in \mathbb{R}, \eta>0\}.$
\vspace{5pt}

\noindent {\bf Acknowledgments.} The authors are supported by NSF-DMS  2113489 and a grant from UC Davis COVID-19
Research Accelerator Funding Track. 
\section{The model and asymptotic laws}\label{sec_modelandlocalaw}


We consider the sample covariance matrix $Q$ in (\ref{eq_samplecovaraincematrix}) and its eigenvalues $\lambda_1 \geq \lambda_2 \geq \cdots \geq \lambda_n>0=\lambda_{n+1}=\cdots=\lambda_p. $ Throughout the paper, to avoid repetition, we impose the following assumptions on $\Sigma$ and $X.$ 
\begin{assum}\label{assum:XSigma} We assume that the following conditions are satisfied:
\begin{enumerate}
\item[](1). For the $p \times n$ matrix $X=(x_{ij}),  1 \leq i \leq p, 1 \leq j \leq n,$ we assume that they are i.i.d. real random variables satisfying
\begin{equation}\label{eq:normalization}
	\mathbb{E} x_{ij}=0, \quad \mathbb{E} x_{i j}^{2}=\frac{1}{\sqrt{pn}}.
\end{equation}  
Moreover, for all $k \in \mathbb{N},$  we assume that there exists some constant $C_{k}$ such that
\begin{equation}\label{eq_momentassumption}
\mathbb{E}\left|(n p)^{1 / 4} x_{i j}\right|^{k} \leqslant C_{k} .
\end{equation}
\item[] (2). For the $p \times p$ deterministic population covariance matrix $\Sigma,$ we assume it is diagonal such that $\Sigma=\operatorname{diag}\left( \sigma_1, \sigma_2, \cdots, \sigma_p \right).$ Moreover, for some small constant $0<\tau<1,$
\begin{equation*}
\tau \leq \sigma_p \leq \sigma_{p-1} \leq \cdots \leq \sigma_1 \leq \tau^{-1}. 
\end{equation*}
\item[] (3). For the dimensionality, we assume that (\ref{eq_ratioassumption}) holds. 
\end{enumerate}
\end{assum}

\begin{rem}
We provide some remarks on Assumption \ref{assum:XSigma}. First, in (\ref{eq:normalization}), inspired by \cite{LL}, we choose the scaling $(pn)^{-1/2}$ to deal with the regime (\ref{eq_ratioassumption}) that $p$ is much larger than $n.$ From the viewpoint of multivariate statistics, it is different from the $n^{-1}$ scaling used in the low dimensional regime that $\alpha<1$ \cite{anderson2009introduction} and the mean field regime $\alpha=1$ \cite{yao2015large}. Second, motivated by the statistical applications, we mainly focus on the regime (\ref{eq_ratioassumption}) when $\alpha>1$. However, our techniques and arguments can also be applied to the setting $0<\alpha \leq 1.$ We will make this clear in our discussions from time to time. Third,  the moment assumption (\ref{eq_momentassumption}) holds for all $k \in \mathbb{N}$ can be easily generalized using the comparison argument as in \cite{erdHos2017dynamical}. Since this is not the focus of the paper, we will not pursue these generalizations. Finally, the diagonal assumption of $\Sigma,$ on the one hand, can improve the interpretability of our results in the context of statistical inference. On the other hand, it can simplify our technical arguments. Such an assumption can be removed with additional technical efforts; see Remark \ref{rem_mainresultone} for more details.     
\end{rem}

In what follows, we first summarize the results on the asymptotic law of the ESD of $Q$. Since its nonzero eigenvalues are identical to those of $\mathcal{Q}=X^* \Sigma X $, it is sufficient to focus on the asymptotic deterministic equivalent of ESD of $\mathcal Q,$ denoted as $\varrho,$ which can be best formulated by its Stieltjes transform. Denote 
\begin{equation}\label{eq_aspecratiodefinition}
\phi \equiv \phi_n:=\frac{p}{n},
\end{equation}
and $\pi$ be the ESD of $\Sigma,$ i.e., $\pi:=p^{-1} \sum_{i=1}^p  \delta_{\sigma_i}.$ 
\begin{lem}\label{lem_asymptoticlaw}
	Let $\pi$ be a compactly supported probability measure on $\mathbb{R}$, and let $\phi>0$. 
	Then for each $z\in\mathbb{C}_+,$ there exists a unique $m \equiv m(z) \in \mathbb{C}_+$ satisfying 
		\begin{equation}\label{eq:lsd}
			\frac{1}{m}=-z+\int\frac{\phi}{\phi^{1/2}x^{-1}+m}\pi(\dd x).
	\end{equation}
	Moreover, $m(z)$ is the Stieltjes transform of a probability measure $\varrho$ with support in $[0,\infty)$.
\end{lem}
\begin{proof} The proof follows from arguments similar to the well-known results as in \cite[Lemma 2.2]{AILL} or \cite[Section 5]{silverstein1995empirical}. The only difference is that we use the new scaling (\ref{eq:normalization}). We omit further details. 

\end{proof}

\begin{rem}\label{rem_selfconsistentequation} {\normalfont
First, we point out that the deterministic quantities $m(z)$ and its associated $\varrho$ depend implicitly on $n.$ Moreover, even though the current paper focuses on the regime (\ref{eq_ratioassumption}), 
the results of Lemma \ref{lem_asymptoticlaw} hold for all $0<\alpha<\infty$ in (\ref{eq_ratioassumption}). Second, the support of $\varrho$ generally depends on $n$ and even grows with $n.$ To see this, we consider the example when $\Sigma=I.$ In this setting, we find that (\ref{eq:lsd}) can be rewritten as 
$$
		m+\frac{1}{z+z \phi^{-1 / 2} m-\left(\phi^{1 / 2}-\phi^{-1 / 2}\right)}=0. 
		$$
		Solving the above equation with the branch cut that $m(z)$ is holomorphic in the upper half-plane and satisfies $m(z) \rightarrow 0$ as $z \rightarrow \infty$, we find that the unique solution of the above equation satisfying $\operatorname{Im} m(z)>0$ for $\operatorname{Im} z>0$ is 
			$$
	m(z)=\frac{\phi^{1 / 2}-\phi^{-1 / 2}-z+\mathrm{i} \sqrt{\left(z-\gamma_{-}\right)\left(\gamma_{+}-z\right)}}{2 \phi^{-1 / 2} z}.
	$$
	Using the inversion formula \cite{bai2010spectral}, we find that under (\ref{eq_ratioassumption}) 
		\begin{equation}\label{eq_explicitdensityfunction}
	\varrho(\mathrm{d} x):=\frac{\sqrt{\phi}}{2 \pi} \frac{\sqrt{\left[\left(x-\gamma_{-}\right)\left(\gamma_{+}-x\right)\right]_{+}}}{x} \mathrm{~d} x,
	\end{equation}
	where $\gamma_{\pm}:=\phi^{1/2}+\phi^{-1/2} \pm 2.$ These coincide with the results in equations (2.4)-(2.7) of \cite{LL}. 
	}
\end{rem}

According to Lemma \ref{lem_asymptoticlaw}, we find that $m \equiv m(z)$ can also be characterized as the unique solution of the equation
\begin{equation}\label{eq_otherformoflsd}
z=f(m), \quad \operatorname{Im} m>0,
\end{equation}
where we define
 \begin{equation}\label{eq:discrete}
		f(x):=-\frac{1}{x}+\frac{1}{n} \sum_{i=1}^p \frac{1}{x+s_i^{-1}}=-\frac{1}{x}+\frac{\phi^{1/2}}{p} \sum_{i=1}^p \frac{\sigma_i}{1+\phi^{-1/2}\sigma_ix},
	\end{equation}
	where $s_i=\phi^{-1/2}\sigma_i$. In what follows, we will show that the properties of $\varrho$ can be understood via the analysis of $f$. As in \cite{AILL}, we see that it is convenient to extend the domain of $f$ to the real projective line $\overline{\mathbb{R}}=\mathbb{R} \cup\{\infty\}$. Clearly, $f$ is smooth on the $p+1$ open intervals of $\overline{\mathbb{R}}$ defined through
$$
I_1:=\left(-s_1^{-1}, 0\right), \quad I_i:=\left(-s_i^{-1},-s_{i-1}^{-1}\right) \quad(i=2, \ldots, p), \quad I_0:=\overline{\mathbb{R}} \backslash \bigcup_{i=1}^p \bar{I}_i .
$$
\begin{lem}\label{lem_proofgloballawresult} Suppose Assumption \ref{assum:XSigma} holds, when $n$ is sufficiently large, for $f(x)$ defined in (\ref{eq:discrete}), there will be two critical points, denoted as $x_1$ and $x_{2},$ where $x_1 \in I_1$ and $x_{2} \in I_0.$ Moreover, define $\gamma_+:=f(x_1)$ and $\gamma_-:=f(x_2).$ We have that $\gamma_+ \geq \gamma_->0$ and 
\begin{equation*}
\operatorname{supp} \varrho \cap(0, \infty)=[\gamma_-, \gamma_+]. 
\end{equation*}   
Finally, we have that 
\begin{equation*}
\gamma_+-\gamma_-=\rO(1), \ \operatorname{and} \ \gamma_-, \gamma_+ \asymp \phi^{1/2}. 
\end{equation*}
\end{lem}
\begin{proof}
See Appendix \ref{appendix_sectionlemma25}. 
\end{proof}

\begin{rem}\label{rem_lemma25} {\normalfont
We provide two remarks regarding (\ref{eq_ratioassumption}). First, analogous results of Lemma \ref{lem_proofgloballawresult} for $\alpha=1$ in (\ref{eq_ratioassumption}) have been proved in Lemmas 2.4-2.6 of \cite{AILL}. Our results in Lemma \ref{lem_proofgloballawresult} demonstrate a structural difference on $\varrho$ between the setting $\alpha=1$ and $\alpha>1.$ Especially, when $\alpha=1,$ the function $f$ may have an even number of more than two critical points so that the support of $f$ may have several components as a union for several closed intervals; see Figure \ref{bulkillustration} for an illustration.  Second, as discussed in Remark \ref{rem_selfconsistentequation}, our proof of Lemma \ref{lem_proofgloballawresult} can be easily generalized to study the setting $\alpha<1.$ In this case, the results will be similar to those of $\alpha=1$ in \cite{AILL}. Especially, by introducing the multiset $\mathcal{C} \subset \overline{\mathbb{R}}$ of critical points of $f$, using the conventions that a nondegenerate critical point is counted once and a degenerate critical point twice, we can show that $\left|\mathcal{C} \cap I_0\right|=\left|\mathcal{C} \cap I_1\right|=1$ and $\left|\mathcal{C} \cap I_i\right| \in\{0,2\}$ for $i=2, \ldots, q.$ Consequently, we see that $|\mathcal{C}|=2 q$ is even. We denote by $x_1 \geqslant x_2 \geqslant \cdots \geqslant x_{2 q-1}$ the $2 q-1$ critical points in $I_1 \cup \cdots \cup I_p$, and by $x_{2 q}$ the unique critical point in $I_0$. For $k=1, \ldots, 2 q$ we define the critical values $a_k:=f\left(x_k\right)$. We can show that $a_1 \geqslant a_2 \geqslant \cdots \geqslant a_{2 q} \geq 0$ and $a_k=\rO(\phi^{1/2}+\phi^{-1/2}).$ Consequently, this leads to  $
	\operatorname{supp} \varrho \cap(0, \infty)=\left(\bigcup_{k=1}^q\left[a_{2 k}, a_{2 k-1}\right]\right) \cap(0, \infty) . $ For more details, we refer the readers to Appendix \ref{appendix_sectionlemma25}. }
\end{rem}

\begin{figure}[!ht]
\centering
\subfigure[$\phi=0.6$]{\includegraphics[width=7cm,height=5cm]{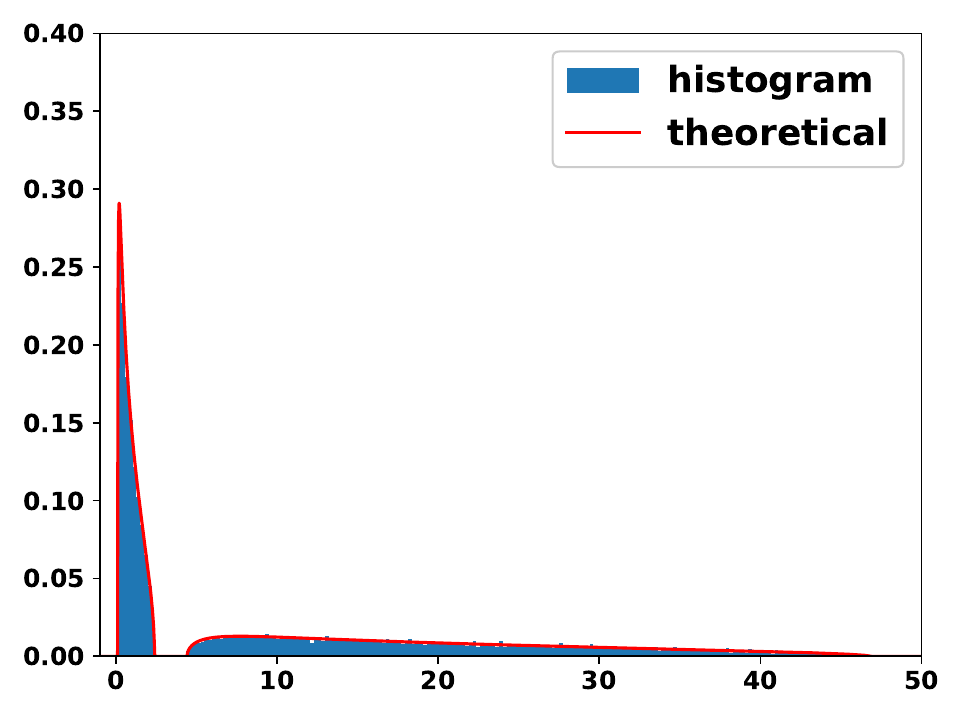}}
\subfigure[$\phi=100$]{\includegraphics[width=7cm,height=5cm]{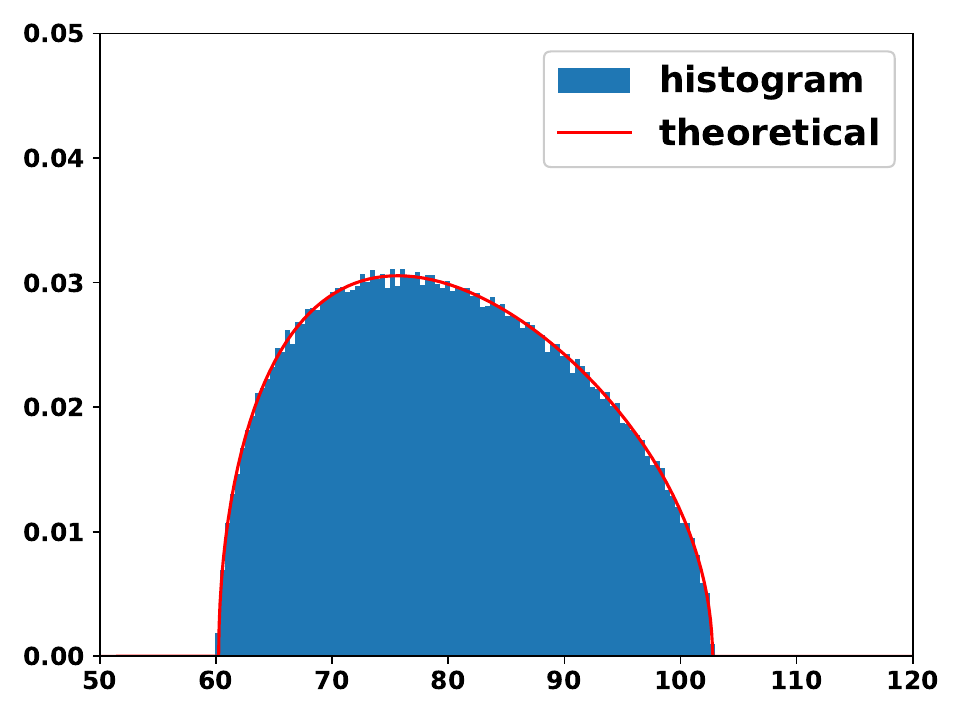}}
\caption{The densities and histograms for $\pi=0.5\delta_1+0.5 \delta_{15}$ for $\phi=0.6$ and $\phi=100.$ The densities are computed using the Stieltjes transform via (\ref{eq:lsd}) with the inversion formula of the Stieltjes transform that $\rho(x)=\lim_{\eta \downarrow 0} \operatorname{Im} m(x+\mathrm{i} \eta)$ and the histograms are obtained using (\ref{eq_samplecovaraincematrix}) with $x_{ij} \sim \mathcal{N}(0, (pn)^{-1/2}), n=800$ via $500$ Monte Carlo simulations. We can see that when $\phi=0.6$ is smaller, we can have two separate bulk components in the support while $\phi=100$ is larger we only have one bulk component. }
\label{bulkillustration}
\end{figure}

\begin{rem}
{\normalfont
Two remarks are in order. First, according to Lemma \ref{lem_proofgloballawresult}, the edges of the support of $\varrho$ can be determined if $\Sigma$ is known. For example, when the ESD of $\Sigma$ follows $\pi=\frac{1}{2}\delta_a+\frac{1}{2}\delta_b$ for some positive constants $a,b>0,$ it is easy to see that  $$\gamma_\pm=\frac{(a+b)\phi^{1/2}}{2} \pm \sqrt{2(a^2+b^2)}.$$ Second, as discussed in Section 3.2 of \cite{bloemendal2016principal}, after properly being scaled and centered, the asymptotic law obtained from Lemma \ref{lem_asymptoticlaw} has a natural connection with the Wigner's semicircle law. In fact, using Lemma \ref{lem:mphizabs} below and a straightforward expansion, we see from (\ref{eq:lsd}) that   
\begin{equation*}
	\begin{aligned}
		\frac{1}{m}=-z+\phi^{1/2}\int x \pi (\dd x)-\int x^2 \pi(\dd x)m+\mathrm{O}(\phi^{-1/2}).
	\end{aligned}
\end{equation*}
Especially, when $\pi(x)=\delta_1,$ we see that 
\begin{equation*}
\frac{1}{m}+m+z-(\phi^{1/2}+\rO(\phi^{-1/2}))=0.
\end{equation*}
This is nothing but Wigner's semicircle law centered at $\phi^{1/2}+\rO(\phi^{-1/2})$ which recovers the discussions of equations (3.16) and (3.17) of \cite{bloemendal2016principal} at the cost of $\rO(\phi^{-1/2}).$

	}
\end{rem}

\section{Main results}\label{sec_mainresultandlocalaw}
In this section, we prove the main results and sketch the proof strategies. We first prepare some notations. Fix some small constant $0<\tau<1,$ we denote  
\begin{equation}\label{eq_realevalue}
\mathbf{R} \equiv \mathbf{R}(\tau):=\left\{ E \in \mathbb{R}: \ |E-\mathfrak{m}_1(\pi) \phi^{1/2}| \leq \tau^{-1} \right\}, \ \text{where} \  \mathfrak{m}_1(\pi)=\int x \pi (\dd x). 
\end{equation}
Let $\mathcal{C}_c^2(\mathbb{R})$ be the function space on some compact set $\mathbf{C} \subset \mathbb{R}$  of continuous functions that have continuous first two derivatives. For some fixed integer $\mathsf{K} \in \mathbb{N}$ and a sequence of test functions $g_1(x), g_2(x), \cdots, g_\mathsf{K}(x) \in \mathcal{C}^2_c (\mathbb{R})$ and some constants $E \in \mathbf{R}, \eta_0>0$ we define $f_i(x) \equiv f_i(x; E,\eta_0), 1 \leq i \leq \mathsf{K},$ 
\begin{equation}\label{eq_gxdefinition}
f_i(x) \equiv f_i(x; E,\eta_0):=g_i\left( \frac{x-E}{\eta_0} \right), \ 1 \leq i \leq \mathsf{K}.
\end{equation}
Recall from Lemma \ref{lem_asymptoticlaw} that $\varrho$ is the asymptotic density of the limiting ESD. For $1 \leq i \leq \mathsf{K},$ denote 
\begin{equation}\label{eq_zf}
Z_{\eta_0, E}(f_i)=\sum_{j=1}^n f_i(\lambda_j)-n \int_{\mathbb{R}} f_i(x) \mathrm{d} \varrho(x) . 
\end{equation} 

To ease our statements, following \cite[Definition 2.2]{bao2022statistical}, we use the following definition. 

\begin{defn}
Two sequences of random vectors $\bm{x}_n, \bm{y}_n  \in \mathbb{R}^{\mathsf{K}}, \ n \geq 1,$ are asymptotically equal in distribution, denoted by $\bm{x}_n \simeq \bm{y}_n,$ if they are tight (i.e., for any $\epsilon>0,$ there exists a $D>0$ such that $\sup_n \mathbb{P}(\| \bm{x}_n\| \geq D) \leq \epsilon$) and satisfy
\begin{equation*}
\lim_{n \rightarrow \infty} \left( \mathbb{E} h(\bm{x}_n)-\mathbb{E} h(\bm{y}_n) \right)=0,
\end{equation*}
for any bounded continuous function $h: \mathbb{R}^{\mathsf{K}} \rightarrow \mathbb{R}.$  
\end{defn}

\subsection{Global CLTs under (\ref{eq_ratioassumption})}\label{sec_CLTs}

Let $\kappa_k$ be the $k$-th cumulant of $(pn)^{1/4} x_{ij}$ in (\ref{eq:normalization}), given by $\kappa_k:=(-\mathrm{i})^k \frac{\mathrm{d}}{\mathrm{d} t} \log \mathbb{E} e^{\mathrm{i} t (pn)^{1/4}x_{ij}}|_{t=0}.$ Moreover, for $m(z)$ defined in (\ref{eq:lsd}), we denote
 	{\begin{equation}\label{eq_b(zdefinition)}
			b(z) :=b_1(z)+b_2(z),
		\end{equation}
		where $b_1(z)$ and $b_2(z)$ are defined as 
\begin{equation}\label{eq_b1b2definition}				
b_1(z):=\frac{{m}^{\prime \prime}(z)}{2 {m}^{\prime}(z)}-\frac{{m}^{\prime}(z)}{{m(z)}}, \ b_2(z)=\kappa_4 \left(\frac{m^2(z)m''(z)}{2{m'(z)}^2}-m(z)\right).
\end{equation}		
	Moreover, for any real number $x \in \mathbb{R},$ we use the convention that 
		$$b_1^{\pm}(x)=\lim_{\eta\downarrow 0}b_1(x \pm \mathrm{i}\eta), \ b_2^{\pm}(x)=\lim_{\eta\downarrow 0}b_2(x \pm \mathrm{i}\eta), \ \ m^{\pm}(x)=\lim_{\eta\downarrow 0}m(x \pm \mathrm{i}\eta).$$  }	
Armed with the above conventions, for $\mathfrak{a},\mathfrak{b}\in \{+,-\},$ we further define 
 $$\widehat{\alpha}_{\mathfrak{a} \mathfrak{b}}\left(x_1, x_2\right):=\kappa_4 \phi \frac{\partial^2}{\partial x_1 \partial x_2}\left(\frac{1}{p} \sum_{i=1}^p \frac{1}{\left(1+\phi^{-1/2}{m}^{\mathfrak{a}}(x_1) \sigma_i\right)\left(1+\phi^{-1/2}{m}^{\mathfrak{b}}(x_2) \sigma_i\right)}\right),$$
	$$
	\widehat{\beta}_{\mathfrak{a} \mathfrak{b}}\left(x_1, x_2\right):=2 \left(\frac{(m^{\mathfrak{a}})'(x_1)(m^{\mathfrak{b}})'(x_2)}{(m^{\mathfrak{a}}(x_1)-m^{\mathfrak{b}}(x_2))^2}-\frac{1}{(x_1-x_2)^2} \right).
	$$
We will consistently use the following conventions 	
\begin{equation}\label{eq_moreconvention}
\alpha(x_1,x_2) \equiv \alpha:=\widehat{\alpha}_{++}+\widehat{\alpha}_{--}-\widehat{\alpha}_{+-}-\widehat{\alpha}_{-+}, \ \ \beta(x_1, x_2) \equiv  \beta:=\widehat{\beta}_{++}+\widehat{\beta}_{--}-\widehat{\beta}_{+-}-\widehat{\beta}_{-+}. 
\end{equation}
The main results are the following two theorems. The first theorem establishes the CLT for the LSS on the global scale when $\eta_0 \asymp 1.$
\begin{thm}[Global CLT]\label{thm_mainone}
	 Suppose Assumption \ref{assum:XSigma} holds. For $E \in \mathbf{R}$ and $\eta_0 \asymp 1,$ the random vector $(Z_{\eta_0,E}(f_i))_{1 \leq i \leq \mathsf{K}}$ in (\ref{eq_zf}) follows that $(Z_{\eta_0,E}(f_i))_{1 \leq i \leq \mathsf{K}} \simeq \left(\mathscr{G}_i\right)_{1 \leq i \leq \mathsf{K}},$
where $\left(\mathscr{G}_i\right)_{1 \leq i \leq \mathsf{K}}$ is a Gaussian random vector with mean and covariance functions defined as
	\begin{equation}\label{eq_meanfunction}
		\mathbb{E}\left(\mathscr{G}_i \right)  =\frac{1}{2\pi\mathrm{i}}\left(\int_{\mathbb{R}} f_i(x)b^{\mathfrak{+}}(x) \mathrm{d} x-\int_{\mathbb{R}} f_i(x)b^{\mathfrak{-}}(x) \mathrm{d} x\right),
	\end{equation}
and
	\begin{align}\label{eq_varianceglobal}
		\operatorname{Cov}\left(\mathscr{G}_i, \mathscr{G}_j\right)  =-\frac{1}{4\pi^2} \iint_{\mathbb{R}^2} f_i\left(x_1\right) f_j\left(x_2\right) \alpha \left(x_1, x_2\right) \mathrm{d} x_1 \mathrm{d} x_2 -\frac{1}{4\pi^2}  \iint_{\mathbb{R}^2} {f_i\left(x_1\right) f_j\left(x_2\right)} \beta \left(x_1, x_2\right) \mathrm{d} x_1 \mathrm{d} x_2,
	\end{align}
	where $b^{\pm}(x), \alpha(x_1,x_2)$ and $\beta(x_1, x_2)$ are defined between (\ref{eq_b(zdefinition)}) and (\ref{eq_moreconvention}) and we assume their limits exist.  
\end{thm}
\begin{rem}\label{rem_contourform}
We point out that analogous results have been stated in the comparable regime (i.e., $\alpha=1$ in (\ref{eq_ratioassumption})) in \cite{ZBY} under stronger assumptions on the test functions. Even though the statements are different, we notice that the kernels of the covariance functions in (\ref{eq_varianceglobal}) are essentially the same with the one in \cite{ZBY} using a contour integral representation which is an intermediate step in our calculation; see (\ref{eq_maincontourrough}) below.  Compared to the contour integral form, (\ref{eq_varianceglobal}) eases the numerical approximations.    
\end{rem}

\begin{rem}\label{rem_mainresultone}
Several remarks are in order. First, we notice that the construction in (\ref{eq_gxdefinition}) involves using a function with compact support so that the test functions needed to be well constructed. For example, we cannot directly use $g(x)=x.$ However, this issue can be easily addressed using the following construction. For some $h(x) \in \mathcal{C}^2(\mathbb{R}),$ we define 
\begin{equation}\label{eq_gixform}
g(x)=h(x) \mathcal{K}(x),
\end{equation}
where $\mathcal{K}(x)$ is a mollifier that for some fixed small constant $a>0$ and large constant $b$
\begin{equation} \label{eq_millifier}
		\mathcal{K}(x)\equiv \mathcal{K}_{a,b}(x)  := \begin{cases}
			0 &  |x-b| \geq a \\
			1 &  |x| \leq b \\
			\exp\left(\frac{1}{a^2} - \frac{1}{a^2-(x +b)^2}\right) &  - (b + a) < x < -b \\
			\exp\left(\frac{1}{a^2}  - \frac{1}{a^2-(x -b)^2}\right) &  b < x < b + a
		\end{cases}.
	\end{equation} 
It is then straightforward to check that the above construction will generate a $\mathcal{C}_c^2$ function which preserves the properties of $h(x)$, e.g., when dealing with Theorem \ref{thm_mainone}, we can let $b$ be large enough such that the effects of test functions $g(x)$ on eigenvalues are identical with $h(x)$. Second, we believe that the condition $\mathcal{C}_c^2$ can be weakened to the class of  $\mathcal{C}^{1, a, b}\left(\mathbb{R}_{+}\right):=\left\{f \in \mathcal{C}_c^1\left(\mathbb{R}_{+}\right): f^{\prime}\right.$ is $a$ H\"{o}lder continuous uniformly in $x$, and $|f(x)|+\left|f^{\prime}(x)\right| \leqslant C(1+|x|)^{-(1+b)}$ for some constant $\left.C>0\right\}$ following \cite{He2017,Yang2020}. We will pursue this direction in the future works. Third, in the current paper, for simplicity, we assume that  $\Sigma$ is diagonal. However, such an assumption can be moved with additional technical efforts following \cite{Yang2020}. More specifically, we can assume that $\Sigma^{1/2}$ admits $\Sigma=O^{*}\Lambda O$ where $O$ is orthogonal and $\Lambda$ is diagonal. Consequently, we will need to work with  $\Lambda^{1 / 2} O X X^{*} O^{*} \Lambda^{1/2}.$ Since this is out of the scope of the current paper, we will consider this generalization in the future works. 

\end{rem}

The above theorem establishes the joint distribution for LSS indexed by different test functions on the global scale. In general, the mean and covariance functions depend on the test functions, $\phi,$ $\Sigma,$ the 4th cumulant, $E$ and $\eta_0.$ For some properly chosen test functions and $\Sigma,$ the mean and covariances can be largely simplified and some explicit formulas could be obtained. In the following corollary, based on the above results, we can obtain the  asymptotic normality for some individual specific test functions when $\Sigma=I$ and $\eta_0=1.$ These will be used in Section \ref{sec_statapplication} regarding the statistical applications. 

\begin{cor}\label{cor_globalexamples} Suppose Assumption \ref{assum:XSigma} holds with $\Sigma=I.$ 
\begin{enumerate}
\item[(1).] For $E \in \mathbf{R}$ and $\eta_0=1,$ we have that 
\begin{equation}
\frac{Z_{1,E}(f)-\mathsf{M}(f)}{\sqrt{\mathsf{V}(f)}} \simeq \mathcal{N}(0,1),
\end{equation}
where $\mathsf{M}(f)$ and $\mathsf{V}(f)$ are defined as
\begin{align}\label{eq_meanfunction}
\mathsf{M}(f)& := \lim_{r \downarrow 1} \frac{-1}{2\pi\mathrm{i}}\oint_{|\xi|=1}f \left(\phi^{1/2}+\phi^{-1/2}+\xi+\frac{1}{\xi} \right)\left(\frac{1}{\xi}-\frac{1}{2}\frac{1}{\xi+\frac{1}{r}}-\frac{1}{2}\frac{1}{\xi-\frac{1}{r}}\right)\mathrm{d} \xi  \nonumber \\
& +\kappa_4 \frac{-1}{2\pi\mathrm{i}}\oint_{|\xi|=1}f(\phi^{1/2}+\phi^{-1/2}+\xi+\frac{1}{\xi})\left(-\frac{1}{\xi^3}\right)\mathrm{d} \xi,
\end{align}
and 
\begin{align}\label{eq_variancefunction}
\mathsf{V}(f)& :=\lim_{\substack{r_2>r_1 \\ r_1, r_2 \downarrow 1}}-\frac{1}{2\pi^2}\oint_{|\xi_1|=1}\oint_{|\xi_2|=1} \frac{f(\phi^{1/2}+\phi^{-1/2}+\xi_1+\xi_1^{-1})f(\phi^{1/2}+\phi^{-1/2}+\xi_2+\xi_2^{-1})}{(\xi_1-r_2/r_1\xi_2)^2}\mathrm{d}\xi_1\mathrm{d}\xi_2 \nonumber \\
& -\frac{\kappa_4}{4\pi^2}\left( \oint_{|\xi|=1}f(\phi^{1/2}+\phi^{-1/2}+\xi+\xi^{-1})\frac{1}{\xi^2}\mathrm{d}\xi \right)^2. 
\end{align}
\item[(2).] The formulas in (\ref{eq_meanfunction}) and (\ref{eq_variancefunction}) can be further simplified for some commonly used test functions and  properly parametrized $E.$ In particular, for $E=\phi^{1/2}+\phi^{-1/2}-c$ for some constant $c>0,$ when $h_1(x)=x$  in (\ref{eq_gixform}) and the associated $f_1(x)=(x-E) \mathcal{K}(x-E)$, we have that 
\begin{equation*}
\mathsf{M}(f_1)=0, \  \mathsf{V}(f_1)=2+\kappa_4. 
\end{equation*}
Moreover, for $E=\phi^{1/2}+\phi^{-1/2}-c$ for some constant $c>0,$ when $h_2(x)=x^2$ in (\ref{eq_gixform}) and the associated $f_2(x)=(x-E)^2 \mathcal{K}(x-E)$, we have that 
\begin{equation*}
\mathsf{M}(f_2)=1+\kappa_4, \  \mathsf{V}(f_2)=4+4c^2(\kappa_4+2). 
\end{equation*}
Finally, for $E=\phi^{1/2}+\phi^{-1/2}-(t+t^{-1})$ for some constant $t>1,$ when $h_3(x)=\log x$ in (\ref{eq_gixform}) and the associated $f_3(x)=\log(x-E) \mathcal{K}(x-E)$, we have that 
\begin{equation*}
\mathsf{M}(f_3)=\frac{1}{2} \log (1-t^{-2})-\frac{\kappa_4}{2 t^2}, \  \mathsf{V}(f_3)=2(\log t-\log(t-t^{-1}))+\frac{\kappa_4}{t^2}. 
\end{equation*}
\end{enumerate}
\end{cor}

\begin{rem}
In addition to the mean and covariance functions, in general, the second term of the right-hand side of (\ref{eq_zf}) needs to be computed numerically. However when $\Sigma=I,$ since the density function of MP can be made explicit as in (\ref{eq_explicitdensityfunction}), this term can be further simplified for some test functions. Especially, for $f_1(x)$, $f_2(x)$ and $E=\phi^{1/2}+\phi^{-1/2}-c$ considered in (2) of Corollary \ref{cor_globalexamples}, we see from (\ref{eq_explicitdensityfunction}) that    
\begin{equation}\label{eq_generalmeanform}
\int_{\mathbb{R}}f_1(x) \mathrm{d} \varrho(x) = c-\phi^{-1/2}, \ \int_{\mathbb{R}}f_2(x) \mathrm{d} \varrho(x) =1+c^2-2c \phi^{-1/2}+\phi^{-1}.
\end{equation}
\end{rem}

\subsection{Local CLTs under (\ref{eq_ratioassumption})}

The second theorem establishes the CLT for the LSS on the local scale when $\eta_0 \ll 1. $ Denote $$ \kappa \equiv \kappa(E):=\min \left\{ |E-\gamma_-|, |E-\gamma_+| \right\}.$$

\begin{thm}[Local CLT]\label{thm_maintwo}
 Suppose Assumption \ref{assum:XSigma} holds. Fix $E \in \mathbf{R}.$ For some small constants $\tau_1, \tau_2>0$, $ \eta_0 \leq n^{-\tau_1}$  and  $\eta_0 \sqrt{\kappa+\eta_0} \geq n^{-1+\tau_2},$ the random vector $(Z_{\eta_0,E}(f_i))_{1 \leq i \leq \mathsf{K}}$ in (\ref{eq_zf}) follows that $(Z_{\eta_0,E}(f_i))_{1 \leq i \leq \mathsf{K}} \simeq \left(\mathscr{G}_i\right)_{1 \leq i \leq \mathsf{K}},$
where $\left(\mathscr{G}_i\right)_{1 \leq i \leq \mathsf{K}}$ is a Gaussian random vector with the mean and covariance functions defined as	 
\begin{equation}\label{eq_meanlocal}
\mathbb{E} \mathscr{G}_i=\frac{1}{2\pi\mathrm{i}}\left(\int_{\mathbb{R}} f_i(x)b_1^{\mathfrak{+}}(x) \mathrm{d} x-\int_{\mathbb{R}} f_i(x)b_1^{\mathfrak{-}}(x) \mathrm{d} x\right),
\end{equation}
and 	
	\begin{align}\label{eq_covariancelocal}
		\operatorname{Cov}\left(\mathscr{G}_i, \mathscr{G}_j\right) =
		-\frac{1}{4\pi^2}  \iint_{\mathbb{R}^2 } {f_i\left(x_1\right) f_j\left(x_2\right)}  \beta\left(x_1, x_2\right) \mathrm{d} x_1 \mathrm{d} x_2, 
	\end{align}
where $b_1^{\pm}(x)$ and $\beta(x_1, x_2)$ are defined in (\ref{eq_b1b2definition}) and (\ref{eq_moreconvention}) and we assume their limits exist.
\end{thm}

\begin{rem}
Two remarks are in order. First, in the comparable regime when $p \asymp n,$ similar results have been obtained for a single test function in \cite{Li2021} with additional assumptions on $\Sigma$ that $|1+\phi^{-1/2}m\sigma_i|$ is bounded from below  to ensure the square root behavior of the LSD near the edges; see Assumption 8.3 therein. However, 
these conditions will be automatically satisfied under (\ref{eq_ratioassumption}). Second, the proof \cite{Li2021} relies on a characteristic function approach. It is not clear how easily it can be generalized to study the multiple functions setting. We will utilize a different approach; see Section \ref{sec_examproofsrategy} for more details.    
\end{rem}

The above theorem shows that the LSS converges to some Gaussian processes who mean and covariance functions can be explicitly identified on the local scale. Compared to Theorem \ref{thm_mainone}, we find that when $\eta_0=\mathrm{o}(1),$ the results do not reply on $\kappa_4$ which makes it possible to create some moment free statistics to test (\ref{eq_generaltesting}). We illustrate this significant difference between the local and global CLTs in Figure \ref{fig_generaldiscussionlocalglobal1111}.

\begin{figure}[!ht]
\centering
\subfigure[Global CLTs.]{\label{fig:a}\includegraphics[width=7cm,height=5cm]{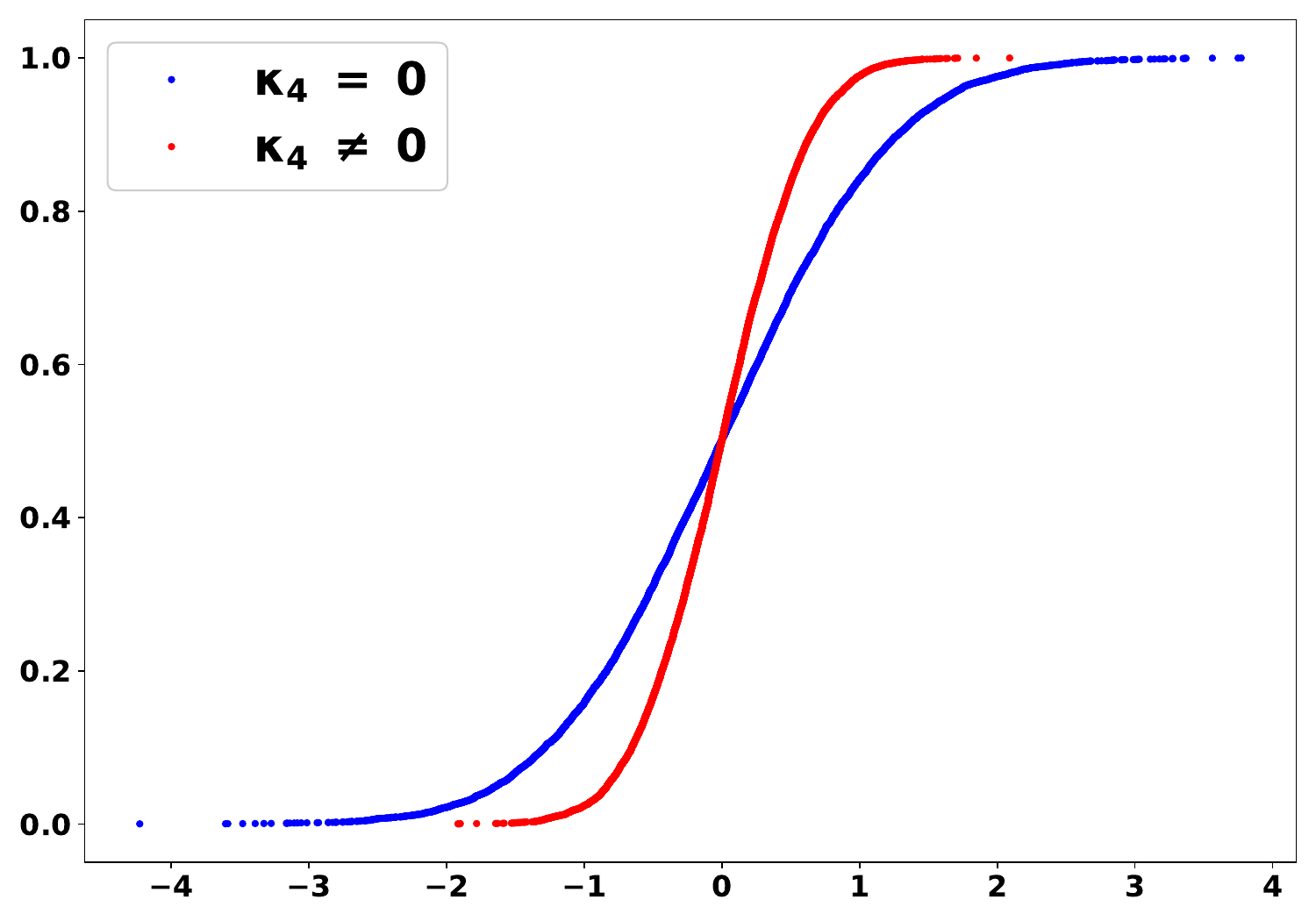}}
\hspace*{0.2cm}
\subfigure[Local CLTs.]{\label{fig:b}\includegraphics[width=7cm,height=5cm]{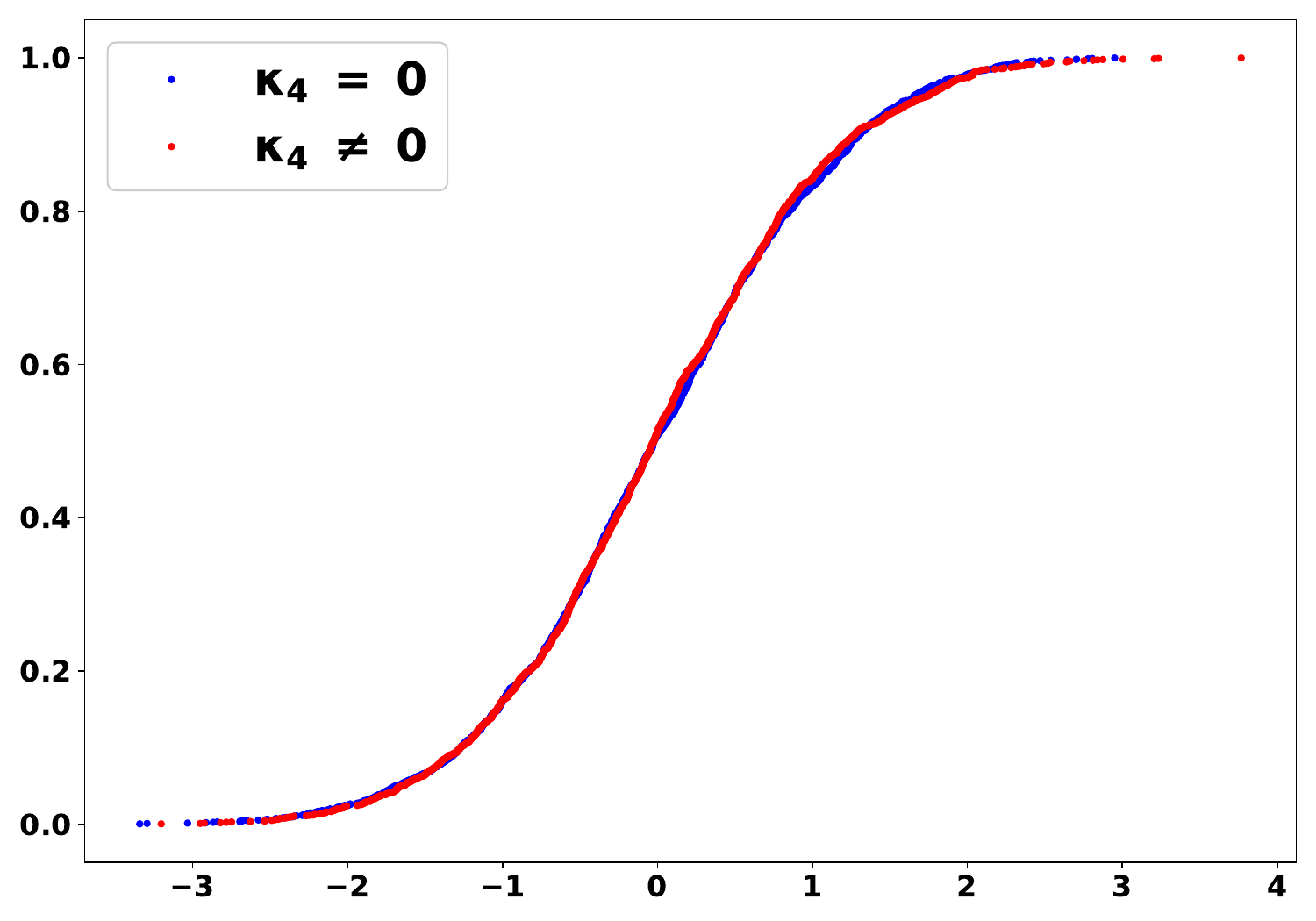}}
\caption{Empirical Cumulative Distribution Functions (ECDFs) for global and local CLTs with different $\kappa_4$. We consider two different distributions for $x_{ij}:$ the Gaussian distribution with $\kappa_4=0$ and the two-point distribution that $\frac{1}{3} \delta_{\sqrt{2}}+\frac{2}{3} \delta_{-1/\sqrt{2}}$ with $\kappa_4=-3/2.$ We can conclude that for the CLTs with $\kappa_4=0$ and $\kappa_4 \neq 0$ are very different on the scale that $\eta_0 \asymp 1$ but identical on the scale that $\eta_0=\mathrm{o}(1).$ Here we choose a single test function that $h(x)=x$ in (\ref{eq_gixform}), $n=400, \phi=100,$ $\eta_0=1$ for global CLTs and $\eta_0=n^{-1/4}$ for local CLTs, and the results are reported based on 1,000 repetitions.
}
\label{fig_generaldiscussionlocalglobal1111}
\end{figure}

Next, we show that when $E$ is properly chosen, the formulas in Theorem \ref{thm_maintwo} can be further simplified. 

\begin{cor}\label{cor_localexample}
Suppose the assumptions of Theorem \ref{thm_maintwo} hold. Recall the function $g_i(x)$ in (\ref{eq_gxdefinition}). Then the followings holds.   
\begin{enumerate}
\item[(1).] If $E$ lies in the bulk of $\varrho$ in the sense that $E \in (\gamma_-+\tau', \gamma_+-\tau')$ for some fixed small constant $\tau'>0,$ then (\ref{eq_meanlocal}) and (\ref{eq_covariancelocal}) can be further simplified as $\mathbb{E} \mathscr{G}_i=0$ and 
$$
 \operatorname{Cov}\left(\mathscr{G}_i, \mathscr{G}_j\right)=\frac{1}{2 \pi^2} \int_{\mathbb{R}} \int_{\mathbb{R}} \frac{\left(g_i\left(x_1\right)-g_i\left(x_2\right)\right)\left(g_j\left(x_1\right)-g_j\left(x_2\right)\right)}{\left(x_1-x_2\right)^2} \mathrm{d} x_1 \mathrm{d} x_2. 
 $$
 \item[(2).] If $E=\gamma_+,$ then (\ref{eq_meanlocal}) and (\ref{eq_covariancelocal}) can be further simplified as $\mathbb{E} \mathscr{G}_i=g_i(0)/4$ if $g_i(0)$ exists and
 $$ \operatorname{Cov}\left(\mathscr{G}_i, \mathscr{G}_j\right)=\frac{1}{4  \pi^2} \int_{\mathbb{R}} \int_{\mathbb{R}}\frac{\left(g_i\left(-x_1^2\right)-g_i\left(-x_2^2\right)\right)\left(g_j\left(-x_1^2\right)-g_j\left(-x_2^2\right)\right)}{(x_1-x_2)^2} \mathrm{d} x_1 \mathrm{d} x_2.$$
 Similarly, if $E=\gamma_-,$ we have that $\mathbb{E} \mathscr{G}_i=g_i(0)/4$ and  
  $$ \operatorname{Cov}\left(\mathscr{G}_i, \mathscr{G}_j\right)=\frac{1}{4  \pi^2} \int_{\mathbb{R}} \int_{\mathbb{R}}\frac{\left(g_i\left(x_1^2\right)-g_i\left(x_2^2\right)\right)\left(g_j\left(x_1^2\right)-g_j\left(x_2^2\right)\right)}{(x_1-x_2)^2} \mathrm{d} x_1 \mathrm{d} x_2.$$
 \end{enumerate}
\end{cor}

\subsection{Proof strategies and routine}\label{sec_examproofsrategy}
In this section, we give a sketch of the proof strategies.
We first review the results and methods used in the most related works \cite{Li2021,Yang2020,Zheng2012,ZBY}, and then discuss how to generalize their methods and highlight our novelties.

 We emphasize that all of the aforementioned works focus on the regime that $p$ and $n$ are comparably large, i.e., $\alpha=1$ in (\ref{eq_ratioassumption}), where the support of $\varrho$ is bounded. Under this setup, \cite{Zheng2012,ZBY} established the global CLTs when $\eta_0=1.$ On the technical level, \cite{Zheng2012,ZBY} follow and generalize the strategies developed in \cite{Bai2004,bai2010spectral}. However, to our best knowledge, their strategies require the testing functions to be sufficiently smooth and cannot be applied to study the local CLTs. When $\eta_0=\mathrm{o}(1),$ \cite{Li2021} studied the local CLTs by studying the characteristic functions of the LSS for a single test functions, i.e., $\mathsf{K}=1$ in (\ref{eq_gxdefinition}). Even though it is possible  to generalize \cite{Li2021} to general $\mathsf{K},$  we believe it is a nontrivial task considering the complicated form of characteristic functions for random vectors. Finally, in \cite{Yang2020}, the author studies both the global and local CLTs for LSS of the eigenvectors whose mean part is always zero. We point out that the detailed analysis of the LSS for the eigenvalues and eigenvectors is quite different in terms of their tracial representation. For concreteness, while both statistics can be written as $\operatorname{Tr}f(Q)A,$ the eigenvalue LSS corresponds to a full rank matrix $A=I$ and the eigenvector LSS corresponds to a rank-one matrix $A=\mathbf{v}\mathbf{v}^*$ for some deterministic vector $\mathbf{v}.$ 

To prove the main results, we need to generalize the methods used in these works, especially \cite{Li2021, Yang2020}. Our first step is to decompose the LSS in (\ref{eq_zf}) into a deterministic mean bias part and a random mean zero part in the sense that
\begin{equation}\label{eq_decomposition}
Z_{\eta_0, E}(f_i)=\mathbb{E} \operatorname{Tr} f_i\left(Q\right)-n \int_{\mathbb{R}} f_i(x) \mathrm{d} \varrho(x) +\mathcal{Z}_{{\eta_0},E}(f_i):=\mathcal{M}_{\eta_0,E}(f_i)+ \mathcal{Z}_{{\eta_0},E}(f_i).
\end{equation}        
For the random part $\mathcal{Z}_{{\eta_0},E}(f_i),$ similar to \cite{Yang2020}, with the help of Helffer-Sj{\" o}strand formula (c.f. Lemma \ref{lem_HSformula}), we find that it suffices to understand distributions of $\mathcal{Y}(z_i):=\operatorname{Tr} R_1(z_i)-\mathbb{E} \operatorname{Tr} R_1(z_i), \ R_1(z_i)=(Q-z_i)^{-1}$ at multiple points $z_i=E_i+\mathrm{i} \eta_i \in \mathbb{C},  1 \leq i \leq l.$ We now focus our explanation on $l=2.$ According to Wick's theorem (c.f. Lemma \ref{lem_wicktheorem}), it suffices to prove that for any $l_1, l_2 \in \mathbb{N},$ 
\begin{align}\label{eq_wickl2case}
	\mathbb{E}\left[\mathcal{Y}^{l_1}\left(z_1\right) \mathcal{Y}^{l_2}\left( z_2\right)\right] & =l_1\mathbb{E}\left[\mathcal{Y}^{l_1-1}\left(z_1\right) \mathcal{Y}^{l_2-1}\left( z_2\right)\right]\omega(z_1,z_2) \nonumber \\
&	+(l_2-1)\mathbb{E}\left[\mathcal{Y}^{l_1}\left(z_1\right) \mathcal{Y}^{l_2-2}\left( z_2\right)\right]\omega(z_1,z_2)+\mathrm{o}(1), 
	\end{align}
where $\omega(z_1,z_2)$ encodes the variances and covariance (i.e., $\omega(z_1, z_1), \omega(z_2,z_2)$ and $\omega(z_1,z_2)$ are respectively the variances and covariance.) To prove (\ref{eq_wickl2case}), we utilize the cumulant expansions (c.f. Lemma \ref{lem_cumulantexpansion}) .  Roughly speaking, we need to estimate the terms of the following form 
\begin{equation}\label{eq_controlledterms}
\sum_{i,j} \sum_{k=0}^3 \frac{1}{k!} \kappa_{k+1} \mathbb{E} \frac{\partial^k (\mathcal{Y}(z_1)^{l_1} \mathcal{Y}(z_2)^{l_2-1}(R_1(z_2) \Sigma^{1/2} X)_{ij})}{\partial x_{ij}^k},
\end{equation}
plus a negligible  error term. A key ingredient is the local laws for $Q$ under (\ref{eq_ratioassumption}) (c.f. Theorem \ref{thm_locallaw}) which we will prove in Appendix \ref{appendix_proof54}. With the local laws, the detailed discussion for estimating (\ref{eq_controlledterms}) will be provided in the proof of Lemma \ref{lem:resolventwick}. We point out that in the actual result Theorem \ref{them_CLTresolvent}, in order to unify the regimes $\eta_0 \asymp 1$ and $\eta_0=\mathrm{o}(1),$ we  need to work with $ \eta_i \mathcal{Y}(z_i).$

Armed with the distributions of the resolvents, we can proceed to establish the distributions for $\mathcal{Z}_{\eta_0, E}(f_i)$ using the integral representation as in (\ref{eq_mathcalzformformformform}). The arguments are similar to those in \cite{Li2021, Yang2020} while we need address some challenges in order to get our final results. We now elaborate this as follows. In contrast to our statements in Theorems \ref{thm_mainone} and \ref{thm_maintwo},  the results stated in \cite{Li2021} are in a regional integral form which may not be friendly for numerical calculations. In \cite{Yang2020}, the author simplifies the contour integral by decomposing the extended functions into three parts and each part can be well controlled (see Section 6.2 therein). However, such a strategy fails in the eigenvalue LSS setting under the setup $p \gg n$ as in (\ref{eq_ratioassumption}). Instead, we first use the complex Green's theorem (c.f. Lemma \ref{lem_Greenthm}) to translate the regional integral to a contour integral and then refine the formulas by some finer controls as in (\ref{eq_reducedreducedreducedform}). For the mean part $\mathcal{M}_{\eta_0,E}(f_i),$ the arguments are similar. Again using the Helffer-Sj{\" o}strand formula (c.f. (\ref{eq_meanonereduce}) and (\ref{eq_meanonereduceone})), it suffices to estimate $\mathbb{E}(\operatorname{Tr}(R_1)-p m_p)$ which can be calculated in the same way as $\mathcal{Y}(z).$ The analysis is provided in Section \ref{sec_subproofmaintheorem} and the result is provided in (\ref{eq_biaserrb4integral}).

We point out that there are two advantages of the proof strategies. One is that  it unifies the regimes $\eta_0 \asymp 1$ and $\eta_0=\mathrm{o}(1).$ The other is that it demonstrates why the $\kappa_4$ dependence vanishes for local CLTs. More specifically, by using a simple change of variable (\ref{eq_changeofvariable}), the terms involving $\kappa_4$ can be shown to be dependent on $\eta_0$ explicitly (c.f. (\ref{eq_controlcontrolsmall})) so that they will disappear when $\eta_0=\mathrm{o}(1).$ Finally, to facilitate statistical applications, in Corollaries \ref{cor_globalexamples} and \ref{cor_localexample}, we provide more compact formulas for some important examples. 
Technically, for Corollary \ref{cor_globalexamples}, we use the  change of variables from \cite{ZBY} to further calculate (\ref{eq_meanfunction}) and (\ref{eq_varianceglobal}), and we follow \cite{Li2021} to simplify (\ref{eq_meanlocal}) and (\ref{eq_covariancelocal})  depending on whether $E$ is in the bulk or at edges.

\section{Statistical applications}\label{sec_statapplication}
In this section, we consider statistical applications of our results in Theorems \ref{thm_mainone} and \ref{thm_maintwo} and Corollaries \ref{cor_globalexamples} and \ref{cor_localexample}  to hypothesis testings on large covariance matrices. Especially, we want to test whether the population covariance matrix  $\Sigma$ in (\ref{eq_samplecovaraincematrix}) is equal to a given deterministic positive definite matrix $\Sigma_0.$ This is an important problem in multivariate data analysis \cite{anderson2009introduction} and high dimensional statistics \cite{yao2015large}. Without loss of generality, we can set $\Sigma_0=I$ so that the hypothesis testing problem can be formulated as 
\begin{equation}\label{eq_hypothesistesting}
\mathbf{H}_0: \ \Sigma=I. 
\end{equation}

\subsection{Testing statistics and their asymptotic normality}\label{sec_stattheory}

Inspired by the above discussions in Section \ref{sec_introduction} and motivated by our results in Section \ref{sec_CLTs}, we consider two classes of statistics, the \emph{global statistics} and the \emph{local statistics}.  For the global statistics, following the research line of \cite{ZBY}, for some test functions, we utilize (\ref{eq_zf}) at the edge $\gamma_+$ with $\eta_0=1.$ As mentioned earlier, the successful application of the global statistics requires the prior knowledge of the fourth cumulant of $x_{ij}.$ If such information is unavailable, we can use the local statistics with $\eta_0=\mathrm{o}(1).$ 
Recall (\ref{eq_millifier}). For concreteness and computational simplicity, in view of Corollaries \ref{cor_globalexamples} and \ref{cor_localexample}, we propose the following groups of statistics to test $\mathbf{H}_0$ in (\ref{eq_hypothesistesting}): 
\begin{enumerate}
\item  We consider the statistics based on the test function $h_1(x)=x$ in (\ref{eq_gixform}) that  
\begin{equation*}
\mathcal{T}_1^{\mathrm{g}}:=\sum_{j=1}^n h_1(\lambda_j-\phi^{1/2}-\phi^{-1/2}+c), \ \  \mathcal{T}_1^{\mathrm{l}}:=\sum_{j=1}^n h_1 \left( \frac{\lambda_j-\gamma_+}{\eta_0} \right) \mathcal{K} \left( \frac{\lambda_j-\gamma_+}{\eta_0} \right). 
\end{equation*}
\item  We consider the statistics based on the test function $h_2(x)=x^2$ in (\ref{eq_gixform}) that  
\begin{equation*}
\mathcal{T}_2^{\mathrm{g}}:=\sum_{j=1}^n h_2(\lambda_j-\phi^{1/2}-\phi^{-1/2}+c), \  \mathcal{T}_2^{\mathrm{l}}:= \sum_{j=1}^n h_2\left(\frac{\lambda_j-\gamma_+}{\eta_0} \right) \mathcal{K} \left( \frac{\lambda_j-\gamma_+}{\eta_0} \right).
\end{equation*}
\item We consider the modified log-likelihood ratio test based on test function $h_3(x)=(x+c)-\log (x+c)$, $c>0$ is some fixed small constant, in (\ref{eq_gixform}) that 
\begin{equation*}
\mathcal{T}_3^{\mathrm{g}}:=\sum_{j=1}^n h_3(\lambda_j-\phi^{1/2}-\phi^{-1/2}), \  \mathcal{T}_3^{\mathrm{l}}:= \sum_{j=1}^n h_3\left(\frac{\lambda_j-\gamma_+}{\eta_0} \right) \mathcal{K} \left( \frac{\lambda_j-\gamma_+}{\eta_0} \right).
\end{equation*}
\item We consider the modified John's statistic that 
\begin{equation*}
\mathcal{T}_4^{\mathrm{g}}:=n^2 \frac{\mathcal{T}_2^{\mathrm{g}}}{(\mathcal{T}_1^{\mathrm{g}})^2}, \  \mathcal{T}_4^{\mathrm{l}}:=n^2 \frac{\mathcal{T}_2^{\mathrm{l}}}{(\mathcal{T}_1^{\mathrm{l}})^2}.
\end{equation*}
\end{enumerate} 
 
\begin{rem}
Two remarks are in order. First, in the four groups of test functions, except for the third group to make the logarithm function valid, $c$ can be zero in all other statistics. In fact, our statistics are robust again the choices of $c.$ Second, for the local statistics, we propose statistics based on LSS evaluated at the right-most edge, i.e., $E=\gamma_+$. According to Corollary \ref{cor_localexample}, one can also construct statistics for $E$ exactly in the bulk or at the left-most edge. They will have similar performance to the current ones.  
\end{rem}

Then we establish the asymptotics of the proposed statistics under $\mathbf{H}_0$ in (\ref{eq_hypothesistesting}). For notional convenience, in this section, we let $\varrho_0$ be the measure associated with the null hypothesis (\ref{eq_hypothesistesting}) as in Remark \ref{rem_selfconsistentequation}. For $1 \leq k_1, k_2 \leq 3,$
\begin{equation*}
\mathsf{m}_{k_1}^{\mathrm{l}}= \int_{\mathbb{R}} h_{k_1} \left( \frac{x-\gamma_+}{\eta_0} \right) \mathrm{d} \varrho_0(x), \ \ \mathsf{v}_{k_1, k_2}^{\mathrm{l}}=\frac{1}{4 \pi^2} \int_{\mathbb{R}} \int_{\mathbb{R}}  \frac{ \left( h_{k_1}(-x_1^2) \mathcal{K}(-x_1^2)-h_{k_2}(-x_2^2) \mathcal{K}(-x_2^2) \right)^2}{(x_1-x_2)^2} \mathrm{d} x_1 \mathrm{d} x_2.  
\end{equation*}
If $k_1=k_2,$ we simply write $\mathsf{v}_{k_1}^{\mathrm{l}} \equiv \mathsf{v}_{k_1, k_1}^{\mathrm{l}}.$ Moreover, we define 
\begin{equation*}
\mathsf{v}_4^{\mathrm{l}}:=\left(  \frac{-2(\mathsf{m}_2^{\mathrm{l}})}{(\mathsf{m}_1^{\mathrm{l}})^3} ,\frac{1}{(\mathsf{m}_1^{\mathrm{l}})^2} \right)\left(\begin{array}{cc}
	\mathsf{v}_1^{l}& \mathsf{v}_{1,2}^{\mathrm{l}}\\
	\mathsf{v}_{1,2}^{\mathrm{l}} &\mathsf{v}_2^{l}
\end{array}\right) \left(  \frac{-2(\mathsf{m}_2^{\mathrm{l}})}{(\mathsf{m}_1^{\mathrm{l}})^3} ,\frac{1}{(\mathsf{m}_1^{\mathrm{l}})^2} \right)^*.
\end{equation*}

\begin{cor}\label{coro_statisticalapplications}
Suppose the assumptions of Theorems \ref{thm_mainone} and \ref{thm_maintwo} hold. Then when the null hypothesis $\mathbf{H}_0$ in (\ref{eq_hypothesistesting}) holds, we have that:
\begin{enumerate}
\item For the global statistics, we have that 
\begin{equation}\label{eq_globalnullone}
\mathsf{T}_1^{\mathrm{g}}:=\frac{\mathcal{T}_1^{\mathrm{g}}-n(c-\phi^{-1/2})}{\sqrt{2+\kappa_4}} \simeq \mathcal{N}(0,1), \  \mathsf{T}_2^{\mathrm{g}}:=\frac{\mathcal{T}_2^{\mathrm{g}}-n(1+c^2-2c \phi^{-1/2}+\phi^{-1})-1-\kappa_4}{\sqrt{4+4c^2(\kappa_4+2)}}  \simeq \mathcal{N}(0,1),
\end{equation}
\begin{equation}\label{eq_globalnullthree}
\mathsf{T}_4^{\mathrm{g}}:= \frac{\mathcal{T}_4^{\mathrm{g}}-n(1+(c-2\phi^{-1/2})^{-2})-(\kappa_4+1)(c-2\phi^{-1/2})^2}{\sqrt{c^{-4} \left(4+4(\kappa_4+2)/c^2)^2 \right) }} \simeq \mathcal{N}(0,1), 
\end{equation}
and for $c=t+t^{-1}, t>1$
\begin{equation}\label{eq_globalnulltwo}
\mathsf{T}_3^{\mathrm{g}}:=\frac{\mathcal{T}_3^{\mathrm{g}}-n \left[(c-\phi^{-1/2})+\int \log (x-\phi^{1/2}-\phi^{-1/2}+c)\mathrm{d}\varrho_0(x) \right]-\frac{1}{2}\log(1-t^{-2})+\kappa_4/2t^2}{\sqrt{(2+\kappa_4)\left(1-2/t\right)+ 2\left(\log t-\log(t-1/t)\right)+\kappa_4/t^2}} \simeq \mathcal{N}(0,1). 
\end{equation} 
\item For the local statistics, we have that for $1 \leq \ell \leq 3,$
\begin{equation}\label{eq_locallnullone}
\mathsf{T}_\ell^{\mathrm{l}}:=\frac{\mathcal{T}_\ell^{\mathrm{l}}-n\mathsf{m}_\ell^{\mathrm{l}}-h_\ell(0)/4}{\sqrt{\mathsf{v}_\ell^{\mathrm{l}}}} \simeq \mathcal{N}(0,1),
\end{equation}
and 
\begin{equation}\label{eq_locallnulltwo}
\mathsf{T}_4^{\mathrm{l}}:=\frac{\mathcal{T}_4^{\mathrm{l}}-n(\mathsf{m}_2^{\mathrm{l}}+h_2(0)/4n)/(\mathsf{m}_1^{\mathrm{l}}+h_1(0)/4n)^2}{\sqrt{\mathsf{v}_4^{\mathrm{l}}}} \simeq \mathcal{N}(0,1). 
\end{equation}
\end{enumerate}
\end{cor}

Before concluding this section, we conduct a few numerical simulations to illustrate the accuracy of our proposed statistics in Figures \ref{fig_typeonefigone} and \ref{fig_typeonefigtwo}. We can conclude that all the proposed statistics are accurate and robust against various different choices of parameters.

\begin{figure}[!ht]
\hspace*{-0.5cm}
\subfigure{\includegraphics[width=4.3cm,height=4cm]{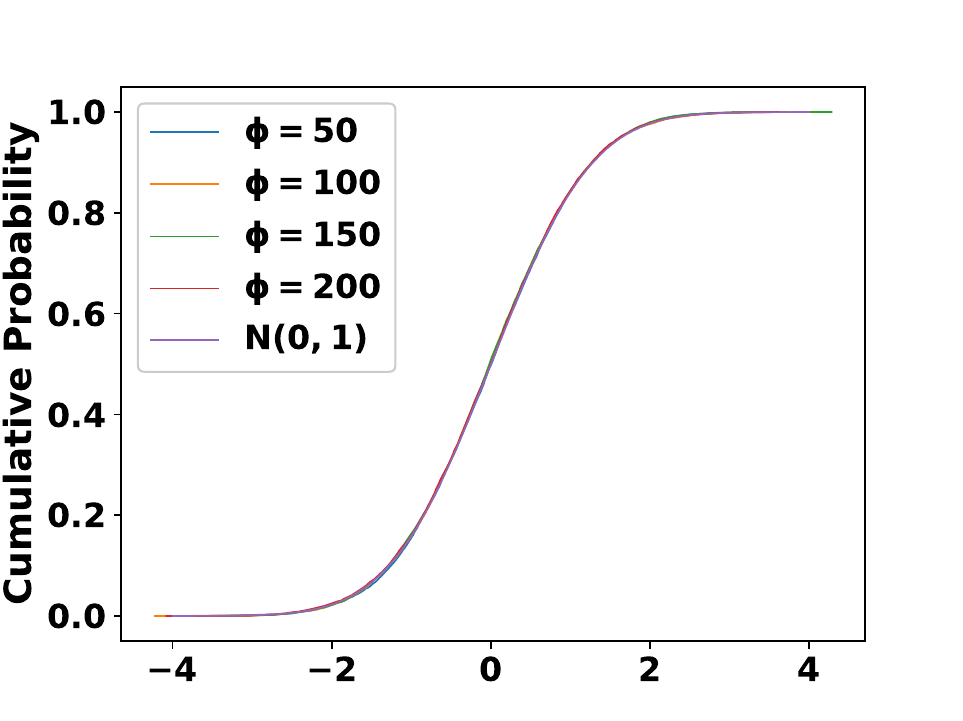}}
\hspace*{-0.5cm}
\subfigure{\includegraphics[width=4.3cm,height=4cm]{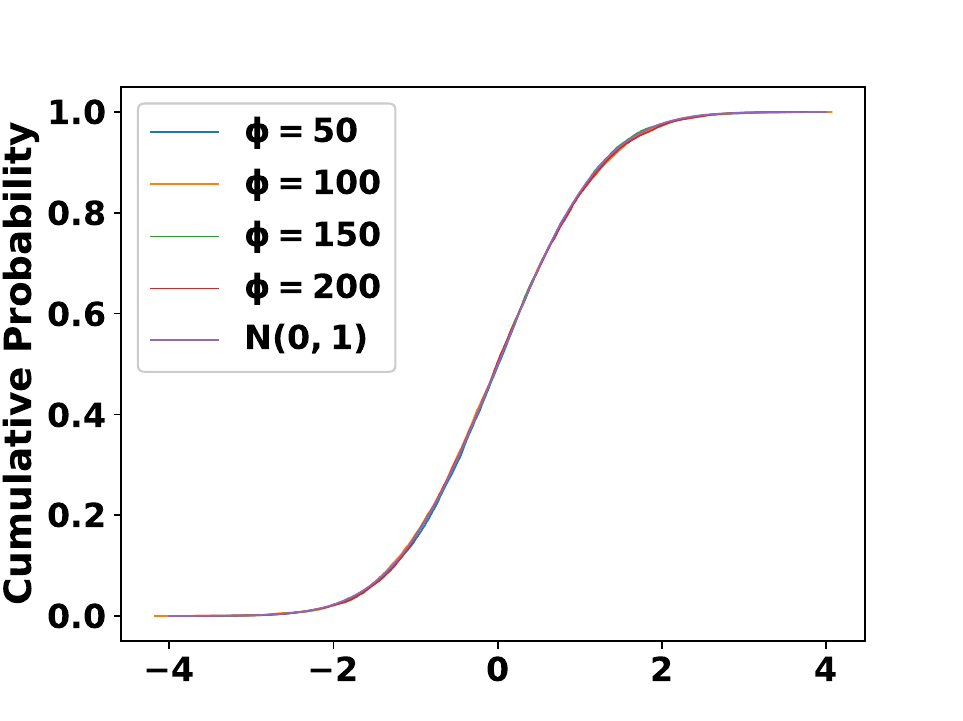}}
\hspace*{-0.2cm}
\subfigure{\includegraphics[width=4.3cm,height=4cm]{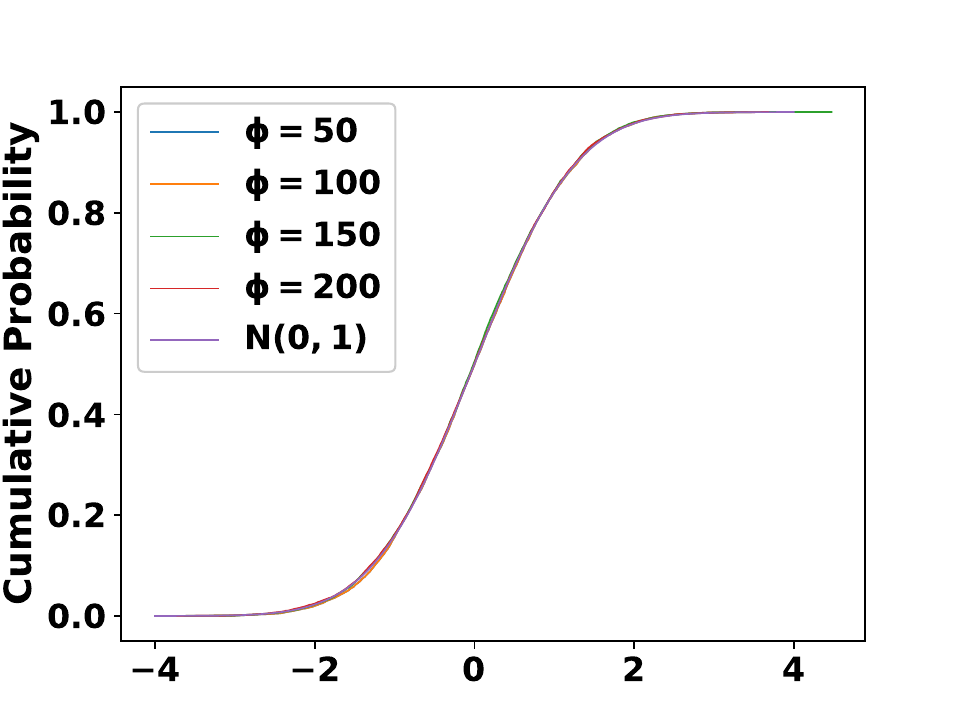}}
\hspace*{-0.2cm}
\subfigure{\includegraphics[width=4.3cm,height=4cm]{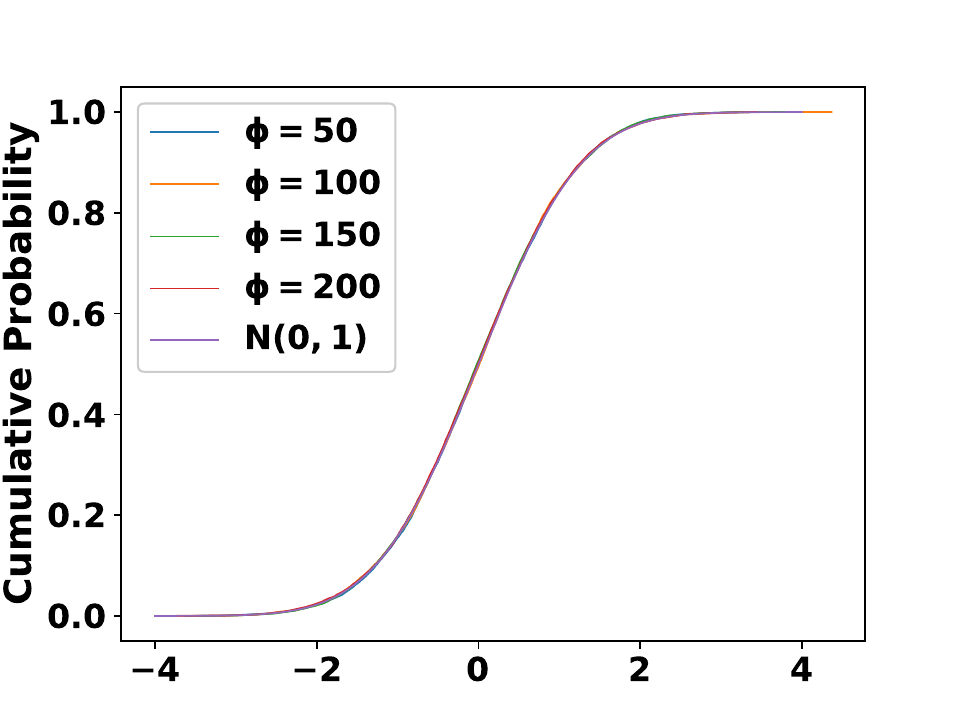}}
\caption{ECDFs for global testing statistics with different $\phi$'s. From left to right are: $\mathsf{T}_1^\mathrm{g}$ with $c=3$,$\mathsf{T}_2^\mathrm{g}$ with $c=3,$ $\mathsf{T}_3^\mathrm{g}$ with $c=3+3^{-1}$ and $\mathsf{T}_4^\mathrm{g}$ with $c=3.$ Here $x_{ij} \sim \mathcal{N}(0, (pn)^{-1/2})$,  $n=400$ and the results are reported based on 1,000 repetitions.  }\label{fig_typeonefigone}
\end{figure}

\begin{figure}[!ht]
\hspace*{-0.5cm}
\subfigure{\includegraphics[width=4.3cm,height=4cm]{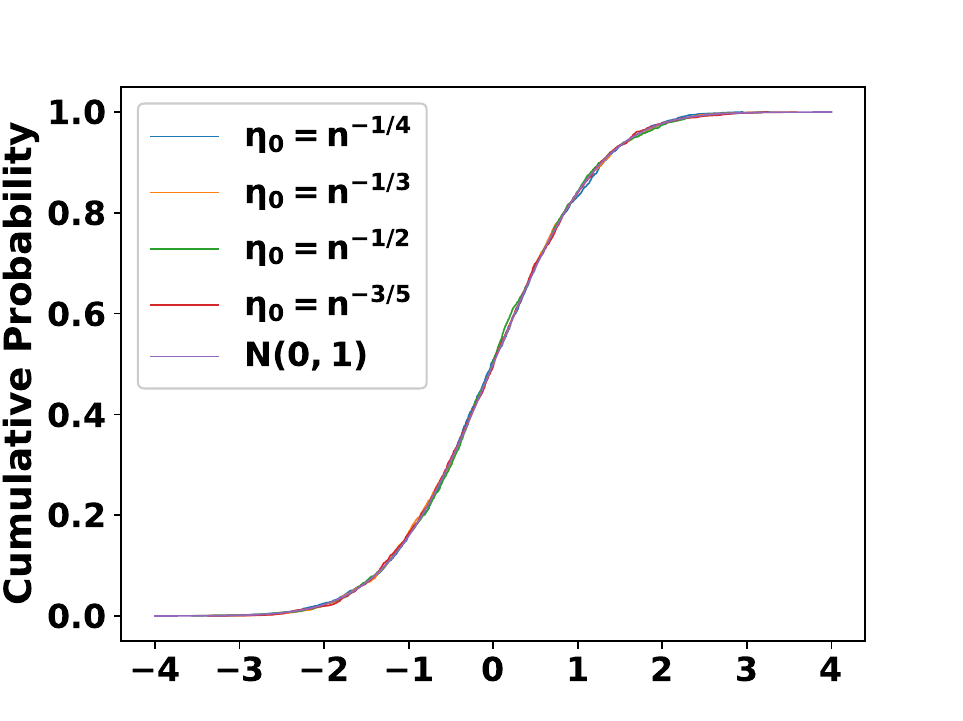}}
\hspace*{-0.5cm}
\subfigure{\includegraphics[width=4.3cm,height=4cm]{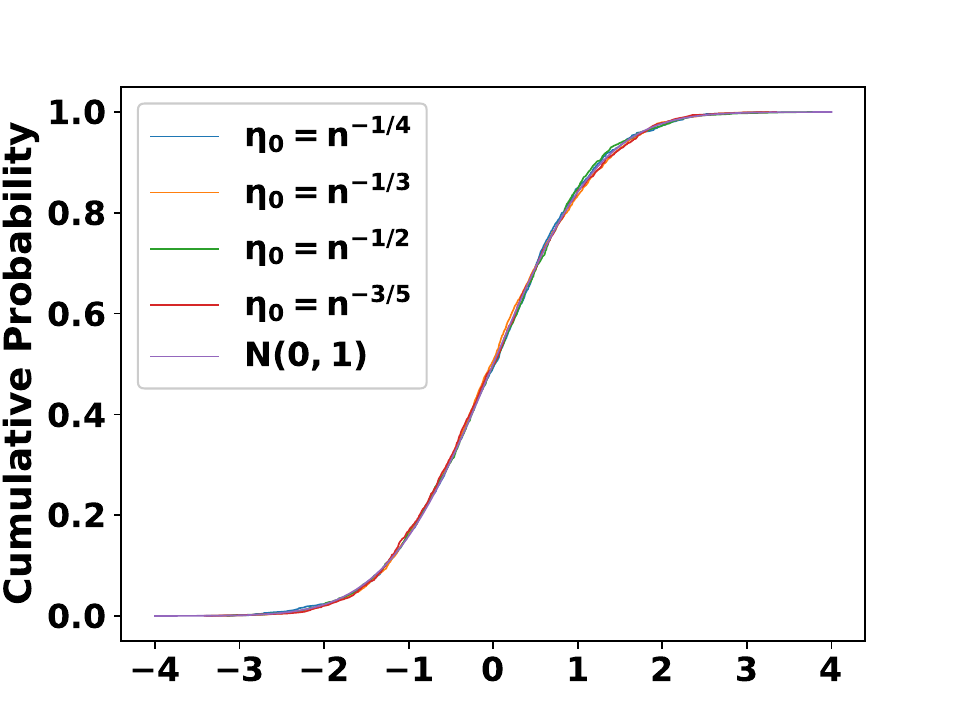}}
\hspace*{-0.2cm}
\subfigure{\includegraphics[width=4.3cm,height=4cm]{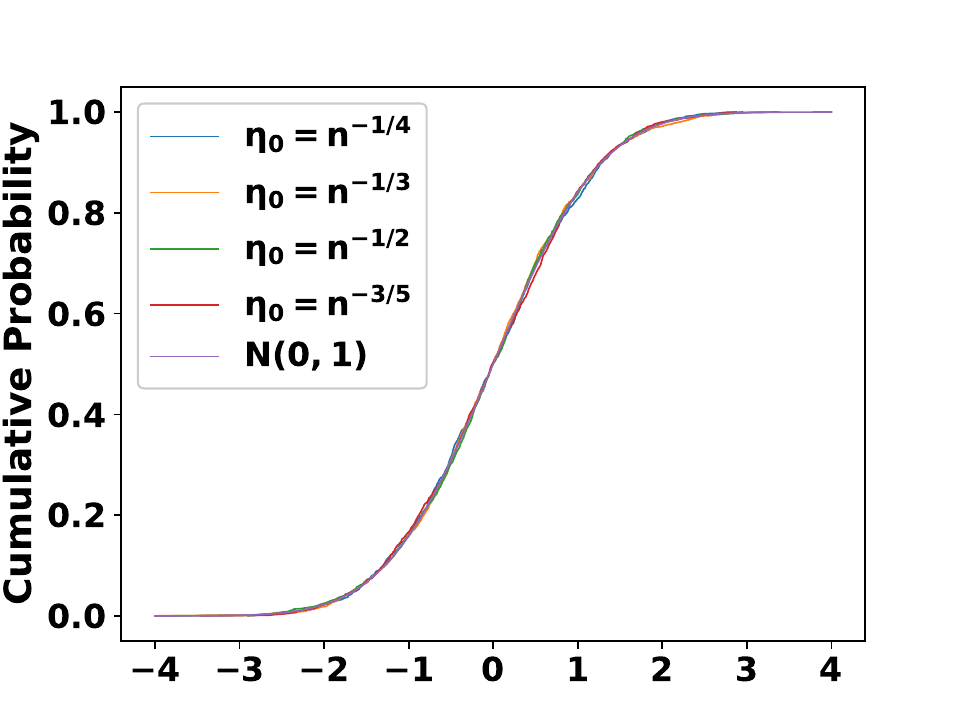}}
\hspace*{-0.2cm}
\subfigure{\includegraphics[width=4.3cm,height=4cm]{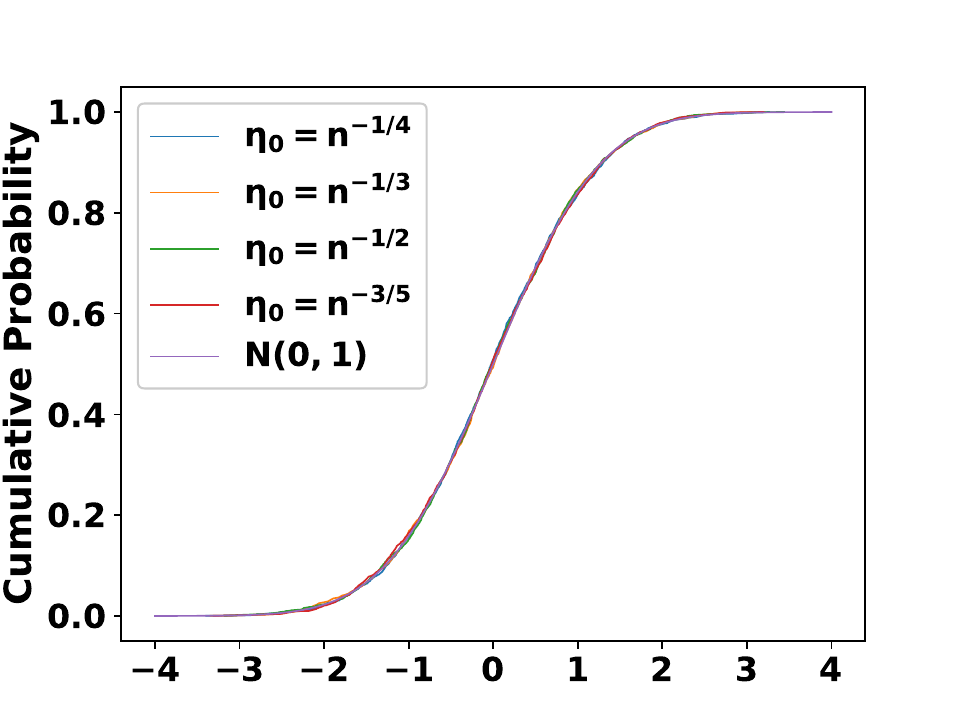}}
\caption{ECDFs for local testing statistics with different $\eta_0$'s. From left to right are: $\mathsf{T}_1^\mathrm{l}$,$\mathsf{T}_2^\mathrm{g}$, $\mathsf{T}_3^\mathrm{g}$ with $c=3+3^{-1}$ and $\mathsf{T}_4^\mathrm{g}$. Here $x_{ij} \sim \mathcal{N}(0, (pn)^{-1/2})$,  $\phi=100, n=400$ and the results are reported based on 1,000 repetitions.  }\label{fig_typeonefigtwo}
\end{figure}

\subsection{Power analysis of the proposed statistics}
In this section, we analyze the power of statistics proposed in Section \ref{sec_stattheory} under the alternative that
\begin{equation}\label{alternative}
\mathbf{H}_a: \Sigma \neq I.
\end{equation} 
For simplicity, we denote the ESD of $\Sigma$ under (\ref{alternative}) as $\pi_a$ and for $s=1,2,$
\begin{equation}\label{eq_momentalternative}
\mathfrak{m}_s(\pi_a)=\int x^s \pi_a(\mathrm{d}x). 
\end{equation}
Let $z_{1-\alpha/2}$ be the $1-\alpha/2$ quantitle of the $\mathcal{N}(0,1)$ random variable. The power of our proposed statistics is summarized in the following corollary.

\begin{cor}\label{cor_poweranalysis}
Suppose the assumptions of Theorems \ref{thm_mainone} and \ref{thm_maintwo} hold. Given some type I error rate $\alpha,$ under (\ref{alternative}), we have that  
\begin{equation*}
\lim_{n \rightarrow \infty} \mathbb{P} \left(|\mathsf{T}_\ell^{\mathsf{sy}}|>z_{1-\alpha/2} \right)=1,
\end{equation*}
where $1 \leq \ell \leq 4$ and $\mathsf{sy}=\mathrm{g}, \mathrm{l}$ are defined in Corollary \ref{coro_statisticalapplications}, if the following conditions are satisfied for some divergent sequence $C_{\alpha} \equiv C_{\alpha}(n) \rightarrow \infty$: 
\begin{enumerate}
\item For the global statistics $\mathsf{sy}=\mathrm{g},$ when $\ell=1,3,$ we require 
\begin{equation}\label{eq_conditiononeoeoneonenenen}
|\mathfrak{m}_1(\pi_a)-1|>C_{\alpha}n^{-1} \phi^{-1/2},
\end{equation}
and when $\ell=2,4$
\begin{equation}\label{eq_conditiontwo}
\left((\mathfrak{m}_1(\pi_{a})-1)^2{\phi}+c|\mathfrak{m}_1(\pi_{a})-1|\phi^{1/2}+|\mathfrak{m}_2(\pi_{a})^2-1|\right)>C_{\alpha} n^{-1}.
\end{equation}
\item For the local statistics $\mathsf{sy}=\mathrm{l},$ we require (\ref{eq_conditiontwo}) for all $\ell=1,2,3,4.$ 
\end{enumerate}
\end{cor}

We conclude that our proposed statistics are powerful even for weak local alternatives once (\ref{eq_conditiononeoeoneonenenen}) or (\ref{eq_conditiontwo}) are satisfied. These conditions can be easily satisfied by some commonly used alternatives. For example, we can consider the spiked covariance models in the sense that
\begin{equation*}
\mathbf{H}_a: \Sigma= \epsilon I_r+I,
\end{equation*}
where $I_r$ is a $r \times r$ identity matrix. For this case, to make (\ref{eq_conditiononeoeoneonenenen}) valid, we only require that $\epsilon r>C_{\alpha} p^{-1/2} n^{-1/2}.$ For another instance,  we can consider the covariance matrices with different clusters in the sense that 
\begin{equation}\label{eq_usedalternative}
\mathbf{H}_a: \Sigma= a \delta_{1+\epsilon}+(1-a) \delta_1. 
\end{equation}
In this setting, to satisfy (\ref{eq_conditiononeoeoneonenenen}), we only require $a \epsilon>C_{\alpha} p^{-1/2} n^{-1/2}.$ The requirements are weaker than those used in the most relevant work \cite{QiuPreprint} which utilizes global statistics based on the normalized matrix (\ref{eq)A}). This partially explains why we prefer (\ref{eq_samplecovaraincematrix}) to (\ref{eq)A}).

Before concluding this section, we conduct a few numerical simulations to illustrate the powerfulness of our proposed statistics in Figures \ref{fig_power11} under the alternative in (\ref{eq_usedalternative}). We can see that our proposed statistics are very powerful even for small values of $\epsilon.$
\begin{figure}[!ht]
\centering
\subfigure[Global testing statistics.]{\label{fig:a}\includegraphics[width=7.5cm,height=5cm]{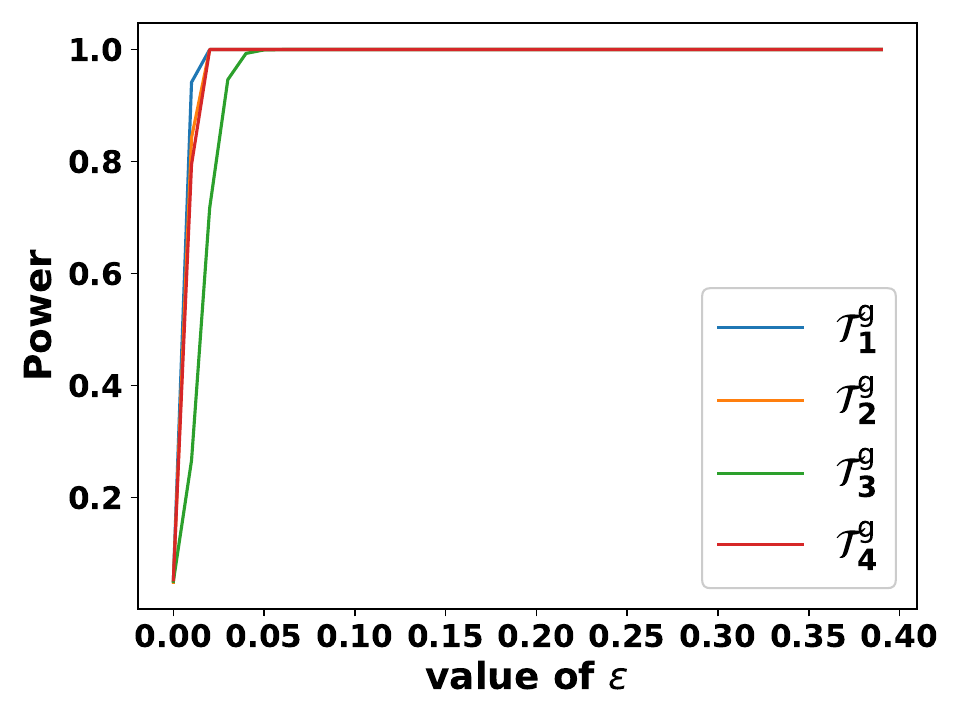}}
\hspace*{0.2cm}
\subfigure[Local testing statistics.]{\label{fig:b}\includegraphics[width=7.5cm,height=5cm]{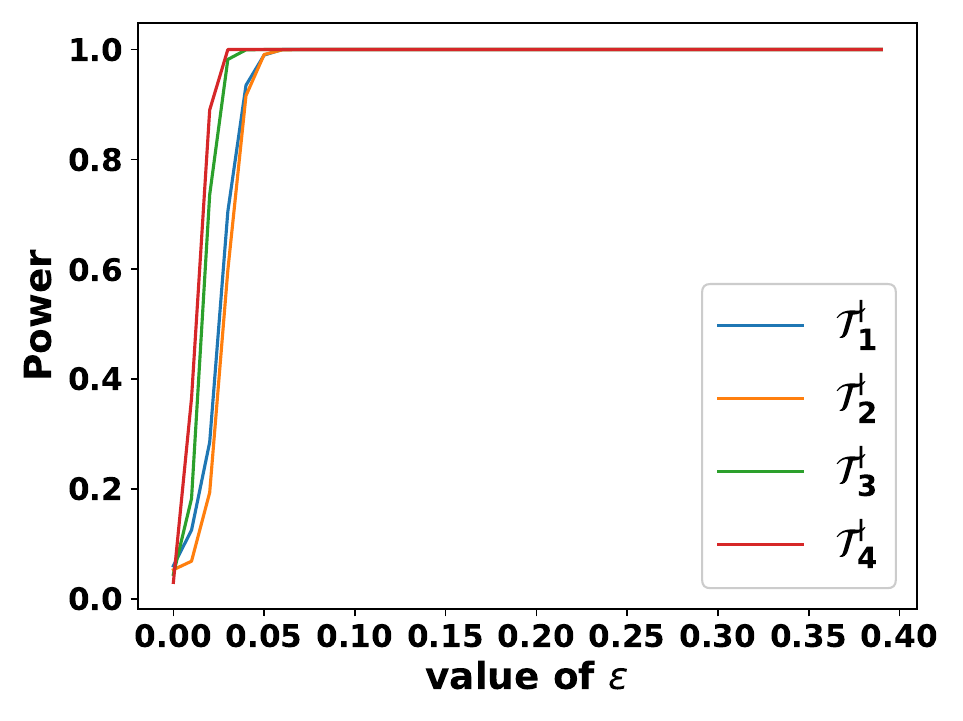}}
\caption{Simulated power for proposed statistics. Here we use Gaussian distributions for $x_{ij},$ $a=0.5$ in (\ref{eq_usedalternative}), $\phi=100, n=400$ and the other parameters $c$ and $\eta_0$ are chosen in the same way as in the captions of Figures \ref{fig_typeonefigone} and \ref{fig_typeonefigtwo}. The simulations are reported based on 1,000 repetitions.}
\label{fig_power11}
\end{figure}

\subsection{Comparison with existing methods}
In this section, we compare our proposed methods with  several existing ones in the literature. For concreteness, we focus on 
	\cite{Nagao1973,LedoitWolf2002,Srivastava2005,Chen2010,Srivastava2011,Fisher2012,BaiJiangYaoZheng2009,Ahmad2015,QiuPreprint}. 
	
To evaluate the performance of these tests, and to discern the trade-off between Type I and Type II error rates, we employ the Receiver Operating Characteristic (ROC) curve and the Area Under Curve (AUC) score. The ROC curve is a plot where the $x$-axis represents the Type I error rate, or the probability of falsely rejecting a true null hypothesis (also known as a false positive rate), and the $y$-axis represents the power of the statistical test, which is the true positive rate (the ability to correctly accept an alternative hypothesis when it is true). The AUC score then quantifies the overall ability of a test to maximize the true positive rate while minimizing the Type I error rate. It does this by measuring the total area under the ROC curve - the larger the area, the better the test is at balancing these two types of errors.

%


We report our simulation results in Figures \ref{fig_ROCglobal} and \ref{fig_ROClocal} below. In addition to the aforementioned nine methods in the literature, we also study our proposed statistics. For concreteness, we use  $\mathsf{T}_1^{\mathrm{g}}$ and $\mathsf{T}_1^{\mathrm{l}}$ with $c=0$ and $\eta_0=n^{-1/4},$ denoted as PM1 and PM2 in the simulations. For the alternative, we consider  (\ref{eq_usedalternative}) with $a=0.5.$

As shown in Figure \ref{fig:6a}, when $p\gg n$ and $\epsilon=0.03$ is small, only three methods demonstrate significant power while maintaining appropriate size control. These methods are: our methods PM1, PM2 and the method suggested by \cite{QiuPreprint} (Qiu21) which coincides with our PM1 in this particular scenario. The remaining tests perform at a level that is largely indistinguishable from a random guess, indicated by an AUC score around $0.5$. Among the tests compared with unsatisfactory performance, \cite{Nagao1973} (Nagao73)  is derived in the low dimensional regime when $p$ is fixed. Moreover, \cite{LedoitWolf2002,BaiJiangYaoZheng2009,Fisher2012,Srivastava2005} denoted as Wolf02, CLRT09, Fisher12, Srivastava05, respectively in the figures, are derived in the comparable regime $p \asymp n$. Further,  \cite{Chen2010, Srivastava2011, Ahmad2015} denoted as Chen10,  Srivastava11, Ahmad15 in the figures, respectively require stronger alternatives for powerfulness.  When $\epsilon$ increases to 0.3, as shown in Figure \ref{fig:6b}, more methodologies like \cite{Srivastava2011,Chen2010,Fisher2012,BaiJiangYaoZheng2009} show significant improvements while \cite{QiuPreprint} and our two proposed methods are still among the best. 

We emphasize again, except for our PM2, all the other methods require prior knowledge of $\kappa_4.$	In Figure \ref{fig_ROClocal}, we examine the situation when wrongly estimated $\kappa_4$ is used to the conduct tests. It turns out our PM2 which is independent of $\kappa_4$ has the best performance.

\begin{figure}[!ht]
\centering
\subfigure[$\epsilon=0.03$]{\label{fig:6a}\includegraphics[width=8cm,height=6.5cm]{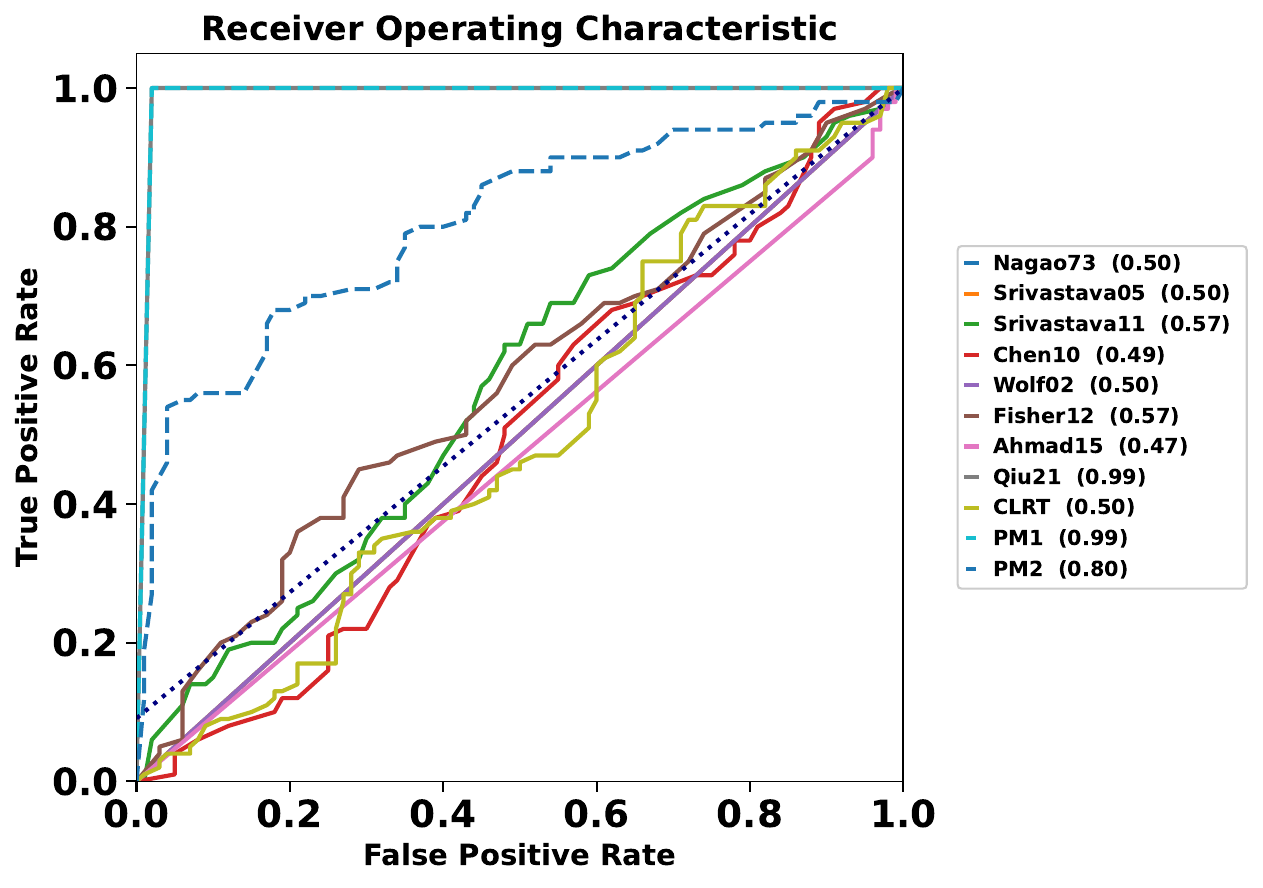}}
\hspace*{0.2cm}
\subfigure[$\epsilon=0.3$]{\label{fig:6b}\includegraphics[width=8cm,height=6.5cm]{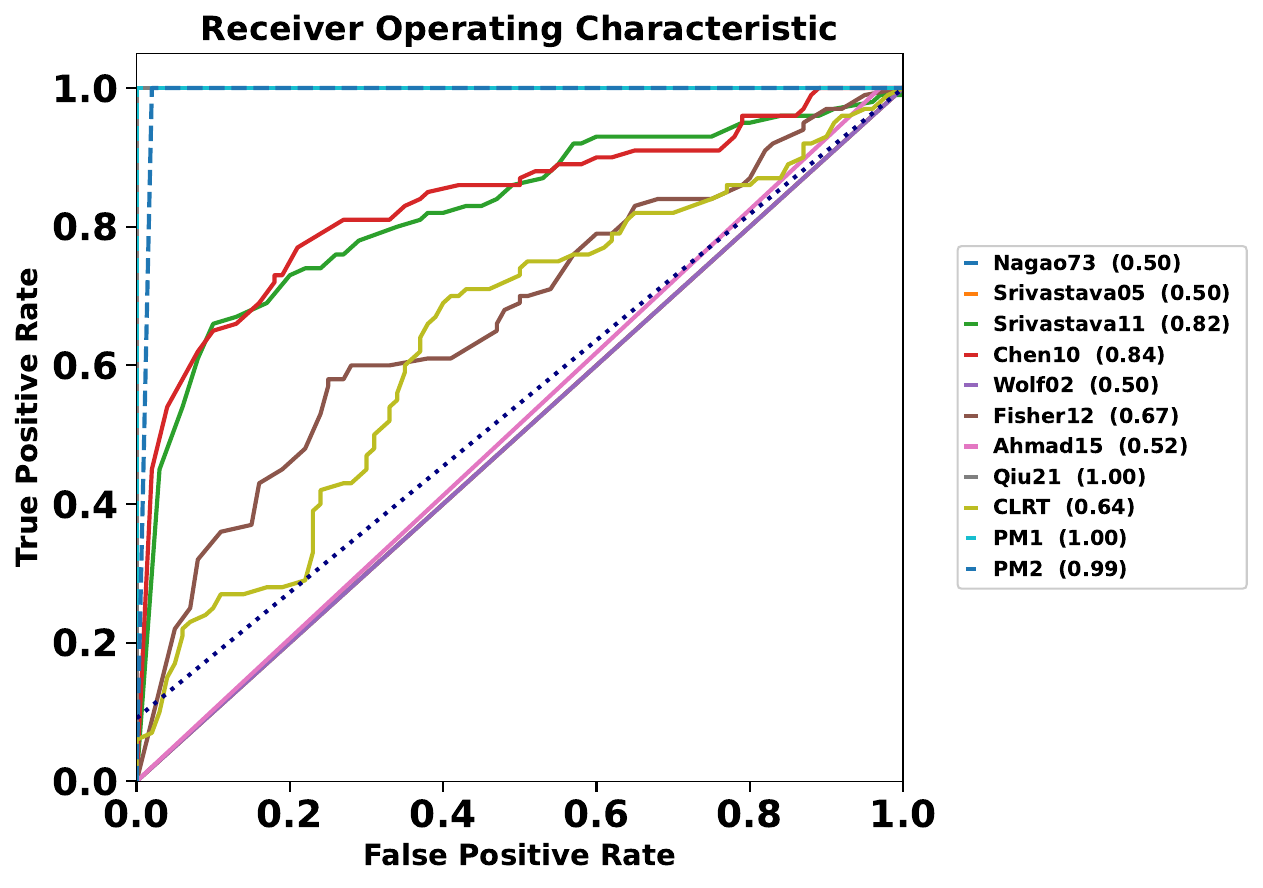}}
\caption{ROC curves assuming $\kappa_4=0$ is known. Here we use Gaussian distribution for $x_{ij},$ $\phi=100, n=400$ and the results are reported based on 1,000 repetitions. For the legends, the values inside the parentheses are the AUC scores.} \label{fig_ROCglobal}
\end{figure}

\begin{figure}[!ht]
\centering
\includegraphics[width=8cm,height=6cm]{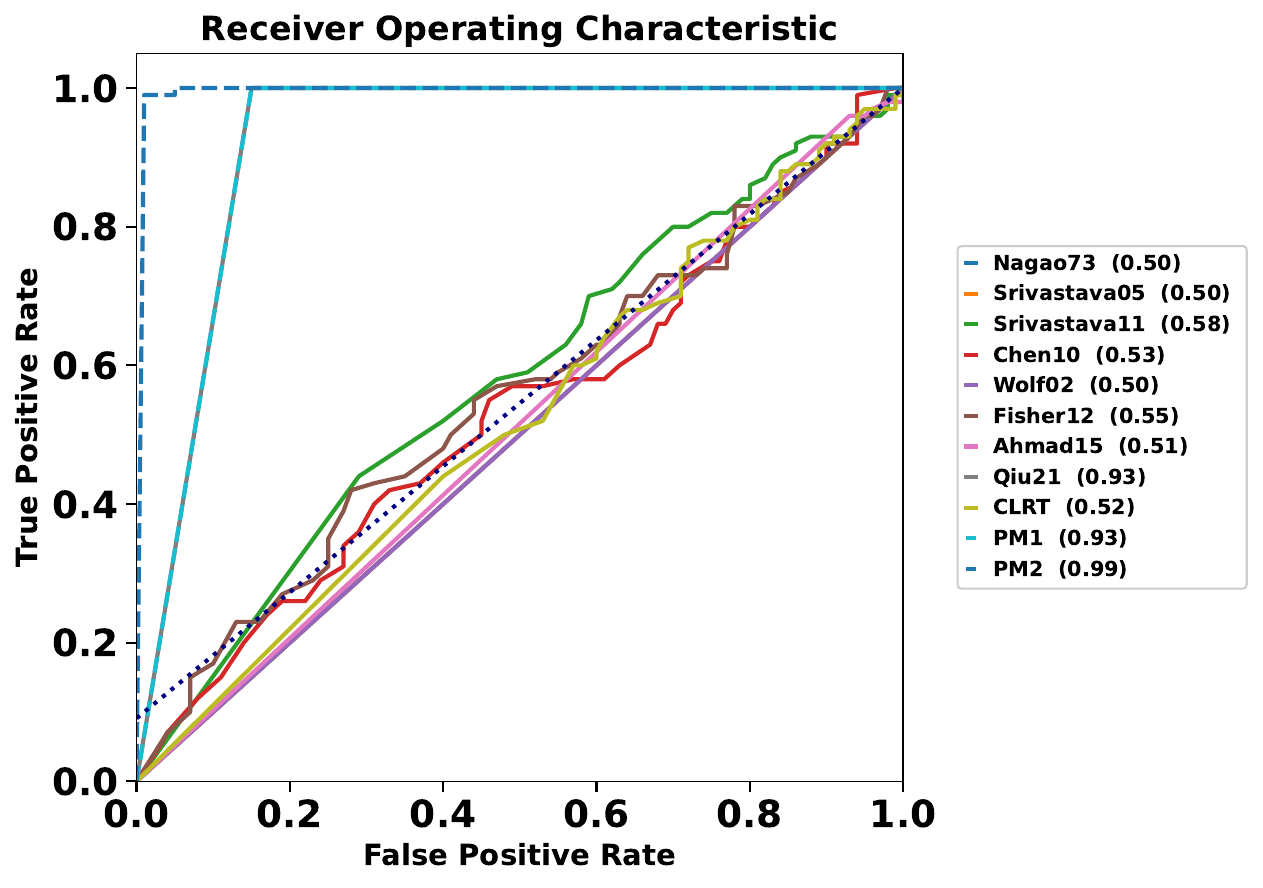}
\caption{ROC curves assuming $\kappa_4 \neq 0$ is misspecified. Here we use $\frac{\sqrt{2}+1}{2\sqrt{2}} \delta_{\sqrt{2}-1}+\frac{\sqrt{2}-1}{2\sqrt{2}} \delta_{-1-\sqrt{2}}$ with true $\kappa_4=2.$ In actual implementation, we used the misspecified estimate $\widehat{\kappa}_4=0.$ Here  $\epsilon=0.08, \phi=100, n=400$ and the results are reported based on 1,000 repetitions. }\label{fig_ROClocal}
\end{figure}

\section{Preliminaries, local laws and tools}\label{sec_preliminary}
The rest of the paper is devoted to the proof of the results in Sections  \ref{sec_CLTs} and \ref{sec_stattheory}. In this section, we provide the key ingredients for our proofs. We first summarize the elementary properties of $m(z)$ defined in Lemma \ref{lem_asymptoticlaw}. 
For $z=E+\ii \eta \in \mathbb{C}$, define 
	$$
	\kappa \equiv \kappa(z):=\min \left\{\left|\gamma_{+}-E\right|,\left|\gamma_{-}-E\right|\right\} .
	$$  
Recall $\mathfrak{m}_1(\pi)=\int x \pi (\dd x).$ Moreover, for some fixed small constant $0<\tau<1,$ denote the spectral parameter set as 	
			\begin{equation}\label{eq_spectralparameterset}
				\begin{aligned}
					\mathbf{S} \equiv \mathbf{S}(\tau)
					:=\left\{z=E+\mathrm{i} \eta \in \mathbb{C}: |E-\mathfrak{m}_1(\pi)\phi^{1/2}|\le \tau^{-1}, n^{-1+\tau} \leqslant \eta \leqslant \tau^{-1},|z| \geqslant \tau\right\}.
				\end{aligned}
			\end{equation}

\begin{lem} \label{lem:mphizabs}
Fix some small constant $0<\tau<1.$ Suppose the assumptions of Lemma \ref{lem_asymptoticlaw} and (\ref{eq_ratioassumption}) hold. Then for all $z \in \mathbf{S}$ in (\ref{eq_spectralparameterset}), we have that for some constant $c>0,$
\begin{equation}\label{eq_constantbound}
m(z) \asymp 1, \ \  \operatorname{Im} m(z) \geq  c \eta,
\end{equation}
and 
\begin{equation}\label{eq_squarerootbehavior}
 \operatorname{Im} m(z) \asymp 
 \begin{cases}
 \sqrt{\kappa+\eta} , & E \in [\gamma_-,\gamma_+] \\
 \frac{\eta}{\sqrt{\kappa+\eta}}, & E \notin [\gamma_-, \gamma_+]. 
 \end{cases}
\end{equation}
Moreover, for its derivative $m',$ we have that 
\begin{equation}\label{eq_derivativecontrol}
m' \asymp \frac{1}{\sqrt{\kappa+\eta}}.
\end{equation}
\end{lem}
\begin{proof}
See Appendices \ref{appendix_section_priorbound} and \ref{appendix_sectionlemma25}. 
\end{proof}

Let  $m_p(z)$ be the limit of Stieltjes transform of $\mathcal{Q}.$ Since $Q$ and $\mathcal{Q}$ share the same non-zero eigenvalues, for $m(z)$ defined in (\ref{eq:lsd}), we have that 
\begin{equation}\label{eq_transformtransferdefinition}
m(z)=-\frac{1-{\phi}}{z}+{\phi} m_{p}(z).
\end{equation}
Further, we use $\rho_p$ to denote the measure induced by $m_p(z).$ In the following lemma, we summarize some important identities. 

\begin{lem}\label{lem_elementaryidentityes}
For $m(z)$ defined in (\ref{eq_otherformoflsd}) and $m_p(z)$ defined in (\ref{eq_transformtransferdefinition}), we have the following identities 
\begin{equation}\label{eq_eeeeeee}
		\begin{aligned}
			\frac{1}{m}=&-z+\frac{\phi}{p}\sum_{i=1}^p\frac{\phi^{-1/2}\sigma_i}{1+\phi^{-1/2}m\sigma_i}, \\ 
			\frac{1}{z}=&-m+\frac{\phi}{p}\sum_{i=1}^p\frac{\phi^{-1/2}\sigma_im}{z(1+\phi^{-1/2}m\sigma_i)}=-m+\frac{\phi}{pz}\sum_{i=1}^p\left(1-\frac{1}{1+\phi^{-1/2}m\sigma_i} \right),\\ 
			m_p=&\frac{1}{\phi}\left(m+\frac{1-\phi}{z} \right)=-\frac{1}{p}\sum_{i=1}^p\frac{1}{z(1+\phi^{-1/2}m\sigma_i)},
		\end{aligned}
	\end{equation}
and 
\begin{equation}\label{eq_mmprimeeq1}
	1-\frac{1}{p} \sum_{i=1}^p \frac{\phi^{1/2} \sigma_i}{z\left(1+\phi^{-1/2}{m}(z) \sigma_i\right)^2}=-\frac{{m}(z)}{z {m}^{\prime}(z)}.
\end{equation}	
\end{lem}
\begin{proof}
(\ref{eq_eeeeeee}) follows directly from (\ref{eq_otherformoflsd}) and (\ref{eq_transformtransferdefinition}). To see (\ref{eq_mmprimeeq1}), we can rewrite (\ref{eq_otherformoflsd}) as 
\begin{equation}\label{eq_mmprimeeq2}
	1+zm(z)=\phi-\frac{1}{p}\sum_{i=1}^p\frac{\phi}{1+\phi^{-1/2}\sigma_im(z)}.
\end{equation}
Then taking derivative on both sides of the above equation, we obtain that 
\begin{equation}\label{eq_mmprimeeq3}
	m(z)+zm^{\prime}(z)=\frac{1}{p}\sum_{i=1}^p\frac{\phi^{1/2}\sigma_im^{\prime}(z)}{(1+\phi^{-1/2}\sigma_im(z))^2}.
\end{equation}
This proves (\ref{eq_mmprimeeq1}). 
\end{proof}

Next, as in \cite{DY, AILL, LS, Yang2020}, we see that it is more convenient to work with some linearization of the matrix $X.$ For notational convenience, we introduce the following conventions.  

\begin{defn}\label{eq_indexset} We introduce the index sets
	$$
	\mathcal{I}_1:=\llbracket 1, p \rrbracket, \quad \mathcal{I}_2:=\llbracket p+1, p+n \rrbracket, \quad \mathcal{I}:=\mathcal{I}_1 \cup \mathcal{I}_2=\llbracket 1, p+n \rrbracket .
	$$
	We consistently use the letters $i, j \in \mathcal{I}_1, \mu, \nu \in \mathcal{I}_2$, and $s, t \in \mathcal{I}$. With the above notations,  we can label the indices of the matrices according to
	$$
	X=\left(X_{i \mu}: i \in \mathcal{I}_1, \mu \in \mathcal{I}_2\right), \quad \Sigma=\left(\Sigma_{i j}: i, j \in \mathcal{I}_1\right) .
	$$
\end{defn}
The following linearizing block matrices will play important roles in our calculations. For $z \in \mathbb{C}_+,$ define  
\begin{equation}\label{eq_BIGG}
G(z)\equiv G(z, \Sigma, X):=\left(\begin{array}{cc}
	-\Sigma^{-1} & X \\
	X^{*} & -z I_n
\end{array}\right)^{-1}.
\end{equation}
Moreover, for $m(z)$ defined in (\ref{eq:lsd}), we denote 
\begin{equation}\label{eq:pidefi}
 \Pi(z)\equiv \Pi(z, \Sigma):=-z^{-1}\left(I_p+\phi^{-1/2}m(z) \Sigma\right)^{-1}. 
\end{equation}
For $z \in \mathbb{C}_+,$ we define the resolvents of $Q$ and $\mathcal{Q}$ respectively as 
\begin{equation}\label{eq_twoblockmatrices}
R_1(z):=\left(\Sigma^{1/2} X X^* \Sigma^{1/2}-z\right)^{-1}, \ R_2(z):=\left(X^{*}\Sigma X-z\right)^{-1}. 
\end{equation}
We will see later from Theorem \ref{thm_locallaw} that $\Pi(z)$ is the deterministic counterpart of $R_1(z).$  Denote the control parameter $\Psi(z)$ as 
\begin{equation}\label{eq_fundementalparameter}
\Psi(z):=\sqrt{\frac{\operatorname{Im} m(z)}{n \eta}}+\frac{1}{n \eta}. 
\end{equation}
We now state the main results on the local laws which provide precise controls on $G, R_1$ and $R_2.$  We will need the following notion of stochastic domination which is commonly used in modern random matrix theory \cite{erdHos2017dynamical}. 
\begin{defn}[Stochastic domination]
(i) Let
\[A=\left(A^{(n)}(u):n\in\mathbb{N}, u\in U^{(n)}\right),\hskip 10pt B=\left(B^{(n)}(u):n\in\mathbb{N}, u\in U^{(n)}\right),\]
be two families of nonnegative random variables, where $U^{(n)}$ is a possibly $n$-dependent parameter set. We say $A$ is stochastically dominated by $B$, uniformly in $u$, if for any fixed (small) $\epsilon>0$ and (large) $D>0$, 
\[\sup_{u\in U^{(n)}}\mathbb{P}\left(A^{(n)}(u)>n^\epsilon B^{(n)}(u)\right)\le n^{-D},\]
for large enough $n \ge n_0(\epsilon, D)$, and we shall use the notation $A\prec B$. Throughout this paper, the stochastic domination will always be uniform in all parameters that are not explicitly fixed, such as the matrix indices and the spectral parameter $z$.  Throughout the paper, even for negative or complex random variables, we also write $A \prec B$ or $A=\mathrm{O}_\prec(B)$ for $|A| \prec B$.
\vspace{5pt}

\noindent (ii) We say an event $\Xi$ holds with high probability if for any constant $D>0$, $\mathbb P(\Xi)\ge 1- n^{-D}$ for large enough $n$.
\end{defn}

\begin{thm}\label{thm_locallaw}  Suppose the assumptions of Theorem \ref{thm_mainone} hold. For all $z \in \mathbf{S}$ defined in (\ref{eq_spectralparameterset}) uniformly and  any deterministic vectors $\mathbf{u}_k \in \mathbb{R}^p, \mathbf{v}_k \in \mathbb{R}^n, k=1,2,$ we have that:
\begin{enumerate}
\item[(1).] For the resolvents of $Q$ and $\mathcal{Q}$ defined in (\ref{eq_twoblockmatrices}), we have that 
\begin{equation}\label{eq_locallawequationone}
\mathbf{u}_1^* R_1(z) \mathbf{u}_2=\mathbf{u}_1^* \Pi(z) \mathbf{u}_2+\rO_{\prec} \left( \phi^{-1} \Psi(z) \right),
\end{equation} 
where $\Pi(z)$ is defined in (\ref{eq:pidefi}) and $\Psi(z)$ is defined in (\ref{eq_fundementalparameter}). Moreover, we have that
\begin{equation}\label{eq_locallawequationtwo}
\mathbf{v}_1^* R_2(z) \mathbf{v}_2=m(z)\mathbf{v}_1^*\mathbf{v}_2+\rO_{\prec} \left( \Psi(z) \right).
\end{equation} 
\item[(2).] Denote the normalized resolvents as 
\begin{equation*}
m_{1n}(z)=\frac{1}{p} \sum_{i=1}^p (R_1(z))_{ii},  \  m_{2n}(z)=\frac{1}{n} \sum_{i=1}^n (R_2(z))_{ii}. 
\end{equation*}
We have that
\begin{equation}\label{eq_averagelocallawone}
m_{1n}(z)=-\frac{1}{zp} \sum_{i=1}^p \frac{1}{1+ \sigma_i \phi^{-1/2} m(z)}+\rO_{\prec} \left( (p \eta)^{-1}\right) ,  
\end{equation}
and 
\begin{equation}\label{eq_averagelocallawtwo}
m_{2n}(z)=m(z)+\rO_{\prec} \left( (n \eta)^{-1}\right).
\end{equation}
\item[(3).] Recall the conventions for the indices in Definition \ref{eq_indexset}. For $G(z)$ defined in (\ref{eq_BIGG}), we have that uniformly in $\mu\in \mathcal{I}_2$ and $i\in \mathcal{I}_1$ 
	\begin{equation}\label{eq:entrywiselaw}	
	G_{i\mu}(z)=\mathrm{O}_{\prec}\left(\phi^{-1/4} \Psi(z)\right).
	\end{equation}
%
\end{enumerate}

\end{thm}
\begin{proof}
See Appendix \ref{appendix_proof54}. 
\end{proof}

\begin{rem} {\normalfont
Several remarks are in order. First, the counterparts of the results for $\alpha=1$ have been proved in \cite{AILL}. In their setting, since $p$ and $n$ are comparably large, the results are essentially stated for the linearization matrix $G(z)$. However, as can be seen from our results in Theorem \ref{thm_locallaw}, the convergence rates of different blocks of $G(z)$ vary according to their sizes so that we need to state the results separately by carefully keeping tracking of $p,n$ and $\phi$. Second, for the results in \cite{AILL}, as mentioned in Remark \ref{rem_lemma25}, since there may exist several bulk components so that we need more regularity conditions as discussed in Remark \ref{rem_mainresultone}. In our setting (\ref{eq_ratioassumption}), most of these conditions are satisfied automatically. In fact, as can be concluded from the proofs of Theorem \ref{thm_locallaw}, the arguments applies to all $0<\alpha<\infty$ so that the results of \cite{AILL} can be actually recovered by our results. Third, in the current paper, we only need the entry-wise control (\ref{eq:entrywiselaw}) for the off-diagonal block terms of the matrix $G(z)$ in (\ref{eq_BIGG}). In fact, with additional technical efforts, we can show that for all $\bar{\mathbf{u}}_1 \in \mathbb{R}^{p+n}$ and $\bar{\mathbf{v}}_1 \in \mathbb{R}^{p+n},$ where  $\bar{\mathbf{u}}_1$  is the natural embedding of $\mathbf{u}_1$ and $\bar{\mathbf{v}}_1 $ is that of $\mathbf{v}_1,$ we have that 
\begin{equation*}
\bar{\mathbf{u}}_1^* G(z) \bar{\mathbf{v}}_1=\mathrm{O}_{\prec}\left(\phi^{-1/4} \Psi(z)\right). 
\end{equation*}
}
\end{rem}

Before concluding this section, we collect some useful formulas which will be used in our proof. The following cumulant expansion formula plays a crucial role in our calculations which have been used in calculating the CLTs for various random matrix models, to name but a few, \cite{BDW,bao2022statistical,Li2021,Yang2020}.  

\begin{lem}\label{lem_cumulantexpansion} Fix $\ell \in \mathbb{N}$ and let $h$ be a real-valued random variable with finite moments up to order $2 \ell +2.$ Moreover, let $f$ be a complex-valued smooth function that $f \in \mathcal{C}^{\ell+1}(\mathbb{R}).$ Let $\kappa_k$ be the $k$-th cumulant of $h,$ given by $\kappa_k:=(-\mathrm{i})^k \frac{\mathrm{d}}{\mathrm{d} t} \log \mathbb{E} e^{\mathrm{i} th}|_{t=0}.$ Then we have that 
\begin{equation*}
\mathbb{E}\left[ h f(h) \right]=\sum_{k=0}^{\ell} \frac{1}{k!} \kappa_{k+1} \mathbb{E} f^{(k)}(h)+R_{\ell+1}, 
\end{equation*} 
where the error term satisfies 
\begin{equation*}
|R_{\ell+1}| \leq C_{\ell} \mathbb{E} |h|^{\ell+2} \sup_{|x| \leq M}|f^{(\ell+1)}(x)|+C_{\ell} \mathbb{E} \left[ |h|^{\ell+2} \mathbf{1}_{|h|>M} \right] \|f^{(\ell+1)}(x)\|_{\infty},
\end{equation*}
for some constant $C_{\ell}>0$ and $M>0$ is an arbitrary fixed cutoff.  
\end{lem}
\begin{proof}
See Proposition 3.1 of \cite{LPAOP}. 
\end{proof}
Then we introduce the Helffer-Sj{\" o}strand formula which will connect the LSS with the resolvents. 
\begin{lem}\label{lem_HSformula}
Let $r \in \mathbb{N}$ and $f \in \mathcal{C}^{r+1}(\mathbb{R})$ and $\chi(y)$ be a smooth cutoff function with bounded support. Define its almost-analytic extension of degree $r$ as 
\begin{equation*}
\widetilde{f}_r(x+\mathrm{i}y):=\sum_{k=0}^r \frac{1}{k!} (\mathrm{i} y)^k f^{(k)}(x). 
\end{equation*}
	Let $\chi \in \mathcal{C}_c^{\infty}(\mathbb{C} ;[0,1])$ be a smooth cutoff function. Then for any $\lambda \in \mathbb{R}$ satisfying $\chi(\lambda)=1$ we have
$$
	f(\lambda)=\frac{1}{\pi} \int_{\mathbb{C}} \frac{\bar{\partial}\left(\tilde{f}_r(z) \chi(z)\right)}{\lambda-z} \mathrm{~d}^2 z, \ z=x+\mathrm{i} y, 
	$$
where $\bar{\partial}=\frac{1}{2} \left( \frac{\partial}{\partial x}+\mathrm{i} \frac{\partial}{\partial y} \right)$ is the antiholomorphic derivative and $\mathrm{d}^2 z$ is the Lebesgue measure on $\mathbb{C}.$ 
\end{lem}

\begin{proof}
See \cite{davies1995functional} or \cite[Section 1.13.3]{benaych2017advanced}. 
\end{proof}

Next, we record the Green's theorem in the complex form which  can be used to transfer the region integral to contour integral. 

\begin{lem}\label{lem_Greenthm}
Suppose $B(z, \bar{z})$ is continuous and has continuous partial derivatives in a region $\mathcal{R}$ and on its boundary. We have the following complex form of Green's theorem $$
	\oint_{\partial \mathcal{R}} B(z, \bar{z}) \mathrm{d} z=2 \mathrm{i} \iint_{\mathcal{R}} \frac{\partial B}{\partial \bar{z}} \mathrm{d} x \mathrm{d} y.
	$$
	More specifically, when $B(z,\bar{z})=F(z,\bar{z})h(z)$ with $h(z)$ being holomorphic in $\mathcal{R}$ and on $\partial \mathcal{R}$, we have
	$$
	\oint_{\partial \mathcal{R}} B(z, \bar{z}) \mathrm{d} z=\oint_{\partial \mathcal{R}} F(z, \bar{z})h(z) \mathrm{d} z =2 \mathrm{i} \iint_{\mathcal{R}} \frac{\partial F(z,\bar{z})}{\partial \bar{z}}h(z) \mathrm{d} x \mathrm{d} y.
	$$
\end{lem}

\begin{proof}
See Theorem 4.10 of \cite{spiegel2009schaum}. 
\end{proof} 

Then we introduce the Sokhotski-Plemelj lemma. 
\begin{lem}\label{lem_SPformula}
	Let $f$ be a complex-valued function which can be continuously extended to the real line, and let
$ a<0<b$ be some real numbers. Then
$$\lim _{\varepsilon \to 0^{+}}\int _{a}^{b}{\frac {f(x)}{x\pm \mathrm{i}\varepsilon }} \mathrm{d}x=\mp \mathrm{i}\pi f(0)+{\mathcal {P}}\int _{a}^{b}{\frac {f(x)}{x}} \mathrm{d}x,$$
	where 	$\mathcal {P}$ denotes the Cauchy principal value.
\end{lem} 
\begin{proof}
See Chapter 12 of \cite{muskhelishvili2010some}. 
\end{proof}

Finally, we provide the Wick's probability theorem which can be used to characterize the multivariate Gaussian distribution.  

\begin{lem}\label{lem_wicktheorem}
Suppose $(y_1, y_2, \cdots, y_k)^*$ follows mean zero multivariate Gaussian distribution $\mathcal{N}(0, \Lambda),$ and let $x_1, x_2,\cdots, x_n \in \{y_1, y_2, \cdots, y_k\}$ for all $n \in \mathbb{N},$ then 
\begin{equation*}
\mathbb{E}(x_1 x_2 \cdots x_n)=\sum_{\mathsf{p} \in \mathrm{P}_n^2} \prod_{\{i,j\} \in p} \mathbb{E} x_i x_j=\sum_{\mathsf{p} \in \mathrm{P}_n^2} \prod_{\{i,j\} \in \mathsf{p}} \operatorname{Cov}(x_i, x_j),
\end{equation*}
where $\mathrm{P}_n^2$ constains all distinct ways of partitioning $\{1,2,\cdots, n\}$ into pairs $\{i, j\}.$ Equivalently, by combining the terms that are the same, we have that for $r_1, r_2, \cdots, r_k \in \mathbb{N},$
\begin{equation*}
\mathbb{E}(y_1^{r_1} y_2^{r_2} \cdots y_k^{r_k})=(r_1-1) \Lambda_{11} \mathbb{E}(y_1^{r_1-2} y_2^{r_2} \cdots y_k^{r_k})+\sum_{i=2}^k r_i \Lambda_{1i}  \mathbb{E}(y_1^{r_1} y_2^{r_2} \cdots y_k^{r_k}/(y_1 y_i)). 
\end{equation*} 
Especially, when $k=2,$ we have that 
\begin{equation*}
\mathbb{E} y_i^s y_j^t =(s-1) \Lambda_{ii} \mathbb{E} y_i^{s-2} y_j^t+t \Lambda_{ij} \mathbb{E} y_i^{s-1} y_j^{t-1}.  
\end{equation*}
\end{lem}
\begin{proof}
See Section 1.2 of \cite{zinn2004path}. 
\end{proof}

\section{CLTs for the resolvents}\label{sec_CLTresolvent}

We first prepare some notations. Denote 
\begin{equation}\label{eq_mathcalY}
\mathcal{Y}(z_s)=(1-\mathbb{E})\operatorname{Tr}{{R_1}}(z_s), s=1,2,\ldots, l.
\end{equation}
Denote
$$\widehat{\alpha}\left(z_1, z_2\right):=\kappa_4 \phi \frac{\partial^2}{\partial z_1 \partial z_2}\left(\frac{1}{p} \sum_{i=1}^p \frac{1}{\left(1+\phi^{-1/2}{m}^{\mathfrak{a}}(z_1) \sigma_i\right)\left(1+\phi^{-1/2}{m}^{\mathfrak{b}}(z_2) \sigma_i\right)}\right),$$
	$$
	\widehat{\beta}\left(z_1, z_2\right):=2 \left(\frac{(m)'(z_1)(m)'(z_2)}{(m(z_1)-m(z_2))^2}-\frac{1}{(z_1-z_2)^2} \right).
	$$

\begin{thm}\label{them_CLTresolvent}
	Suppose $X$ and $\Sigma$ satisfy Assumptions  \ref{assum:XSigma}. Fix any $l \in \mathbb{N}$ and  $z_1, \ldots, z_l \in \mathbf{S}$, we have that  $\left(\eta_1\mathcal{Y}(z_1),\ldots,\eta_l\mathcal{Y}(z_l)\right) \simeq \left(\mathscr{G}(z_1),\ldots,\mathscr{G}(z_l)\right),$ where $\left(\mathscr{G}(z_1),\ldots,\mathscr{G}(z_l)\right)$ is a complex Gaussian vector   with mean $\mathbf{0}$ and covariances     
	\begin{equation*}
		\operatorname{Cov}\left(\mathscr{G}(z_i),\mathscr{G}(z_j)\right)=\eta_i\eta_j\left[\hat{\alpha}(z_i,z_j)+\hat{\beta}(z_i,z_j)\right],\quad 1\le i,j\le l.
	\end{equation*}
\end{thm}
The proof of Theorem \ref{them_CLTresolvent} replies on the following lemma. Denote 
\begin{equation*}
\mathsf{Re}:=\frac{1}{\prod_{i=1}^l|\eta_i|}\left(\frac{1}{n \min_{i=1}^l\{|\eta_i|\sqrt{\kappa_i+|\eta_i|}\}}+\frac{1}{n^{3/4}p^{1/4} \min_{i=1}^l\{\sqrt{|\eta_i|(\kappa_i+|\eta_i|)}\}}+\frac{1}{\sqrt{n}}\right).
\end{equation*}

\begin{lem}\label{lem:resolventwick}
Suppose $X$ and $\Sigma$ satisfy Assumptions  \ref{assum:XSigma}. Fix any $l \in \mathbb{N}$ and  $z_1, \ldots, z_l \in \mathbf{S},\ i=1,\ldots,l$ with imaginary part $\ge n^{-1+\varepsilon}$ for some positive $\varepsilon>0$, we have 
	\begin{equation}\label{eq_resolventwick}
		\mathbb{E}\left[\prod_{i=1}^l \mathcal{Y}\left(z_i\right)\right]=\left\{\begin{array}{ll}
			\sum \prod  \omega\left(z_s, z_t\right)+\mathrm{O}_{\prec}\left(\mathsf{Re}\right), & \text { if } l \in 2 \mathbb{N} \\
			\mathrm{O}_{\prec}\left(\mathsf{Re}\right), & \text { otherwise }
		\end{array} .\right.
	\end{equation}
	Here we denoted $\omega\left(z_s, z_t\right):=\widehat{\alpha}\left(z_s, z_t\right)+\widehat{\beta}\left(z_s, z_t\right)$, and $\sum \prod$ means summing over all distinct ways of partitioning of indices into pairs.
\end{lem}

\begin{proof}[\bf Proof of Theorem \ref{them_CLTresolvent}]
The proof follows directly from Lemmas \ref{lem:resolventwick}, \ref{lem_wicktheorem} and the Portmanteau theorem. 
\end{proof}

\subsection{Proof of Lemma \ref{lem:resolventwick}}

For simplicity, we focus on $l=2$ and then discuss how to handle general $l$ in the end of the proof. In this setting, according to Lemma \ref{lem_wicktheorem}, it suffices to prove (\ref{eq_wickl2case}). Furthermore, using (\ref{eq_mathcalY}) and the identity that for  any two random variables $a$ and $b,$ $\mathbb{E}[a(1-\mathbb{E})b]=\mathbb{E}[(1-\mathbb{E})a(1-\mathbb{E})b]=\mathbb{E}[b(1-\mathbb{E})a],$ we see it suffices to control  
$$\begin{aligned}
		&\mathbb{E}\left[((1-\mathbb{E})\operatorname{Tr}{R_1}(z_1))^{l_1}((1-\mathbb{E})\operatorname{Tr}{R_1}(z_2))^{l_2}\right]\\ =& \sum_{i=1}^p \mathbb{E}\left\{
		(1-\mathbb{E})[{((1-\mathbb{E})\operatorname{Tr}{R_1}(z_1))^{l_1}((1-\mathbb{E})\operatorname{Tr}{R_1}(z_2))^{l_2-1}}] ({R_1}(z_2))_{ii}	\right\}.\\ 
	\end{aligned}
	$$
	
Using the identity $zR_1(z)=R_1(z)Q-I$, Lemma \ref{lem_cumulantexpansion} and the convention $Y=\Sigma^{1/2}X$, we have that 
	\begin{equation}\label{eq_ithvariancerough}
		\begin{aligned}
			&z_2\mathbb{E}\left\{
			(1-\mathbb{E})[{((1-\mathbb{E})\operatorname{Tr}{R_1}(z_1))^{l_1}((1-\mathbb{E})\operatorname{Tr}{R_1}(z_2))^{l_2-1}}] ({R_1}(z_2))_{ii}	\right\}\\ 
			& =\mathbb{E}\left\{ (1-\mathbb{E})[{((1-\mathbb{E})\operatorname{Tr}{R_1}(z_1))^{l_1}((1-\mathbb{E})\operatorname{Tr}{R_1}(z_2))^{l_2-1}}]  ({{R_1}(z_2) Q})_{ii}\right\}\\ 
			& =\mathbb{E} \left\{(1-\mathbb{E})[{((1-\mathbb{E})\operatorname{Tr}{R_1}(z_1))^{l_1}((1-\mathbb{E})\operatorname{Tr}{R_1}(z_2))^{l_2-1}}]  ({{R_1}(z_2) \Sigma^{1/2}XX^*\Sigma^{1/2}})_{ii}\right\}\\ 
			& = \mathbb{E}\left\{(1-\mathbb{E})[{((1-\mathbb{E})\operatorname{Tr}{R_1}(z_1))^{l_1}((1-\mathbb{E})\operatorname{Tr}{R_1}(z_2))^{l_2-1}}] \sum_{j=1}^n \sqrt{\sigma_i}( {R_1}(z_2)Y)_{ij}x_{ij}\right\}\\ 
			& = \sum_{j=1}^n\frac{\sqrt{\sigma_i}}{\sqrt{pn}}  \mathbb{E}\left\{ \frac{\partial(1-\mathbb{E})[{((1-\mathbb{E})\operatorname{Tr}{R_1}(z_1))^{l_1}((1-\mathbb{E})\operatorname{Tr}{R_1}(z_2))^{l_2-1}}] ( {R_1}(z_2)Y)_{i j}}{\partial x_{ij}}\right\} +\mathsf{L}_i\\
		&	=  \sum_{j=1}^n\frac{\sqrt{\sigma_i}}{\sqrt{pn}} \mathbb{E}\left\{(1-\mathbb{E})[{((1-\mathbb{E})\operatorname{Tr}{R_1}(z_1))^{l_1}((1-\mathbb{E})\operatorname{Tr}{R_1}(z_2))^{l_2-1}}] \frac{\partial ( {R_1}(z_2)Y)_{i j}}{\partial x_{ij}}\right\} \\
			& +\sum_{j=1}^n\frac{\sqrt{\sigma_i}}{\sqrt{pn}}  \mathbb{E}\left\{ \frac{\partial((1-\mathbb{E})\operatorname{Tr}{R_1}(z_1))^{l_1}((1-\mathbb{E})\operatorname{Tr}{R_1}(z_2))^{l_2-1}}{\partial x_{ij}} ( {R_1}(z_2)Y)_{i j}\right\}+\mathsf{L}_{i} \\
		&	=  :\mathsf{A}_{i}+\mathsf{B}_{i}+\mathsf{L}_{i},
		\end{aligned}
	\end{equation}
	where	
	$$
	\mathsf{L}_{i}:= \sum_{j=1}^n\sqrt{\sigma_i}\left[\sum_{k=2}^{l_0} \frac{1}{k !} \kappa_{k+1}\mathbb{E} \frac{\partial^k\left([{((1-\mathbb{E})\operatorname{Tr}{R_1}(z_1))^{l_1}((1-\mathbb{E})\operatorname{Tr}{R_1}(z_2))^{l_2-1}}] ( {R_1}(z_2)Y)_{i j}\right)}{\partial x_{ij}^k}+\varepsilon_{l_0+1}^{(ij)}\right].
	$$
with $\varepsilon$'s being the remainder term. By the Lemma \ref{lem_cumulantexpansion} and the fact that $\kappa_{k+1}(x_{ij})\sim(np)^{-\frac{k+1}{4}}$, we choose $l_0=4$ so that $\sum_{j}\sqrt{\sigma_i}\varepsilon_{l_0+1}^{(ij)}=\mathrm{o}(p^{-1})$.

Before proceeding to the control of $\mathsf{A}_i, \mathsf{B}_i$ and $\mathsf{L}_i,$ we first prepare some useful identities. 
\begin{equation}\label{eq_resolventderi1}
	\frac{\partial\left({R_1}(z)\right)_{a b}}{\partial y_{ij }}=-\left({R_1}(z)\right)_{a i}\left(Y^* {R_1}(z)\right)_{j b}-\left({R_1}(z) Y\right)_{a j} ({R_1}(z))_{ i b}, \quad \frac{\partial\left({R_1}(z)\right)_{a b}}{\partial x_{ij}}=\frac{\partial\left({R_1}(z)\right)_{a b}}{\partial y_{ij}} \sqrt{\sigma_i}.
\end{equation}

\begin{equation}\label{eq_resolventderi2}
	\sum_{l=1}^p \frac{\partial ({R_1}(z))_{l l}}{\partial x_{ij}}=-2\sqrt{\sigma_i}\frac{\mathrm{d}}{\mathrm{d} z}({R_1}(z) Y)_{ij}.
\end{equation}

\begin{equation}\label{eq_resolventderi3}
	\begin{aligned}
		\frac{\partial ({R_1}Y)_{ij}}{\partial x_{ij}}=&(R_1\Sigma^{1/2}E_{ij})_{ij}-[R_1(\Sigma^{1/2}E_{ij}X^*\Sigma^{1/2}+\Sigma^{1/2}XE_{ji}\Sigma^{1/2})R_1\Sigma^{1/2}X]_{ij}\\ 
		=&\sqrt{\sigma_i}[({R_1})_{ii}-({R_1})_{ii}(Y^*R_1Y)_{jj}-(R_1Y)_{ij}(Y^*R_1)_{ij}],\\ 
	\end{aligned}
\end{equation}
where $E_{ij}$ is a $p \times n$ matrix whose only nonzero entry is one and in the ($i,j$) position.  

For $\mathsf{A}_i,$ using (\ref{eq_resolventderi1})--\ref{eq_resolventderi3}, we can write 
	$$
	\begin{aligned}
		\mathsf{A}_{i}=& 	z_2\mathbb{E}\left\{((1-\mathbb{E})\operatorname{Tr}{R_1}(z_1))^{l_1}((1-\mathbb{E})\operatorname{Tr}{R_1}(z_2))^{l_2-1}(1-\mathbb{E}){\left[({{R_1}(z_2)})_{i i}\right]}\right\}\\
		=& \sum_{ j}\frac{\sqrt{\sigma_i}}{\sqrt{pn}} \mathbb{E}\left\{(1-\mathbb{E})[{((1-\mathbb{E})\operatorname{Tr}{R_1}(z_1))^{l_1}((1-\mathbb{E})\operatorname{Tr}{R_1}(z_2))^{l_2-1}}] \frac{\partial ( {R_1}(z_2)Y)_{i j}}{\partial x_{ij}}\right\} \\
		=&  \sum_{ j}\frac{\sigma_i}{{\sqrt{pn}}} \mathbb{E}\left\{((1-\mathbb{E})\operatorname{Tr}{R_1}(z_1))^{l_1}((1-\mathbb{E})\operatorname{Tr}{R_1}(z_2))^{l_2-1}(1-\mathbb{E})[({{R_1}(z_2)})_{i i}\right.\\ &-\left.\left(Y^* {{R_1}(z_2)}Y\right)_{j j} ({{R_1}(z_2)})_{i i}-({{R_1}(z_2)}Y)_{i j}({{R_1}(z_2)}Y)_{i j}]\right\}\\ 
		:=&(\mathsf{A}_{i})_1+(\mathsf{A}_{i})_2+(\mathsf{A}_{i})_3	\end{aligned}
	$$
	where $(\mathsf{A}_i)_k, 1 \leq k \leq 3,$ are defined as 
\begin{equation*}
	\begin{aligned}
		(\mathsf{A}_{i})_1=& 	\sum_{j=1}^n\frac{\sigma_i}{{\sqrt{pn}}} \mathbb{E}\left\{((1-\mathbb{E})\operatorname{Tr}{R_1}(z_1))^{l_1}((1-\mathbb{E})\operatorname{Tr}{R_1}(z_2))^{l_2-1}(1-\mathbb{E}){\left[({{R_1}(z_2)})_{i i}\right]}\right\}\\ 
		(\mathsf{A}_{i})_2=& -\sum_{j=1}^n\frac{\sigma_i}{{\sqrt{pn}}} \mathbb{E}\left\{((1-\mathbb{E})\operatorname{Tr}{R_1}(z_1))^{l_1}((1-\mathbb{E})\operatorname{Tr}{R_1}(z_2))^{l_2-1}(1-\mathbb{E})[\left(Y^* {{R_1}(z_2)}Y\right)_{j j} ({{R_1}(z_2)})_{i i}]\right\}\\ 
		(\mathsf{A}_{i})_3=&-\sum_{j=1}^n\frac{\sigma_i}{{\sqrt{pn}}} \mathbb{E}\left\{((1-\mathbb{E})\operatorname{Tr}{R_1}(z_1))^{l_1}((1-\mathbb{E})\operatorname{Tr}{R_1}(z_2))^{l_2-1}(1-\mathbb{E})[({{R_1}(z_2)}Y)_{i j}({{R_1}(z_2)}Y)_{i j}]\right\}\\ 
	\end{aligned}
\end{equation*}
Note for $	(\mathsf{A}_{i})_1=\phi^{-1/2}\sigma_i\mathbb{E}\left\{((1-\mathbb{E})\operatorname{Tr}{R_1}(z_1))^{l_1}((1-\mathbb{E})\operatorname{Tr}{R_1}(z_2))^{l_2-1}(1-\mathbb{E}){\left[({{R_1}(z_2)})_{i i}\right]}\right\}$, since it has the same structure in definition of $\mathsf{A}_{i}$, we will merge it in the end. For $(\mathsf{A}_{i})_2$, using Theorem \ref{thm_locallaw}, we have that 
\begin{equation*}
	\begin{aligned}
		(\mathsf{A}_{i})_2=& -\sum_{j=1}^n\frac{\sigma_i}{{\sqrt{pn}}} \mathbb{E}\left\{((1-\mathbb{E})\operatorname{Tr}{R_1}(z_1))^{l_1}((1-\mathbb{E})\operatorname{Tr}{R_1}(z_2))^{l_2-1}(1-\mathbb{E})[\left(Y^* {{R_1}(z_2)}Y\right)_{j j} ({{R_1}(z_2)})_{i i}]\right\}\\ 
		=& -\sum_{j=1}^n\frac{\sigma_i}{{\sqrt{pn}}} \mathbb{E}\left\{((1-\mathbb{E})\operatorname{Tr}{R_1}(z_1))^{l_1}((1-\mathbb{E})\operatorname{Tr}{R_1}(z_2))^{l_2-1}(1-\mathbb{E})[\operatorname{Tr}({R_1}(z_2)YY^*) ( {R_1}(z_2))_{i i}]\right\}\\
		=&-\sigma_i\phi^{1/2}  \mathbb{E}\left[((1-\mathbb{E})\operatorname{Tr}{R_1}(z_1))^{l_1}((1-\mathbb{E})\operatorname{Tr}{R_1}(z_2))^{l_2-1}(1-\mathbb{E})\left(z_2 p^{-1}({R_1}(z_2))_{i i}\right) \operatorname{Tr} {R_1}(z_2) \right]\\ &-\sigma_i\phi^{1/2} \mathbb{E}\left[((1-\mathbb{E})\operatorname{Tr}{R_1}(z_1))^{l_1}((1-\mathbb{E})\operatorname{Tr}{R_1}(z_2))^{l_2-1}(1-\mathbb{E})\left[ ({R_1}(z_2))_{i i}\right]\right] \\
		=&\sigma_i\phi^{1/2}\left(- z_2 m_{p}(z_2)-1\right) \mathbb{E}\left[((1-\mathbb{E})\operatorname{Tr}{R_1}(z_1))^{l_1}((1-\mathbb{E})\operatorname{Tr}{R_1}(z_2))^{l_2-1}(1-\mathbb{E})\left(({R_1}(z_2))_{i i}\right)\right]\\ 
		&+\sigma_i\phi^{1/2}/(1+\phi^{-1/2}m(z_2)\sigma_i) \mathbb{E}\left[((1-\mathbb{E})\operatorname{Tr}{R_1}(z_1))^{l_1}((1-\mathbb{E})\operatorname{Tr}{R_1}(z_2))^{l_2-1}(1-\mathbb{E})\left(p^{-1} \operatorname{Tr} {R_1}(z_2)\right)\right]\\ 
		&+\mathrm{O}_{\prec}\left(\frac{1}{|\eta_1|^{l_1}|\eta_2|^{l_2-1}}\frac{\sqrt{\phi}}{p|\eta_2|}\Psi(z_2)\right).
	\end{aligned}
\end{equation*}
Similarly, for $(\mathsf{A}_{i})_3$, using Theorem \ref{thm_locallaw}, we see that  		
\begin{equation*}
	\begin{aligned}
		(\mathsf{A}_{i})_3&=-\sum_{j=1}^n\frac{\sigma_i}{{\sqrt{pn}}} \mathbb{E}\left\{((1-\mathbb{E})\operatorname{Tr}{R_1}(z_1))^{l_1}((1-\mathbb{E})\operatorname{Tr}{R_1}(z_2))^{l_2-1}(1-\mathbb{E})[({{R_1}(z_2)}Y)_{i j}({{R_1}(z_2)}Y)_{i j}]\right\}\\ 
		&=-\frac{\sigma_i}{{\sqrt{pn}}} \mathbb{E}\left[((1-\mathbb{E})\operatorname{Tr}{R_1}(z_2))^{l_1}((1-\mathbb{E})\operatorname{Tr}{R_1}(z_2))^{l_2-1}(1-\mathbb{E}) \left(R_1(z_2)YY^*R_1(z_2)\right)_{i i} \right] \\
		&=\mathrm{O}_{\prec}\left(\frac{1}{|\eta_1|^{l_1}|\eta_2|^{l_2-1}}\frac{\sqrt{\phi}}{p|\eta_2|}\Psi(z_2)\right).
	\end{aligned}
\end{equation*}	
Plugging the above estimates back into \eqref{eq_ithvariancerough}, we get
	\begin{equation*}
	\begin{aligned}
		&z_2\mathbb{E}\left\{
		(1-\mathbb{E})[{((1-\mathbb{E})\operatorname{Tr}{R_1}(z_1))^{l_1}((1-\mathbb{E})\operatorname{Tr}{R_1}(z_2))^{l_2-1}}] ({R_1}(z_2))_{ii}	\right\}\\ 	
		& =  (\mathsf{A}_{i})_1+(\mathsf{A}_{i})_2+(\mathsf{A}_{i})_3+\mathsf{B}_{i}+\mathsf{L}_{i}\\ 
	&	=(\phi^{-1/2}\sigma_i-z_2\phi^{-1/2}\sigma_im_p(z_2)-\phi^{1/2}\sigma_i)\mathbb{E}\left\{((1-\mathbb{E})\operatorname{Tr}{R_1}(z_1))^{l_1}((1-\mathbb{E})\operatorname{Tr}{R_1}(z_2))^{l_2-1}({{R_1}(z_2)})_{ii}\right\}\\ 
	&+\frac{\phi^{1/2}\sigma_i}{p(1+\phi^{-1/2}m(z_2)\sigma_i)}\mathbb{E}\left\{((1-\mathbb{E})\operatorname{Tr}{R_1}(z_1))^{l_1}((1-\mathbb{E})\operatorname{Tr}{R_1}(z_2))^{l_2}\right\}\\ 
		&+\mathrm{O}_{\prec}\left(\frac{1}{|\eta_1|^{l_1}|\eta_2|^{l_2-1}}\frac{\sqrt{\phi}}{p|\eta_2|}\Psi(z_2)\right)+\mathsf{B}_{i}+\mathsf{L}_{i}.\\ 
	\end{aligned}
\end{equation*}
With straightforward calculations, we can further obtain
	\begin{equation}\label{eq_variancerough}
	\begin{aligned}
				&\mathbb{E}
		[{((1-\mathbb{E})\operatorname{Tr}{R_1}(z_1))^{l_1}((1-\mathbb{E})\operatorname{Tr}{R_1}(z_2))^{l_2}}]\\ 	
	&	=
		\mathbb{E}\left\{
		(1-\mathbb{E})[{((1-\mathbb{E})\operatorname{Tr}{R_1}(z_1))^{l_1}((1-\mathbb{E})\operatorname{Tr}{R_1}(z_2))^{l_2-1}}] \operatorname{Tr}({R_1}(z_2))	\right\}\\ 	
	&	= \sum_{i=1}^p\frac{1}{z_2(1+\sigma_i\phi^{-1/2}m(z_2))} \left[\frac{\phi^{1/2}\sigma_i}{p(1+\phi^{-1/2}m(z_2)\sigma_i)}\mathbb{E}\left\{((1-\mathbb{E})\operatorname{Tr}{R_1}(z_1))^{l_1}((1-\mathbb{E})\operatorname{Tr}{R_1}(z_2))^{l_2}\right\}\right.\\
		&\left.+\mathrm{O}_{\prec}\left(\frac{1}{|\eta_1|^{l_1}|\eta_2|^{l_2-1}}\frac{\sqrt{\phi}}{p|\eta_2|}\Psi(z_2)\right)+\mathsf{B}_{i}+\mathsf{L}_{i}\right].\\ 
	\end{aligned}
\end{equation}

Then we proceed to the calculation of $\mathsf{B}_i.$ Again using (\ref{eq_resolventderi1})--(\ref{eq_resolventderi3}), we have that 
\begin{equation}\label{eq_Bitwosplit}
	\begin{aligned}
		\mathsf{B}_{i}=& \frac{\sqrt{\sigma_i}}{\sqrt{np}}\sum_{j=1}^n\mathbb{E}\left\{ \frac{\partial\left[(1-\mathbb{E}){((1-\mathbb{E})\operatorname{Tr}{R_1}(z_1))^{l_1}((1-\mathbb{E})\operatorname{Tr}{R_1}(z_2))^{l_2-1}}\right]}{\partial x_{ij}} ({{R_1}(z_2)}Y)_{i j}\right\}\\ 
		=&-2\frac{\sigma_i}{\sqrt{np}}\sum_{j=1}^n\mathbb{E}\left[ l_1((1-\mathbb{E})\operatorname{Tr}{R_1}(z_1))^{l_1-1}((1-\mathbb{E})\operatorname{Tr}{R_1}(z_2))^{l_2-1}( {R_1}(z_2)Y)_{i j}\frac{\mathrm{d}}{\mathrm{d} z_1}( {R_1}(z_1)Y)_{ij} \right.\\
		&\left. + (l_2-1)((1-\mathbb{E})\operatorname{Tr}{R_1}(z_1))^{l_1}((1-\mathbb{E})\operatorname{Tr}{R_1}(z_2))^{l_2-2}( {R_1}(z_2)Y)_{i j}\frac{\mathrm{d}}{\mathrm{d} z_2}( {R_1}(z_2)Y)_{ij}\right] \\
		=&-2\frac{\sigma_i}{\sqrt{np}}\mathbb{E}\left[ l_1((1-\mathbb{E})\operatorname{Tr}{R_1}(z_1))^{l_1-1}((1-\mathbb{E})\operatorname{Tr}{R_1}(z_2))^{l_2-1}\left( \frac{\partial}{\partial z_1}({R_1}(z_1)YY^* {{R_1}(z_2)})\right)_{ii} \right.\\
		&\left. + (l_2-1)((1-\mathbb{E})\operatorname{Tr}{R_1}(z_1))^{l_1}((1-\mathbb{E})\operatorname{Tr}{R_1}(z_2))^{l_2-2}\left( \frac{\partial}{\partial z_2}\frac{1}{2}({R_1}(z_2)YY^* {{R_1}(z_2)})\right)_{ii} \right]  \\  
		&:=(	\mathsf{B}_{i})_1+(	\mathsf{B}_{i})_2.
	\end{aligned}
\end{equation}
For the first term, we have that 
\begin{equation*}
	\begin{aligned}
		(	\mathsf{B}_{i})_1&=-\frac{2\sigma_i}{\sqrt{np}}\mathbb{E}\left[ l_1((1-\mathbb{E})\operatorname{Tr}{R_1}(z_1))^{l_1-1}((1-\mathbb{E})\operatorname{Tr}{R_1}(z_2))^{l_2-1}\left( \frac{\partial}{\partial z_1}({R_1}(z_1)YY^* {{R_1}(z_2)})\right)_{ii} \right]\\
		&=-\frac{2\sigma_i}{{\sqrt{pn}}} \mathbb{E}\left[l_1((1-\mathbb{E})\operatorname{Tr}{R_1}(z_1))^{l_1-1}((1-\mathbb{E})\operatorname{Tr}{R_1}(z_2))^{l_2-1} \frac{\partial}{\partial z_1}\left( {R_1}\left(z_1\right) Q {R_1}\left(z_2\right)\right)_{i i} \right] \\
		&=-\frac{2\sigma_i}{{\sqrt{pn}}} \mathbb{E}\left[l_1((1-\mathbb{E})\operatorname{Tr}{R_1}(z_1))^{l_1-1}((1-\mathbb{E})\operatorname{Tr}{R_1}(z_2))^{l_2-1} \frac{\partial}{\partial z_1}\left(z_1  {R_1}\left(z_1\right) {R_1}\left(z_2\right)\right)_{i i} \right]\\ 
		&=-\frac{2\sigma_i}{{\sqrt{pn}}} \mathbb{E}\left[l_1((1-\mathbb{E})\operatorname{Tr}{R_1}(z_1))^{l_1-1}((1-\mathbb{E})\operatorname{Tr}{R_1}(z_2))^{l_2-1} \frac{\partial}{\partial z_1}\left(z_1\frac{{R_1}\left(z_1\right)-{R_1}\left(z_2\right)}{z_1-z_2}\right)_{i i} \right]\\ 
		&=-\frac{2}{{\sqrt{pn}}} \mathbb{E}\left[l_1((1-\mathbb{E})\operatorname{Tr}{R_1}(z_1))^{l_1-1}((1-\mathbb{E})\operatorname{Tr}{R_1}(z_2))^{l_2-1} \frac{\partial}{\partial z_1}\left( \frac{z_1 \sigma_i(\mathfrak{g}_i(z_1)-\mathfrak{g_i}(z_2))}{z_1-z_2}\right) \right]+\mathcal{E}_{B,1}(i),\\ 
	\end{aligned}
\end{equation*}	
where  $$\begin{aligned}
	\mathcal{E}_{B,1}(i):=-\frac{2\sigma_i}{{\sqrt{pn}}} \mathbb{E}\left[l_1((1-\mathbb{E})\operatorname{Tr}{R_1}(z_1))^{l_1-1}((1-\mathbb{E})\operatorname{Tr}{R_1}(z_2))^{l_2-1} \frac{\partial}{\partial z_1}\left(z_1\frac{({R_1}\left(z_1\right))_{ii}-\mathfrak{g}_i(z_1)-({R_1}\left(z_2\right))_{ii}+\mathfrak{g}_i(z_2)}{z_1-z_2}\right)\right],
\end{aligned}
 $$
 \begin{equation}\label{eq_gidefinition}
 \mathfrak{g}_i(z):=-\frac{1}{z(1+\phi^{-1/2}m(z)\sigma_i)}.
\end{equation} 
  From \eqref{eq_variancerough} we can see that $\sum_{i=1}^p\mathfrak{g}_i(z_2)\mathcal{E}_{B,1}(i)$ is the error term we need to control. Let $\eta_i=\operatorname{Im} z_i, \ i=1,2.$ On the one hand, $|z_1-z_2|\succ |\eta_2|$, using the fact  that $z_1\asymp \sqrt{\phi}$, $\mathfrak{g}_i(z_2)\asymp \phi^{-1/2}$ and  $\frac{\partial}{\partial z_1}\left( \cdot\right) \asymp |\eta_1|^{-1}$ , we can  use the fluctuation averaging argument as in the proof of (\ref{eq_averagelocallawtwo}) (c.f. Lemma \ref{lem:fa}) to obtain that
$$\sum_{i}\mathfrak{g}_i(z_2)\sigma_i((R_1)_{ii}(z)-\mathfrak{g}_i(z))=\mathrm{O}_{\prec}\left(\frac{\phi^{-1/2}}{|\operatorname{Im}(z)|} \right).$$ 
This further results in 
$$\sum_{i=1}^p\mathfrak{g}_i(z_2)\mathcal{E}_{B,1}(i)=\mathrm{O}_{\prec}\left(\frac{1}{\sqrt{pn}}\frac{1}{|\eta_1|^{l_1-1}|\eta_2|^{l_2-1}}\frac{1}{|\eta_1|}\frac{1}{|\eta_2|}\left(\frac{1}{|\eta_1|}+\frac{1}{|\eta_2|}\right)\right)=\mathrm{O}_{\prec}\left(\frac{1}{|\eta_1|^{l_1}|\eta_2|^{l_2}}\frac{1}{\sqrt{pn}}\left(\frac{1}{|\eta_1|}+\frac{1}{|\eta_2|}\right)\right).$$ 

On the other hand,  when $z_1$ and $z_2$ are close, without loss of generality, we assume that $\operatorname{Im}z_1\succ |\eta_2|$. Denote the contour $\Gamma=\partial \mathsf{B}_{c|\eta_2|}(z_1)\cup\partial \mathsf{B}_{c|\eta_2|}(z_2)$ for some constant $c>0$, where choose $c>0$ small enough so that $\Gamma\subset \mathbb{C}_+$ and $\min_{\xi\in\Gamma}\operatorname{Im}\xi\succ|\eta_2|$. Using Theorem \ref{thm_locallaw}, we have that
\begin{equation*}
	\begin{aligned}
		\sum_{i=1}^p\sigma_i\mathfrak{g}_i( R_1(z_1)R_1(z_2))_{ii}
		=&\sum_{i=1}^p\sigma_i\frac{(R_1(z_1)-R_1(z_2))_{ii}}{(z_1-z_2)}=\frac{1}{2 \pi\mathrm{i}}\sum_{i=1}^p\sigma_i  \int_{\Gamma}\frac{(R_1(\xi))_{ii}}{(\xi-z_1)(\xi-z_2)}\mathrm{~d}\xi\\ 
		=&\frac{1}{2 \pi\mathrm{i}}\int_{\Gamma}\frac{\sum_{i=1}^p\sigma_i\left(\Pi(\xi)\right)_{ii}+\mathrm{O}_{\prec}\left(\frac{1}{|\eta_2|}\right)}{(\xi-z_1)(\xi-z_2)}\mathrm{~d}\xi 
		=\frac{\sum_{i=1}^p\sigma_i(\mathfrak{g}_i(z_1)-\mathfrak{g_i}(z_2))}{z_1-z_2}+\mathrm{O}_{\prec}\left(\frac{1}{|\eta_2|^2}\right),
	\end{aligned}
\end{equation*}
which implies that $\sum_{i=1}^p\mathfrak{g}_i(z)\mathcal{E}_{B,1}(i)=\mathrm{O}_{\prec}\left(\frac{1}{|\eta_1|^{l_1-1}|\eta_2|^{l_2-1}}\frac{1}{|\eta_1|}\frac{1}{\sqrt{pn}(|\eta_2|)^2}\right)=\mathrm{O}_{\prec}\left(\frac{1}{|\eta_1|^{l_1}|\eta_2|^{l_2}}\frac{1}{\sqrt{pn}|\eta_2|}\right).$ In summary, we always have that	
\begin{equation}\label{eq_errorgiB1i}	
	\sum_{i=1}^p\mathfrak{g}_i(z_2)\mathcal{E}_{B,1}(i)=\mathrm{O}_{\prec}\left(\frac{1}{|\eta_1|^{l_1}|\eta_2|^{l_2}}\frac{1}{\sqrt{pn}}\left(\frac{1}{|\eta_1|}+\frac{1}{|\eta_2|}\right)\right).
\end{equation}

Similarly, we can handle	 $(\mathsf{B}_{i})_2$ in (\ref{eq_Bitwosplit}) as follows
\begin{equation*}
	\begin{aligned}
		(\mathsf{B}_{i})_2=&-\frac{2\sigma_i}{\sqrt{np}}\mathbb{E}\left[ (l_2-1)((1-\mathbb{E})\operatorname{Tr}{R_1}(z_1))^{l_1}((1-\mathbb{E})\operatorname{Tr}{R_1}(z_2))^{l_2-2}\left( \frac{\partial}{\partial z_2}\frac{1}{2}({R_1}(z_2)YY^* {{R_1}(z_2)})\right)_{ii} \right] \\  	
		=&-\frac{2\sigma_i}{\sqrt{np}}\mathbb{E}\left[ (l_2-1)((1-\mathbb{E})\operatorname{Tr}{R_1}(z_1))^{l_1}((1-\mathbb{E})\operatorname{Tr}{R_1}(z_2))^{l_2-2}\left( \frac{1}{2}\frac{\partial}{\partial z_2}({R_1}(z_2)+z_2({R_1}(z_2))^2)\right)_{ii} \right] \\  	
		=&-\frac{2\sigma_i}{\sqrt{np}}\mathbb{E}\left[ (l_2-1)((1-\mathbb{E})\operatorname{Tr}{R_1}(z_1))^{l_1}((1-\mathbb{E})\operatorname{Tr}{R_1}(z_2))^{l_2-2} \frac{1}{2}\frac{\partial}{\partial z_2}\left((\mathfrak{g}_i(z_2)+z_2\mathfrak{g}_i^{\prime}(z_2))\right) \right] +\mathcal{E}_{B,2}(i),
	\end{aligned}
\end{equation*}	
where  $$\begin{aligned}
	\mathcal{E}_{B,2}(i):=-\frac{2\sigma_i}{{\sqrt{pn}}} \mathbb{E}\left[l_1((1-\mathbb{E})\operatorname{Tr}{R_1}(z_1))^{l_1}((1-\mathbb{E})\operatorname{Tr}{R_1}(z_2))^{l_2-1} \frac{1}{2}\frac{\partial}{\partial z_2}\Big(({R_1}\left(z_2\right))_{ii}-\mathfrak{g}_i(z_2)+z_2(({R_1}\left(z_2\right))_{ii}-\mathfrak{g}_i(z_2))^{\prime}\Big)\right].
\end{aligned}
$$
For the error term, by a discussion similar to (\ref{eq_errorgiB1i}), we have that  
\begin{equation}\label{eq_errorgiB2i}	
	\sum_{i=1}^p\mathfrak{g}_i(z_2)\mathcal{E}_{B,2}(i)=\mathrm{O}_{\prec}\left(\frac{1}{|\eta_1|^{l_1}|\eta_2|^{l_2-2}}\frac{1}{\sqrt{pn}}\frac{1}{|\eta_2|}\frac{1}{|\eta_2|^2}\right)=\mathrm{O}_{\prec}\left(\frac{1}{|\eta_1|^{l_1}|\eta_2|^{l_2}}\frac{1}{\sqrt{pn}|\eta_2|}\right). 
\end{equation} 
This completes the analysis of $\mathsf{B}_i$ in view of (\ref{eq_Bitwosplit}). For $\mathsf{L}_i,$ we summarize the results in the following lemma and put its proof into next subsection. 

\begin{lem}\label{lem_controlofLi}
For $\mathsf{L}_i$ defined in (\ref{eq_ithvariancerough}), we have that 
	\begin{align*}
		&\sum_{i=1}^p\frac{1}{z_2(1+\sigma_i\phi^{-1/2}m(z_2))} \mathsf{L}_{i}\\
			=&\sum_{i=1}^p\frac{1}{z_2(1+\sigma_i\phi^{-1/2}m(z_2))}
		\frac{-\kappa_4}{p}\\ 
		&\mathbb{E}\left\{l_1((1-\mathbb{E})\operatorname{Tr}{R_1}(z_1))^{l_1-1}((1-\mathbb{E})\operatorname{Tr}{R_1}(z_2))^{l_2-1}\frac{\partial}{\partial z_1}\left(\frac{	-\sigma_im(z_1)}{1+\phi^{-1/2}m(z_1) \sigma_i}\frac{	-\sigma_im(z_2)}{1+\phi^{-1/2}m(z_2) \sigma_i}\right)\right.\\ 
		&+\left. (l_2-1)((1-\mathbb{E})\operatorname{Tr}{R_1}(z_1))^{l_1}((1-\mathbb{E})\operatorname{Tr}{R_1}(z_2))^{l_2-2}\frac{1}{2}\frac{\partial}{\partial z_2}\left(\frac{	-\sigma_im(z_2)}{1+\phi^{-1/2}m(z_2) \sigma_i}\right)^2\right\}\\
		&+\mathrm{O}_{\prec}\left( \frac{1}{|\eta_1|^{l_1}|\eta_2|^{l_2}}\left(\frac{n^{1/4}}{p^{3/4}}\Psi(z_2)+\frac{1}{p|\eta_1|}+\frac{1}{p|\eta_2|}+\phi^{-1/2}\frac{\sqrt{|\eta_2|}}{\sqrt{n}}\right) \right). 
	\end{align*}
\end{lem}

Inserting Lemma \ref{lem_controlofLi}, (\ref{eq_Bitwosplit}),  \eqref{eq_errorgiB1i} and \eqref{eq_errorgiB2i} into \eqref{eq_variancerough}, we readily obtain that 
\begin{align}\label{eq_mid_explainerror}
		&\mathbb{E}\left[((1-\mathbb{E})\operatorname{Tr}{R_1}(z_1))^{l_1}((1-\mathbb{E})\operatorname{Tr}{R_1}(z_2))^{l_2}\right] \nonumber \\
		=&\sum_{i=1}^p\frac{1}{z_2(1+\sigma_i\phi^{-1/2}m(z_2))}\left\{\frac{\phi^{1/2}\sigma_i}{p(1+\phi^{-1/2}m(z_2)\sigma_i)}\mathbb{E}\left\{((1-\mathbb{E})\operatorname{Tr}{R_1}(z_1))^{l_1}((1-\mathbb{E})\operatorname{Tr}{R_1}(z_2))^{l_2}\right\}\right. \nonumber\\ 
		& -\frac{2}{{\sqrt{pn}}} \mathbb{E}\left[l_1((1-\mathbb{E})\operatorname{Tr}{R_1}(z_1))^{l_1-1}((1-\mathbb{E})\operatorname{Tr}{R_1}(z_2))^{l_2-1} \frac{\partial}{\partial z_1}\left( \frac{z_1 \sigma_i(\mathfrak{g}_i(z_1)-\mathfrak{g_i}(z_2))}{z_1-z_2}\right) \right] \nonumber \\ 
		&-\frac{2}{\sqrt{pn}}\mathbb{E}\left[ (l_2-1)((1-\mathbb{E})\operatorname{Tr}{R_1}(z_1))^{l_1}((1-\mathbb{E})\operatorname{Tr}{R_1}(z_2))^{l_2-2} \frac{1}{2}\frac{\partial}{\partial z_2}\left((\sigma_i\mathfrak{g}_i(z_2)+\sigma_iz_2\mathfrak{g}_i^{\prime}(z_2))\right) \right] \nonumber \\ 
		&-\frac{\kappa_4}{p}\mathbb{E}\left[l_1((1-\mathbb{E})\operatorname{Tr}{R_1}(z_1))^{l_1-1}((1-\mathbb{E})\operatorname{Tr}{R_1}(z_2))^{l_2-1}\frac{\partial}{\partial z_1}\left(\frac{	-\sigma_im(z_1)}{1+\phi^{-1/2}m(z_1) \sigma_i}\frac{	-\sigma_im(z_2)}{1+\phi^{-1/2}m(z_2) \sigma_i}\right)\right. \nonumber\\ 
		&+\left.\left. (l_2-1)((1-\mathbb{E})\operatorname{Tr}{R_1}(z_1))^{l_1}((1-\mathbb{E})\operatorname{Tr}{R_1}(z_2))^{l_2-2}\frac{1}{2}\frac{\partial}{\partial z_2}\left(\frac{	-\sigma_im(z_2)}{1+\phi^{-1/2}m(z_2) \sigma_i}\right)^2\right] \right\} \\
		&+\mathrm{O}\left(\frac{1}{|\eta_1|^{l_1}|\eta_2|^{l_2}}\left(\frac{1}{\sqrt{np}|\eta_1|}+\frac{1}{\sqrt{np}|\eta_2|} \right)\right)+\mathrm{O}_{\prec}\left( \frac{1}{|\eta_1|^{l_1}|\eta_2|^{l_2}}\left(\frac{n^{1/4}}{p^{3/4}}\Psi(z_2)+\frac{1}{p|\eta_1|}+\frac{1}{p|\eta_2|}+\phi^{-1/2}\frac{\sqrt{|\eta_2|}}{\sqrt{n}}\right) \right). \nonumber
	\end{align}	
We can further use Lemma \ref{lem:mphizabs} to simplify \eqref{eq_mid_explainerror} and get 
	\begin{align}\label{eq_wickrough}
		&\mathbb{E}\left[((1-\mathbb{E})\operatorname{Tr}{R_1}(z_1))^{l_1}((1-\mathbb{E})\operatorname{Tr}{R_1}(z_2))^{l_2}\right] \nonumber \\
		=&\left(1-\frac{1}{p} \sum_{i=1}^p \frac{\phi^{1/2} \sigma_i}{z_2\left(1+\phi^{-1/2}{m}(z_2) \sigma_i\right)^2}\right)^{-1}\sum_{i=1}^p\frac{1}{z_2(1+\sigma_i\phi^{-1/2}m(z_2))} \nonumber \\ 
		&\left\{ -\frac{2}{{\sqrt{pn}}} \mathbb{E}\left[l_1((1-\mathbb{E})\operatorname{Tr}{R_1}(z_1))^{l_1-1}((1-\mathbb{E})\operatorname{Tr}{R_1}(z_2))^{l_2-1} \frac{\partial}{\partial z_1}\left( \frac{z_1 \sigma_i(\mathfrak{g}_i(z_1)-\mathfrak{g_i}(z_2))}{z_1-z_2}\right) \right]\right. \nonumber \\ 
		&\left.-\frac{2}{\sqrt{np}}\mathbb{E}\left[ (l_2-1)((1-\mathbb{E})\operatorname{Tr}{R_1}(z_1))^{l_1}((1-\mathbb{E})\operatorname{Tr}{R_1}(z_2))^{l_2-2} \frac{1}{2}\frac{\partial}{\partial z_2}\left((\sigma_i\mathfrak{g}_i(z_2)+\sigma_iz_2\mathfrak{g}_i^{\prime}(z_2))\right) \right] \right. \nonumber \\ 
		&\left.-\frac{\kappa_4}{p}\mathbb{E}\left[l_1((1-\mathbb{E})\operatorname{Tr}{R_1}(z_1))^{l_1-1}((1-\mathbb{E})\operatorname{Tr}{R_1}(z_2))^{l_2-1}\frac{\partial}{\partial z_1}\left(\frac{	-\sigma_im(z_1)}{1+\phi^{-1/2}m(z_1) \sigma_i}\frac{	-\sigma_im(z_2)}{1+\phi^{-1/2}m(z_2) \sigma_i}\right)\right.\right. \nonumber \\ 
		&\left.+\left. (l_2-1)((1-\mathbb{E})\operatorname{Tr}{R_1}(z_1))^{l_1}((1-\mathbb{E})\operatorname{Tr}{R_1}(z_2))^{l_2-2}\frac{1}{2}\frac{\partial}{\partial z_2}\left(\frac{	-\sigma_im(z_2)}{1+\phi^{-1/2}m(z_2) \sigma_i}\right)^2\right]\right\} \nonumber \\
		&+\mathrm{O}\left(\frac{1}{|\eta_1|^{l_1}|\eta_2|^{l_2}}\left(\frac{1}{n|\eta_1|}+\frac{1}{n|\eta_2|} \right)\frac{1}{\sqrt{\kappa_2+\eta_2}}\right)+\mathrm{O}_{\prec}\left(\frac{1}{n^{3/4}p^{1/4} \min_{i=1}^l\{\sqrt{|\eta_i|(\kappa_i+|\eta_i|)}\}}\right) \nonumber \\
		&+\mathrm{O}_{\prec}\left(\frac{1}{|\eta_1|^{l_1}|\eta_2|^{l_2}}\frac{1}{\sqrt{n}} \right). 
	\end{align}

The above equation is already in the form of (\ref{eq_wickl2case}).  Therefore, by Lemma \ref{lem_wicktheorem}, the Gaussianity of the resolvents follows. The rest of this section is devoted to 
 more explicit formulas for the covariance as in Lemma \ref{lem_wicktheorem}.  Let $J_1$ and $J_3$ denote the first and the third terms on the right hand side of \eqref{eq_wickrough}. In order to obtain $\omega(z_1,z_2),$ in view of (\ref{eq_wickl2case}), we only need to  calculate $J_1$ and $J_3$ explicitly. Using Lemma \ref{lem_elementaryidentityes}, we have that 
\begin{equation*}
	\begin{aligned}
		z_1 z_2 \frac{1}{p} \sum_{i=1}^p \sigma_i \mathfrak{g}_i\left(z_1\right) \mathfrak{g}_i\left(z_2\right) & =\frac{1}{p} \sum_{i=1}^n \frac{\sigma_i}{\left(1+\phi^{-1/2}{m}(z_1) \sigma_i\right)\left(1+\phi^{-1/2}{m}(z_2) \sigma_i\right)}=\frac{z_1 {m}(z_1)-z_2 {m}(z_2)}{\phi^{1/2}\left({m}(z_1)-{m}(z_2)\right)}  \\
		z_1^2 \frac{1}{p} \sum_{i=1}^p \sigma_i \mathfrak{g}_i^2\left(z_1\right) & =\frac{1}{p} \sum_{i=1}^n \frac{\sigma_i}{\left(1+\phi^{-1/2}{m}_1 \sigma_i\right)^2}=\frac{\left(z_1 {m}_1\right)^{\prime}}{\phi^{1/2} {m}_1^{\prime}}.
	\end{aligned}
\end{equation*}
Combining with the definitions of $J_1$ and $J_3,$ we see that 
\begin{equation*}
	\begin{aligned}
		J_1 =& -\frac{z_2 {m}^{\prime}(z_2)}{{m}(z_2)}\frac{2p}{{\sqrt{pn}}} \mathbb{E}\left[l_1((1-\mathbb{E})\operatorname{Tr}{R_1}(z_1))^{l_1-1}((1-\mathbb{E})\operatorname{Tr}{R_1}(z_2))^{l_2-1} \frac{\partial}{\partial z_1}\frac{1}{p}\sum_{i=1}^p\left(\mathfrak{g}_i(z_2) \frac{z_1 \sigma_i(\mathfrak{g}_i(z_1)-\mathfrak{g_i}(z_2))}{z_1-z_2}\right) \right] \\ 
		=&l_1\mathbb{E}\left[((1-\mathbb{E})\operatorname{Tr}{R_1}(z_1))^{l_1-1}((1-\mathbb{E})\operatorname{Tr}{R_1}(z_2))^{l_2-1} \right]  2\left(\frac{{m}^{\prime}(z_1) {m}^{\prime}(z_2)}{\left({m}(z_1)-{m}(z_2)\right)^2}-\frac{1}{\left(z_1-z_2\right)^2}\right),
	\end{aligned}
\end{equation*}
\begin{equation*}
	\begin{aligned}
		J_3= & \frac{z_2 {m}(z_2)^{\prime}}{{m}(z_2)} \sum_{i=1}^p \frac{1}{z_2\left(1+\phi^{-1/2}{m}(z_2) \sigma_i\right)}\frac{\kappa_4}{p}\mathbb{E}\left[l_1((1-\mathbb{E})\operatorname{Tr}{R_1}(z_1))^{l_1-1}((1-\mathbb{E})\operatorname{Tr}{R_1}(z_2))^{l_2-1}\right.\\ &\left.\frac{\partial}{\partial z_1}\left(\frac{	-\sigma_im(z_1)}{1+\phi^{-1/2}m(z_1) \sigma_i}\frac{	-\sigma_im(z_2)}{1+\phi^{-1/2}m(z_2) \sigma_i}\right)\right]\\ 
		=&l_1\frac{\kappa_4}{p}\mathbb{E}\left[((1-\mathbb{E})\operatorname{Tr}{R_1}(z_1))^{l_1-1}((1-\mathbb{E})\operatorname{Tr}{R_1}(z_2))^{l_2-1}\right]\frac{\partial}{\partial z_1}\left(\frac{\sigma_im(z_1)}{1+\phi^{-1/2}m(z_1) \sigma_i}\right)\frac{\partial}{\partial z_2}\left(\frac{\sigma_im(z_2)}{1+\phi^{-1/2}m(z_2) \sigma_i}\right)\\
		=&l_1 \mathbb{E}\left[((1-\mathbb{E})\operatorname{Tr}{R_1}(z_1))^{l_1-1}((1-\mathbb{E})\operatorname{Tr}{R_1}(z_2))^{l_2-1}\right]\\ 
		&\times\kappa_4 \phi\left(\frac{\partial^2}{\partial z_1 \partial z_2}\left[\frac{1}{p} \sum_{i=1}^p \frac{1}{\left(1+\phi^{-1/2}{m}(z_1) \sigma_i\right)\left(1+\phi^{-1/2}{m}(z_2) \sigma_i\right)}\right]\right). 
	\end{aligned}
\end{equation*}
This immediately yields that
 $$\omega(z_1,z_2)=2\left(\frac{{m}^{\prime}(z_1) {m}^{\prime}(z_2)}{\left({m}(z_1)-{m}(z_2)\right)^2}-\frac{1}{\left(z_1-z_2\right)^2}\right)+\kappa_4 \phi\frac{\partial^2}{\partial z_1 \partial z_2}\left(\frac{1}{p} \sum_{i=1}^p \frac{1}{\left(1+\phi^{-1/2}{m}(z_1) \sigma_i\right)\left(1+\phi^{-1/2}{m}(z_2) \sigma_i\right)}\right).$$
 
This concludes the proof of (\ref{eq_wickl2case}) when $l=2.$  For general $l,$ the discussion is similar except for some notional modifications. In fact, we can conclude the prove using induction since we can expand $\mathcal{Y}\left(z_1\right)$ as in \eqref{eq_resolventwick} to get
\begin{equation}\label{eq_referrefer}
\mathbb{E}\left[\mathcal{Y}\left(z_1\right) \cdots \mathcal{Y}\left( z_l\right)\right]=\sum_{s=2}^l \omega\left(z_1, z_s\right) \mathbb{E} \prod_{t \notin\{1, s\}} \mathcal{Y}\left( z_t\right)+\mathrm{O}_{\prec}\left(\mathsf{Re}\right) .
\end{equation}
We omit further details.

\subsection{Proof of Lemma \ref{lem_controlofLi}}	
Recall that $\mathsf{L}_{i}=\sum_{k=2}^{l_0}\mathsf{L}_{i}(k),$  where $l_0=4$ and
$$
\mathsf{L}_{i}(k):=\sum_{ j=1}^n\sqrt{\sigma_i}\frac{\kappa_{k+1}}{(np)^{(k+1)/4}k!} \mathbb{E} \frac{\partial^k\left((1-\mathbb{E})[{((1-\mathbb{E})\operatorname{Tr}{R_1}(z_1))^{l_1}((1-\mathbb{E})\operatorname{Tr}{R_1}(z_2))^{l_2-1}}]({R_1}(z_2)Y)_{i j}\right)}{\partial x_{ ij}^k} .
$$
First, when $k=2$, we have that 
\begin{equation}\label{eq_Li2a}
	\mathsf{L}_{i}(2):=\sum_{ j=1}^n\sqrt{\sigma_i}\frac{\kappa_3}{(np)^{3/4}2} \mathbb{E} \frac{\partial^2\left((1-\mathbb{E})[{((1-\mathbb{E})\operatorname{Tr}{R_1}(z_1))^{l_1}((1-\mathbb{E})\operatorname{Tr}{R_1}(z_2))^{l_2-1}}]({R_1}(z_2)Y)_{i j}\right)}{\partial x_{ ij}^2} .
\end{equation}
Due to similarity, we only focus on some representative terms. By chain rule, one of the terms in $\mathsf{L}_i(2)$ is  
$$\begin{aligned}
&\sum_{ j=1}^n\sqrt{\sigma_i}\frac{\kappa_3}{2(np)^{3/4}} \mathbb{E} \left((1-\mathbb{E})[{((1-\mathbb{E})\operatorname{Tr}{R_1}(z_1))^{l_1}((1-\mathbb{E})\operatorname{Tr}{R_1}(z_2))^{l_2-1}}]\frac{\partial^2({R_1}(z_2)Y)_{i j}}{\partial x_{ ij}^2} \right)\\
=&\sum_{ j=1}^n\sqrt{\sigma_i}\frac{\kappa_3}{2(np)^{3/4}} \mathbb{E} \left([{((1-\mathbb{E})\operatorname{Tr}{R_1}(z_1))^{l_1}((1-\mathbb{E})\operatorname{Tr}{R_1}(z_2))^{l_2-1}}](1-\mathbb{E})\frac{\partial^2({R_1}(z_2)Y)_{i j}}{\partial x_{ ij}^2} \right). 
\end{aligned}$$
Using \eqref{eq_resolventderi1}--\eqref{eq_resolventderi3}, we have that
\begin{equation}\label{eq_resolventderi4}\begin{aligned}
			\frac{\partial^2 ({R_1}Y)_{ij}}{\partial x_{ij}^2}=&\sigma_i\left\{-2({R_1})_{ii}({R_1}Y)_{ij}+2({R_1})_{ii}({R_1}Y)_{ij}(Y^*R_1Y)_{jj}\right.\\ 
			&-({R_1})_{ii}\left[ 2({R_1}Y)_{ij} -2({R_1}Y)_{ij}(Y^*{R_1}Y)_{jj}  \right]\\
		&\left.-2\left[
		({R_1})_{ii}-({R_1})_{ii}(Y^*R_1Y)_{jj}-(R_1Y)_{ij}(R_1Y)_{ij}	\right](R_1Y)_{ij}
		\right\}\\ 
		=&\sigma_i ( -6({R_1})_{ii}({R_1}Y)_{ij}+6({R_1})_{ii}({R_1}Y)_{ij}(Y^*R_1Y)_{jj}+2(({R_1}Y)_{ij})^3).
	\end{aligned}
\end{equation}
Before proceeding to the actually control, we summarize some important observations. Note that every time when we take the derivatives with respect to $x_{ij},$ it always generates some terms with one additional $i$ and one additional $j$ before further simplification. Consequently, we at least have one $({R_1}Y)_{ij}$ factor, which can be well controlled Theorem \ref{thm_locallaw} in the sense that $|(1-\mathbb{E})({R_1}Y)_{ij}|\prec \phi^{-1/4}\Psi(z)$. 

We now control these terms.  Recall (\ref{eq_gidefinition}). In view of \eqref{eq_variancerough}, it suffices to control  terms such as $\sum_{i=1}^p\mathfrak{g}_i(z_2)(1-\mathbb{E})\sigma_i\sum_{j=1}^n[({R_1})_{ii}({R_1}Y)_{ij}(Y^*R_1Y)_{jj}](z_2)$. Note that the diagonals can be replaced by their deterministic counterparts as in Theorem \ref{thm_locallaw}. With a continuity argument similar to the proof of Theorem \ref{thm_locallaw} (see Section 5.2 of \cite{Li2021} for more details), we can show that 
$$(1-\mathbb{E})\sum_{i=1}^p\mathfrak{g}_i(z_2)\sigma_i\sum_{j=1}^n\left[ \frac{({R_1})_{ii}}{\sqrt{p}}({R_1}Y)_{ij}\frac{(Y^*R_1Y)_{jj}}{\sqrt{n}}\right](z_2)=\phi^{-1/2}\Psi(z_2).$$
 Similarly we can handle the rest terms in \eqref{eq_resolventderi4}.  In summary, we can show that 
\begin{equation}\label{eq_errorgiL2i}
	\begin{aligned}
		&\sum_{i=1}^p\mathfrak{g}_i(z_2)\mathsf{L}_i(2)\\
		=&\sum_{i=1}^p\sum_{ j=1}^n\sqrt{\sigma_i}\mathfrak{g}_i(z_2)\frac{\kappa_3}{(np)^{3/4}2} \mathbb{E} \left([{((1-\mathbb{E})\operatorname{Tr}{R_1}(z_1))^{l_1}((1-\mathbb{E})\operatorname{Tr}{R_1}(z_2))^{l_2-1}}](1-\mathbb{E})\frac{\partial^2({R_1}(z_2)Y)_{i j}}{\partial x_{ ij}^2} \right)\\ 
		=&\mathrm{O}_{\prec}\left(\frac{ (np)^{-3/4}}{|\eta_1|^{l_1}|\eta_2|^{l_2-1}}\phi^{-1/2} \sqrt{np}\frac{1}{|\eta_2|}\Psi(z_2)\left(1+\frac{1}{(\sqrt{n|\eta_2|})^2}\right) \right)\\ 
		=&\mathrm{O}_{\prec}\left(\frac{1}{|\eta_1|^{l_1}|\eta_2|^{l_2}} \frac{n^{1/4}}{p^{3/4}}\Psi(z_2)\right).
	\end{aligned}
\end{equation}

Second, for $\mathsf{L}_{i}(3)$, by a discussion similar to (\ref{eq_errorgiL2i}), we have that
\begin{equation*}
\begin{aligned}
	&\mathsf{L}_{i}(3)\\
	=&\sum_{j=1}^n\frac{\sqrt{\sigma_i}3\kappa_4}{3!np}\mathbb{E}\left\{ \frac{\partial^2\left((1-\mathbb{E})[{((1-\mathbb{E})\operatorname{Tr}{R_1}(z_1))^{l_1}((1-\mathbb{E})\operatorname{Tr}{R_1}(z_2))^{l_2-1}}]\right) }{\partial x_{ij}^2}\frac{\partial{({{R_1}(z_2)Y})_{i j}}}{\partial x_{ ij}}\right\}\\
	&+\mathrm{O}\left( \frac{\Psi(z_2)}{|\eta_1|^{l_1}|\eta_2|^{l_2-1}} \right)\\
	=	&\sum_{j=1}^n\frac{\sqrt{\sigma_i}\kappa_4}{2np}\mathbb{E}\left\{ \frac{\partial^2[{((1-\mathbb{E})\operatorname{Tr}{R_1}(z_1))^{l_1}((1-\mathbb{E})\operatorname{Tr}{R_1}(z_2))^{l_2-1}}] }{\partial x_{ij}^2}\frac{\partial{({{R_1}(z_2)Y})_{i j}}}{\partial x_{ ij}}\right\}+\mathrm{O}\left(\frac{\Psi(z_2)}{|\eta_1|^{l_1}|\eta_2|^{l_2-1}} \right).\\
\end{aligned}	
\end{equation*}
For the first term on the right-hand side of the above equation, using \eqref{eq_resolventderi1} and \eqref{eq_resolventderi2}, we find that it suffices to handle the following term
$$\begin{aligned}
	&\sum_{j=1}^n\frac{\sqrt{\sigma_i}\kappa_4}{2np}\mathbb{E}\left\{ l_1((1-\mathbb{E})\operatorname{Tr}{R_1}(z_1))^{l_1-1}\frac{\partial[-2\sqrt{\sigma_i}\frac{\partial (R_1(z_1)Y)_{ij}}{\partial z_1}]}{\partial X_{ij}}((1-\mathbb{E})\operatorname{Tr}{R_1}(z_2))^{l_2-1} \frac{\partial{({{R_1}(z_2)Y})_{i j}}}{\partial x_{ ij}}\right.\\ &+\left. (l_2-1)((1-\mathbb{E})\operatorname{Tr}{R_1}(z_1))^{l_1}((1-\mathbb{E})\operatorname{Tr}{R_1}(z_2))^{l_2-2}\frac{\partial[-2\sqrt{\sigma_i}\frac{\partial (R_2(z_2)Y)_{ij}}{\partial z_2}]}{\partial X_{ij}} \frac{\partial{({{R_1}(z_2)Y})_{i j}}}{\partial x_{ ij}}\right\}.\\		
\end{aligned}$$
Using 	\eqref{eq_resolventderi3}
and 	${R_1}Y=Y{R_2}$, $Y^* {R_1}={R_2} Y^*$, we can rewrite 	$$
({R_1})_{ii}-({R_1})_{ii}(Y^*R_1Y)_{jj}-(R_1Y)_{ij}(Y^*R_1)_{ji}=-z({R_2})_{jj}({R_1})_{ii}-(R_1Y)_{ij}(Y^*R_1)_{ji}.
$$
Together with Theorem \ref{thm_locallaw}, we have that 
$$\begin{aligned}
	&\sum_{j=1}^n\frac{\sqrt{\sigma_i}\kappa_4}{2np}\mathbb{E}\left\{ l_1((1-\mathbb{E})\operatorname{Tr}{R_1}(z_1))^{l_1-1}\frac{\partial[-2\sqrt{\sigma_i}\frac{\partial (R_1(z_1)Y)_{ij}}{\partial z_1}]}{\partial x_{ij}}((1-\mathbb{E})\operatorname{Tr}{R_1}(z_2))^{l_2-1} \frac{\partial{({{R_1}(z_2)Y})_{i j}}}{\partial x_{ ij}}\right.\\ &+\left. (l_2-1)((1-\mathbb{E})\operatorname{Tr}{R_1}(z_1))^{l_1}((1-\mathbb{E})\operatorname{Tr}{R_1}(z_2))^{l_2-2}\frac{\partial[-2\sqrt{\sigma_i}\frac{\partial (R_2(z_2)Y)_{ij}}{\partial z_2}]}{\partial x_{ij}} \frac{\partial{({{R_1}(z_2)Y})_{i j}}}{\partial x_{ ij}}\right\}\\	
	=&\frac{-\kappa_4}{p}\left\{l_1((1-\mathbb{E})\operatorname{Tr}{R_1}(z_1))^{l_1-1}((1-\mathbb{E})\operatorname{Tr}{R_1}(z_2))^{l_2-1}\frac{\partial}{\partial z_1}\left(\frac{	-\sigma_im(z_1)}{1+\phi^{-1/2}m(z_1) \sigma_i}\frac{	-\sigma_im(z_2)}{1+\phi^{-1/2}m(z_2) \sigma_i}\right)\right.\\ 
	&+\left. (l_2-1)((1-\mathbb{E})\operatorname{Tr}{R_1}(z_1))^{l_1}((1-\mathbb{E})\operatorname{Tr}{R_1}(z_2))^{l_2-2}\frac{1}{2}\frac{\partial}{\partial z_2}\left(\frac{	-\sigma_im(z_2)}{1+\phi^{-1/2}m(z_2) \sigma_i}\right)^2\right\}+\mathcal{E}_{L,3}(i),
\end{aligned}$$
where $\mathcal{E}_{L,3}(i)$ is defined as
$$\begin{aligned}
	&\mathcal{E}_{L,3}(i):\\
		=&\frac{-\kappa_4}{p}\left\{l_1((1-\mathbb{E})\operatorname{Tr}{R_1}(z_1))^{l_1-1}((1-\mathbb{E})\operatorname{Tr}{R_1}(z_2))^{l_2-1}\frac{\partial}{\partial z_1}\right.\\ &\left[\left(-z_1({R_2})_{jj}(z_1)({R_1})_{ii}(z_1)-(R_1Y)_{ij}(z_1)(Y^*R_1)_{ji}(z_1)-\frac{	-\sigma_im(z_1)}{1+\phi^{-1/2}m(z_1) \sigma_i}\right)\right.\\
		&\left.\left.\left(-z_2({R_2})_{jj}(z_2)({R_1})_{ii}(z_2)-(R_1Y)_{ij}(z_2)(Y^*R_1)_{ji}(z_2)-\frac{	-\sigma_im(z_2)}{1+\phi^{-1/2}m(z_2) \sigma_i}\right)\right]\right.\\ 
	&+ (l_2-1)((1-\mathbb{E})\operatorname{Tr}{R_1}(z_1))^{l_1}((1-\mathbb{E})\operatorname{Tr}{R_1}(z_2))^{l_2-2}\frac{1}{2}\frac{\partial}{\partial z_2}\left(-z_2({R_2})_{jj}(z_2)({R_1})_{ii}(z_2)\right.\\ 
	&\left.\left.-(R_1Y)_{ij}(z_2)(Y^*R_1)_{ji}(z_2)-\frac{	-\sigma_im(z_2)}{1+\phi^{-1/2}m(z_2) \sigma_i}\right)^2\right\}.\\ 
\end{aligned} $$

In view of \eqref{eq_variancerough}, we need to control $\sum_{i=1}^p\mathfrak{g}_i(z_2)\mathcal{E}_{L,3}(i)$. This can be done via an argument similar to (\ref{eq_errorgiB1i}). We have that 
\begin{equation*}
	\begin{aligned}			&\sum_{i=1}^p\mathfrak{g}_i(z_2)\mathcal{E}_{L,3}(i)\\ 
	&=\mathrm{O}_{\prec}\left(\frac{1}{|\eta_1|^{l_1-1}|\eta_2|^{l_2-1}}\frac{1}{p}\frac{1}{|\eta_1|}\left(\frac{1}{|\eta_1|}+\Psi(z_1)\right)\left(\frac{1}{|\eta_2|}+\Psi(z_2)\right)+\frac{1}{|\eta_1|^{l_1}|\eta_2|^{l_2-2}}\frac{1}{p}\frac{1}{|\eta_2|}\left(\frac{1}{|\eta_2|}+\Psi(z_2)\right)^2\right)
	\end{aligned}
\end{equation*}

Finally, for $\mathsf{L}_i(4),$ the discussions are analogous except that we need to handle higher order derivatives like $\partial^3 (R_1 Y)_{ij}/\partial x_{ij}^3.$ We have that 
\begin{equation*}
\sum_{i=1}^p\mathfrak{g}_i(z_2)\mathsf{L}_i(4)
		=\mathrm{O}_{\prec}\left(\frac{1}{|\eta_1|^{l_1}|\eta_2|^{l_2}} \frac{\sqrt{\eta_2}}{\sqrt{p}}\right).
\end{equation*}   
This completes our proof. 

\section{Proof of of the main theorems}\label{sec_proofofmainresult}

For notional convenience, we denote that for $f_i(x)$ in (\ref{eq_gxdefinition}) 
 	$$
	\mathcal{Z}_{{\eta_0},E}(f_i):=\operatorname{Tr} f_i\left(Q\right)-\mathbb{E} \operatorname{Tr} f_i\left(Q\right). 
	$$ 
Recall (\ref{eq_decomposition}).	In what follows, we will study the deterministic mean bias part $\mathcal{M}_{\eta_0, E}(f_i)$ and the random term $\mathcal{Z}_{\eta_0,E}(f_i).$ For simplicity, as before, we also denote that $Y:=\Sigma^{1/2}X$.  

We point out that the analysis of $\mathcal{Z}_{\eta_0,E}(f_i)$ is similar to that of Theorem \ref{them_CLTresolvent} and the key is to prove the following lemma which is an analog of Lemma \ref{lem:resolventwick}. Recall (\ref{eq_moreconvention}).

\begin{lem}\label{lem_CLTforgeneralfunction}
		Suppose   $X$ and $\Sigma$ satisfy Assumptions  \ref{assum:XSigma}. Then for  $E \in \mathbf{R}$ any $l \in \mathbb{N}$, we have that 
				$$
				\mathbb{E}\left[\prod_{j=1}^l\mathcal{Z}_{{\eta_0}, E}\left( f_{i_j}\right)\right]=\left\{\begin{array}{ll}
					\sum \prod \varpi\left(f_{i_s}, f_{i_t}\right)+\mathrm{O}_{\prec}\left(n^{-\mathsf{c}}\right), & \text { if } l \in 2 \mathbb{N} \\
					\mathrm{O}_{\prec}\left(n^{-\mathsf{c}}\right), & \text { otherwise }
				\end{array},\right.
				$$
				for some constant $c>0$, where ${i_j}\in\{1,\ldots,\mathsf{K}\}$, $1\le j\le l$ and 		$\sum \prod$ means summing over all distinct ways of partitions into pairs. Here $\varpi\left(f_i, f_j\right) \equiv \varpi_n\left(f_i, f_j\right)$ is defined as
				\begin{equation}\label{eq_importantform}
				\begin{aligned}
					 \begin{cases}
						-\frac{1}{4\pi^2} \iint f_i\left(x_1\right) f_j\left(x_2\right) \alpha\left(x_1,x_2\right) \mathrm{d} x_1 \mathrm{d} x_2  -\frac{1}{4\pi^2}\iint f_i\left(x_1\right) f_j\left(x_2\right)\beta\left(x_1, x_2\right) \mathrm{d} x_1 \mathrm{d} x_2,&\text{ when $\eta_0 \asymp 1$ }\\
						-\frac{1}{4\pi^2}\iint f_i\left(x_1\right) f_j\left(x_2\right)\beta\left(x_1, x_2\right) \mathrm{d} x_1 \mathrm{d} x_2,&\text{ when $\eta_0=\mathrm{o}(1)$ }
					\end{cases}. 
				\end{aligned}
				\end{equation}

		\end{lem}

\subsection{Proof of Lemma \ref{lem_CLTforgeneralfunction}}\label{sec_goodandbadregion}

Throughout this section,  to avoid confusion with the scaling parameters $\eta$'s, in this section we use $y$ for $\operatorname{Im} z$ and write $z=x+\mathrm{i}y.$ According to the definition (\ref{eq_gixform}), we have that for $i=1,\ldots,\mathsf{K}$
		\begin{equation}\label{eq_fbounds}
					\|f_i\|_1 =\mathrm{O}(\eta_0),  \quad\left\|f_i^{\prime}\right\|_1=\mathrm{O}(1), \quad\left\|f_i^{\prime \prime}\right\|_1=\mathrm{O}(\eta_0^{-1}).
		\end{equation}
\begin{proof}[\bf Proof]
Due to similarity, we mainly focus on a single test function and discuss how to generalize the results to $\mathsf{K}>1$ in the end of the proof. In this case, it suffices to calculate 	$\mathbb{E}\left[\mathcal{Z}_{{\eta_0}, E}^l(f)\right], l \in \mathbb{N}. $ Similar to (\ref{eq_meanonereduceone}), the staring point is the following representation
\begin{equation}\label{eq_mathcalzformformformform}
\mathcal{Z}_{{\eta_0},E}(f)=\frac{1}{\pi} \int_{\mathbb{C}} \frac{\partial}{\partial \bar{z}} \tilde{f}(z)(\operatorname{Tr}({R_1}(z))-\mathbb{E} \operatorname{Tr} {R_1}(z)) \mathrm{d}^2 z. 
\end{equation}
In what follows, for notional simplicity, we write $\mathcal{Z}(f) \equiv \mathcal{Z}_{{\eta_0},E}(f).$ 	We let  \begin{equation}\label{eq_etatildeunderline}
			\tilde{{\eta}}:=n^{-\varepsilon_0} {\eta_0}\quad \underline{\eta}:=n^{-\varepsilon_1} {\eta_0}
		\end{equation} for some small constants $\varepsilon_1>\varepsilon_0$, which we will choose later. 	To further apply Lemma \ref{lem_HSformula}, we define the almost analytical extension of $f$ more concretely as follows $$\tilde{f}(z):= \left( f(x)+\mathrm{i}yf^{\prime}(x) \right)\chi\left(\frac{y}{\tilde{\eta}}\right).$$
We further denote that for $z=x+\mathrm{i}y$	
	   \begin{equation}\label{eq_thetafdefn}
			\theta_f(z):=\frac{1}{\pi}\frac{\partial \tilde{f}(z)}{\partial \bar{z}}=
			\frac{\mathrm{i}}{2 \pi}yf^{\prime\prime}(x) \chi(y / \tilde{\eta})-\frac{1}{2 \pi \tilde{\eta}}yf^{\prime}(x) \chi^{\prime}(y / \tilde{\eta})+\frac{\mathrm{i}}{2 \pi \tilde{\eta}} f(x) \chi^{\prime}(y / \tilde{\eta}).
		\end{equation}
Together with (\ref{eq_mathcalzformformformform}) and the definition in (\ref{eq_mathcalY}), it suffices to calculate 		
		\begin{equation}\label{eq:Zfpowers}
			\mathbb{E}[\mathcal{Z}(f)]^l=\int \theta_f\left(z_1\right) \cdots \theta_f\left(z_l\right)  \mathbb{E}\left[\mathcal{Y}\left(z_1\right) \cdots \mathcal{Y}\left(z_l\right)\right] \mathrm{d}^2 z_1 \cdots \mathrm{d}^2 z_l,	
		\end{equation}
		where $z_i:=x_i+\mathrm{i} y_i, \ 1 \le i \le l$. In what follows, we will calculate (\ref{eq:Zfpowers}) using Lemma \ref{lem:resolventwick} by splitting the regions into "good" one $\mathcal{R}$ and "bad" one $\mathcal{R}^c$ where 
\begin{equation*}
		\mathcal{R}:=\left\{z_1, z_2, \cdots, z_l \in \mathbb{C}:\left|y_1\right|, \cdots,\left|y_l\right| \in[\underline{\eta}, 2 \tilde{{\eta}}]\right\}.
\end{equation*}

\noindent{\bf Control on the "bad" region.} We now control the integral on the bad region $\mathcal{R}^c$ assuming the following bound holds 
\begin{equation}\label{eq_keyboundbound}
	|(1-\mathbb{E})\operatorname{Tr}{R_1}(z)|\prec \frac{1}{|\operatorname{Im}z|},
\end{equation}
whose proof will be given in Section \ref{imporveconcentration} after some necessary notations are introduced. We first bound on a single spectral parameter. In this case,  
using the definition of $\mathcal{R}^c,$ we need to bound the following two integrals
		$$
		\int_{|y| \le \underline{\eta}}\left|\theta_f(z)\frac{1}{|y| }\right| \mathrm{d}^2 z, \quad \int_{\underline{\eta} \le|y| \le 2 \tilde{{\eta}} }\left|\theta_f(z)\frac{1}{|y|}\right| \mathrm{d}^2 z .
		$$
		
For the first integral, note that by definition, we have $\theta_f(z)=0$ for $|y| \geqslant 2 \tilde{{\eta}}$ and $\chi^{\prime}(y / \tilde{{\eta}})=0$ for $|y| \le \tilde{{\eta}}$. Together with \eqref{eq_fbounds},  we get for some constant $C^{''}>0$
		\begin{equation}\label{eq_lowerlowerbandintegralbound}
			\begin{aligned}
				\int_{|y| \le \underline{\eta}}\left|\theta_f(z)\frac{1}{|y|}\right| \mathrm{d}^2 z \le \int_{|y| \le \underline{\eta}} |f^{\prime\prime}(x)| \mathrm{d}^2 z 
				\le{{\underline{\eta}}} \int_{|\tilde{y}| \le 1} C''/{\eta_0} \mathrm{d} {x} \mathrm{d} \tilde{y}\le n^{-\varepsilon_1}
			\end{aligned}
		\end{equation}
where in the second step we used change of variable that $
		\tilde{y}:=y/\underline{\eta} .$

For the second integral,	 by a discussion similar to (\ref{eq_lowerlowerbandintegralbound}), we can show that
		\begin{equation}\label{eq_lowerbandintegralbound}
					\int_{\underline{\eta} \le|y| \le \tilde{{\eta}}}\left|\theta_f(z)\frac{1}{|y|}\right| \mathrm{d}^2 z \le C ''n^{-\varepsilon_0} .
		\end{equation}
Moreover, we have that 
\begin{equation}\label{eq_upperbandintegralbound}
			\begin{aligned}
		\int_{\tilde{{\eta}} \le|y| \le 2 \tilde{{\eta}}}\left|\theta_f(z)\frac{1}{|y|}\right|  \mathrm{d}^2 z & \le n^{-\varepsilon_0}+\int_{\tilde{{\eta}} \le|y| \le 2 \tilde{{\eta}}}\frac{1}{|y|\tilde{{\eta}}}\left|\left(f(x)+\mathrm{i} y f^{\prime}(x)\right) \chi^{\prime}(y/\tilde{{\eta}})\right| \mathrm{d}x\mathrm{d}y \le Cn^{\varepsilon_0}.
	\end{aligned} 
\end{equation}
This provides the controls for a single spectral parameter. For the integral (\ref{eq:Zfpowers}), using the above estimates, it is not hard to see that 
\begin{equation}\label{eq_epsilon1constraint}
					\begin{aligned}
				&\int_{\mathcal{R}^c} \theta_f\left(z_1\right) \cdots \theta_f\left(z_l\right) \mathbb{E}\left[\mathcal{Y}\left(z_1\right) \cdots \mathcal{Y}\left(z_l\right)\right] \mathrm{d}^2 z_1 \cdots \mathrm{d}^2 z_l \\
				\prec 	&  \sum_{s=1}^l\int_{\left|y_s\right| \leqslant \underline{\eta}} \prod_{i=1}^l\left|\phi_f\left(z_i\right)\frac{1}{|y_i|}\right| \mathrm{d}^2 z_1 \cdots \mathrm{d}^2 z_l \lesssim n^{-\varepsilon_1} \cdot n^{(l-1) \varepsilon_0} \leqslant n^{-\varepsilon_1/2},
			\end{aligned}
		\end{equation}
by choosing $(2l+2) \varepsilon_0<\varepsilon_1<\tau_1/3$. This shows that the integral in (\ref{eq:Zfpowers}) over the "bad" region $\mathcal{R}^c$ is negligible.\vspace*{0.1in}

{\noindent \bf Calculations on the "good" region.} Combining \eqref{eq:Zfpowers} with (\ref{eq_epsilon1constraint}), we see that 
		\begin{equation}\label{eq:goodregion}
			\mathbb{E}[\mathcal{Z}(f)]^l= \int_{\mathcal{R}} \theta_f\left(z_1\right) \cdots \theta_f\left(z_l\right) \mathbb{E}\mathfrak{G}  \mathrm{d}^2z_1 \cdots \mathrm{d}^2 z_l+\mathrm{O}_{\prec}\left(n^{-\varepsilon_1/2 }\right),	
		\end{equation}
		where we abbreviate $\mathfrak{G}:=\mathcal{Y}\left(z_1\right) \cdots \mathcal{Y}\left(z_l\right)$. By 
	Lemma \ref{lem:resolventwick} and its proof (c.f. (\ref{eq_referrefer})), we have that 
		\begin{align}\label{eq:Grelaxeta}
		&\mathbb{E} \mathfrak{G}=\omega\left(z_1, z_s\right) \mathbb{E} \prod_{t \notin\{1, s\}} \mathcal{Y}\left(z_t\right)  \\
		& +\mathrm{O}_{\prec}\left(\frac{1}{\prod_{i=1}^l|y_i|}\left(\frac{1}{n \min_{i=1}^l\{|y_i|\sqrt{\kappa_i+|y_i|}\}}+\frac{1}{n^{3/4}p^{1/4} \min_{i=1}^l\{\sqrt{|\eta_i|(\kappa_i+|\eta_i|)}\}}+\frac{1}{\sqrt{n}}\right)\right) \nonumber 
\end{align}
Plugging \eqref{eq:Grelaxeta} in \eqref{eq:goodregion} and using  \eqref{eq_fbounds}, \eqref{eq_lowerbandintegralbound} and \eqref{eq_upperbandintegralbound} to  bound the error terms, we obtain that for $l \geqslant 2$,
\begin{equation} \label{eq_CLTinductionsingle}
		\begin{aligned}
			\mathbb{E}[\mathcal{Z}(f)]^l & =(l-1)\left( \int_{\underline{\eta} \le\left|y_1\right|,\left|y_s\right| \le 2 \tilde{{\eta}}} \theta_f\left(z_1\right) \theta_f\left(z_s\right) \omega\left(z_1, z_s\right) \mathrm{d}^2 z_1 \mathrm{~d}^2 z_s\right) \mathbb{E}[\mathcal{Z}(f)]^{l-2} \\
			& +\mathrm{O}_{\prec}\left(n^{-\varepsilon_1/2}+n^{l \varepsilon_0}(n \underline{\eta}\min_{i=1}^l\sqrt{\kappa_i+\underline{\eta}})^{-1}\right),
		\end{aligned}
\end{equation}
where $\underline{\eta}=n^{-\varepsilon_1} {\eta_0} \geqslant n^{-1+\tau_1-\varepsilon_1}$. Since we have chosen the constants such that $(2l+2)\varepsilon_0<\varepsilon_1<\tau_1/3$, we can bound $n^{l \varepsilon_0}(n \underline{\eta}\min_{i=1}^l\sqrt{\kappa_i+\underline{\eta}})^{-1} \le n^{l\varepsilon_0+3\varepsilon_1/2}(n \eta_0\min_{i=1}^l\sqrt{\kappa_i+\eta_0})^{-1}\le n^{2\varepsilon_1}n^{-\tau_1}\le n^{-\varepsilon_1}$. 
		It remains to study the factor
		\begin{equation}\label{eq_regionalintergal}
		\operatorname{Var}(\mathcal{Z}):= \int_{\underline{\eta} \le\left|y_1\right|,\left|y_2\right| \le 2 \tilde{{\eta}}} \theta_f\left(z_1\right) \theta_f\left(z_2\right) \omega\left(z_1, z_2\right) \mathrm{d}^2 z_1 \mathrm{~d}^2 z_2,
		\end{equation}
		where we have taken $s=2$ and abbreviated $$\omega\left(z_1, z_2\right) =\widehat{\alpha}\left(z_1, z_2\right)+\widehat{\beta}\left(z_1, z_2\right).$$

Next, we translate the regional integral (\ref{eq_regionalintergal}) into a contour integral which can be naturally done with Green's theorem. In order to apply Lemma \ref{lem_Greenthm}, note that by the fact that $m(z)$ is holomprphic in $\mathbb{C}\backslash[\gamma_-,\gamma_+]$, fix $z_1$, we can see that $\omega(z_1,z_2)$ is holomorphic with respect to $z_2$ except on $[\gamma_-,\gamma_+]$ and a neighborhood $\mathsf{B}_{\epsilon}(z_1)$ of $z_1$ for a small enough $\epsilon=\epsilon(n)$ such that $\mathsf{B}_{\epsilon}(z_1)$ will not cross any existing boundaries. This readily implies that 
		$$\frac{\partial }{\partial \bar{z}_2}\omega(z_1,z_2)=0.$$
Similarly, we have that 
		 $$\frac{\partial }{\partial \bar{z}_1}\omega(z_1,z_2)=0, \quad \frac{\partial }{\partial \bar{z}_1}\int \omega(z_1,z_2)\mathrm{d}z_2=0.$$

On the one hand, applying Lemma \ref{lem_Greenthm} to (\ref{eq_regionalintergal}) twice, we have that
$$
		 	\begin{aligned}
		 		\operatorname{Var}(\mathcal{Z})=-\frac{1}{4\pi^2}\oint\oint \tilde{f}(z_1)\tilde{f}(z_2)\omega(z_1,z_2) \mathrm{d}z_1\mathrm{d}z_2,
		 	\end{aligned}
$$
	 where the contour for $z_2$ is  $\partial \Omega_2^0\cup \partial \mathsf{B}_{\epsilon}(z_1)$ with 
\begin{equation}\label{eq_omegaregion}	 
\Omega_2^0:=\{z_2=x_2+\mathrm{i}y_2:\underline{\eta}\leqslant|y_2|\leqslant 2\tilde{\eta}\}. 
\end{equation}	 
	 On the other hand, we can rewrite $\omega(z_1,z_2)$ as  $$\omega(z_1,z_2)=\frac{2}{(z_2-z_1)^2}\left(  \frac{m^{\prime}(z_1)m^{\prime}(z_2) }{m(z_1)m(z_2)}-\frac{1}{p}\sum_{i=1}^p\frac{\sigma_i^2m^{\prime}(z_1)m^{\prime}(z_2)}{(1+\phi^{-1/2}\sigma_im(z_1))(1+\phi^{-1/2}\sigma_im(z_2))}   -1\right).$$ 
Note that there are no singularities of the first order in $\mathsf{B}_{\epsilon}(z_1)$. Moreover, by definition, $\tilde{f}(z)\equiv 0$ when $\operatorname{Im}z=2\tilde{\eta}$. Therefore, we can simply take the contour for $z_2$ as $\{z_2=x_2+\mathrm{i}y_2: |y_2|=\underline{\eta}\}$. Similarly, we can take the contour of $z_1$ as   $\{z_1=x_1+\mathrm{i}y_1: |y_1|=\underline{\eta}\}$. For simplicity, we choose the two contours which do not overlap. Without loss of generality, we denote $\underline{\eta}_2:=\underline{\eta}_1/2:=\underline{\eta}/2$ and $$\Gamma_s:=\{z_s=x_s+\mathrm{i}y_s: |y_s|=\underline{\eta}_s\}, \quad s=1,2.$$
Based on the above discussion, it is easy to see that we can rewrite 
	 \begin{equation}\label{eq_maincontourrough}
	 	\operatorname{Var}(\mathcal{Z})=		 		-\frac{1}{4\pi^2}\oint_{\Gamma_2}\oint_{\Gamma_1} \tilde{f}(z_1)\tilde{f}(z_2)\omega(z_1,z_2) \mathrm{d}z_1\mathrm{d}z_2.
	 \end{equation}
	 
To conclude the proof, it suffices to show the equivalence of (\ref{eq_maincontourrough}) with (\ref{eq_importantform}). The starting point is to rewrite \eqref{eq_maincontourrough} as follows
\begin{equation}\label{eq_reducedreducedreducedform}
  \operatorname{Var}(\mathcal{Z})=\mathcal{K}_{++}+\mathcal{K}_{--}-\mathcal{K}_{+-}-\mathcal{K}_{-+},
  \end{equation}
  where the term $\mathcal{K}_{++}$  contains the integral over the region $\mathcal{R}_{++}:=\left\{y_1 =\underline{{\eta}}_1, y_2 =\underline{{\eta}}_2\right\}$, $\mathcal{K}_{--}$  contains that of $\mathcal{R}_{--}:=\left\{y_1 =-\underline{{\eta}}_1, y_2 =-\underline{{\eta}}_2\right\}$ and $\mathcal{K}_{+-},\mathcal{K}_{-+}$ are defined similarly. Note that
  	\begin{align}\label{eq_mainalphabeta}
  		\mathcal{K}_{++}=&-\frac{1}{4 \pi^2} \int\int  \tilde{f}\left(z_1\right) \tilde{f}\left(z_2\right)\left(\widehat{\alpha}\left(x_1+\mathrm{i}\underline{\eta}_1, x_2+\mathrm{i}\underline{\eta}_2\right)+\widehat{\beta}\left(x_1+\mathrm{i}\underline{\eta}_1, x_2+\mathrm{i}\underline{\eta}_2\right)\right) \mathrm{d} z_1 \mathrm{d} z_2\\
  		=& -\frac{1}{4 \pi^2 } \int\int  \tilde{f}\left(x_1\right) \tilde{f}\left(x_2\right)\left(\widehat{\alpha}\left(x_1+\mathrm{i}\underline{\eta}_1, x_2+\mathrm{i}\underline{\eta}_2\right)+\widehat{\beta}\left(x_1+\mathrm{i}\underline{\eta}_1, x_2+\mathrm{i}\underline{\eta}_2\right)\right) \mathrm{d} x_1 \mathrm{d} x_2 \nonumber \\ 
  	   =& -\frac{1}{4 \pi^2 } \int\int  \left({f}\left(x_1\right) +\mathrm{i} \underline{\eta}_1f^{\prime}(x_1)\right) \left({f}\left(x_2\right) +\mathrm{i} \underline{\eta}_2f^{\prime}(x_2)\right)\left(\widehat{\alpha}\left(x_1+\mathrm{i}\underline{\eta}_1, x_2+\mathrm{i}\underline{\eta}_2\right)+\widehat{\beta}\left(x_1+\mathrm{i}\underline{\eta}_1, x_2+\mathrm{i}\underline{\eta}_2\right)\right) \mathrm{d} x_1 \mathrm{d} x_2 \nonumber \\ 
  	      =& -\frac{1}{4 \pi^2 } \int\int  {f}\left(x_1\right)  {f}\left(x_2\right)\left(\widehat{\alpha}\left(x_1+\mathrm{i}\underline{\eta}_1, x_2+\mathrm{i}\underline{\eta}_2\right)+\widehat{\beta}\left(x_1+\mathrm{i}\underline{\eta}_1, x_2+\mathrm{i}\underline{\eta}_2\right)\right) \mathrm{d} x_1 \mathrm{d} x_2 +\mathrm{O}(n^{-\varepsilon_1}\eta_0).  \nonumber
  	\end{align}
The other terms can be studied similarly. The main differences lie in the signs of the imaginary parts. Note that $\underline{\eta_1}\asymp \underline{\eta}_2\asymp \underline{\eta}=n^{-\varepsilon_1}\eta_0$ and $\mathrm{d}z_1=\mathrm{d}x_1,$  $\mathrm{d}z_2=\mathrm{d}x_2$. Then we find that when the imaginary part goes to zero and the limit exists, we have
$$
\operatorname{Var}(\mathcal{Z})=-\frac{1}{4 \pi^2} \iint_{\mathbb{R}^2} f\left(x_1\right) f\left(x_2\right) \alpha\left(x_1, x_2\right) \mathrm{d} x_1 \mathrm{d} x_2-\frac{1}{4 \pi^2} \iint_{\mathbb{R}^2} f\left(x_1\right) f\left(x_2\right) \beta\left(x_1, x_2\right) \mathrm{d} x_1 \mathrm{d} x_2.
$$
This concludes the proof for $\eta_0 \asymp 1.$	

For $\eta_0=\mathrm{o}(1),$ we apply the following change of variable  
\begin{equation}\label{eq_changeofvariable}
\tilde{x}_i=(x_i-E)/\eta_0, \ i=1,2,
\end{equation}
to the last step of \eqref{eq_mainalphabeta} and get 
  $$
  \begin{aligned}
      & -\frac{1}{4 \pi^2 } \int\int  {f}\left(x_1\right)  {f}\left(x_2\right)\left(\widehat{\alpha}\left(x_1+\mathrm{i}\underline{\eta}_1, x_2+\mathrm{i}\underline{\eta}_2\right)+\widehat{\beta}\left(x_1+\mathrm{i}\underline{\eta}_1, x_2+\mathrm{i}\underline{\eta}_2\right)\right) \mathrm{d} x_1 \mathrm{d} x_2\\  
 &= -\frac{\eta_0^2}{4 \pi^2 } \int\int  {g}\left(\tilde{x}_1\right)  {g}\left(\tilde{x}_2\right)\left(\widehat{\alpha}\left(E+\tilde{x}_1 {\eta_0}+\mathrm{i}\underline{\eta}_1, E+\tilde{x}_2 {\eta_0}+\mathrm{i}\underline{\eta}_2\right)+\widehat{\beta}\left(E+\tilde{x}_1 {\eta_0}+\mathrm{i}\underline{\eta}_1, E+\tilde{x}_2 {\eta_0}+\mathrm{i}\underline{\eta}_2\right)\right) \mathrm{d} \tilde{x}_1 \mathrm{d} \tilde{x}_2 \\
 & :=\left(\mathcal{K}_1\right)_{++}+	\left(\mathcal{K}_2\right)_{++}.
  \end{aligned}$$
For $(\mathcal{K}_1)_{++},$ we see that   
  \begin{align}\label{eq_controlcontrolsmall}
  	\left(\mathcal{K}_1\right)_{++} & =-\frac{{\eta_0}^2}{4 \pi^2} \iint g\left(\tilde{x}_1\right) g\left(\tilde{x}_2\right) \widehat{\alpha}_{++}\left(E+\tilde{x}_1 {\eta_0}, E+\tilde{x}_2 {\eta_0}\right)  \mathrm{d} \tilde{y}_1  \mathrm{~d} \tilde{y}_2+\mathrm{O}(\sqrt{\tilde{{\eta}}}) \nonumber \\
  	& =-\frac{{\eta_0}^2}{4 \pi^2} \iint g\left(\tilde{x}_1\right) g\left(\tilde{x}_2\right) \widehat{\alpha}_{++}\left(E+\tilde{x}_1 {\eta_0}, E+\tilde{x}_2 {\eta_0}\right) \mathrm{d} \tilde{x}_1 \mathrm{d} \tilde{x}_2+\mathrm{O}\left(n^{-\varepsilon_0 / 2}\right) \nonumber \\ 
  	&=\mathrm{O}\left(\frac{1}{\sqrt{\kappa_1+\underline{\eta}_1}\sqrt{\kappa_2+\underline{\eta}_2}}\eta_0^2\right)=\mathrm{O}\left(\frac{1}{\sqrt{\underline{\eta}_1}\sqrt{\underline{\eta}_2}}\eta_0^2\right).
  \end{align}
This indicates that when $\eta_0=\mathrm{o}(1)$,  only the $\beta$ terms will have non-vanishing contributions. Similarly,  we can calculate the integral in $\mathcal{K}_1$ over the other three regions: $\left(\mathcal{K}_1\right)_{+-}$for $\mathcal{R}_{+-}:=\left\{y_1 =\underline{{\eta}}_1, y_2 =-\underline{{\eta}}_2\right\}$, $\left(\mathcal{K}_1\right)_{-+}$for $\mathcal{R}_{-+}:=\left\{y_1 =-\underline{{\eta}}_1, y_2 =\underline{{\eta}}_2\right\}$, and $\left(\mathcal{K}_1\right)_{--}$for $\mathcal{R}_{--}:=\left\{y_1 =-\underline{{\eta}}_1, y_2 =-\underline{{\eta}}_2\right\}$ and obtain the same result.  This completes the proof for $\eta_0=\mathrm{o}(1)$ and hence that of  $\mathbb{E}[\mathcal{Z}(f)^l].$
   
   Finally, we discuss how to generalize the results to multiple functions. Similar to (\ref{eq_mathcalzformformformform}), we can write 
  $$
  \mathbb{E}\left[\mathcal{Z}_{{\eta_0}, E}\left(f_{i_1}\right) \cdots \mathcal{Z}_{{\eta_0}, E}\left( f_{i_l}\right)\right]=\int \theta_{f_{i_1}}\left(z_1\right) \cdots \theta_{f_{i_l}}\left(z_l\right) \mathbb{E}\left[\mathcal{Y}\left( z_1\right) \cdots \mathcal{Y}\left( z_l\right)\right] \mathrm{d}^2 z_1 \cdots \mathrm{d}^2 z_l .
  $$
Following the proof for a single function, it is not hard to see that 
  $$
  \mathbb{E}\left[\mathcal{Z}_{{\eta_0}, E}\left(f_{i_1}\right) \cdots \mathcal{Z}_{{\eta_0}, E}\left( f_{i_l}\right)\right]=\sum_{s=2}^l \varpi\left(f_{i_1}, f_{i_s}\right) \mathbb{E} \prod_{t \notin\{1, s\}} \mathcal{Z}_{{\eta_0}, E}\left( f_{i_t}\right)+\mathrm{O}_{\prec}\left(n^{-c}\right)
  $$
  for some constant $c>0$. Then the proof follows from an induction argument. 

\end{proof}

\subsection{Proof of  Theorems \ref{thm_mainone} and \ref{thm_maintwo}}\label{sec_subproofmaintheorem}

\noindent{\bf Control of $\mathcal{M}_{\eta_0, E}(f_i)$.} For 
simplicity, we first rewrite the quantity a little bit. Using the fact that $Q$ and $\mathcal{Q}$ share the same non-zero eigenvalues and $E\sim\sqrt{\phi}$ and $g_i\in\mathcal{C}^{2}_c\left(\mathbb{R}\right),$ we have that $f_i(0)=g_i(\frac{0-E}{\eta_0})$=0.  Combining (\ref{eq_transformtransferdefinition}), we have that   
$n \int_{\mathbb{R}} f_i(x) \mathrm{d} \varrho(x)=p \int_{\mathbb{R}} f_i(x) \mathrm{d} \varrho_p(x).$ Together with the definition of Stieltjes transform and Fubini's theorem, using Lemma \ref{lem_HSformula}, we have that 
\begin{equation}\label{eq_meanonereduce}
n \int_{\mathbb{R}} f_i(x) \mathrm{d} \varrho(x)=\frac{1}{\pi} \int_{\mathbb{C}} \frac{\partial}{\partial \bar{z}} \tilde{f}_i(z) p m_p(z) \mathrm{d}^2 z.
\end{equation}
Using Lemma \ref{lem_HSformula} again, we have that
\begin{equation}\label{eq_meanonereduceone}
\mathbb{E} \operatorname{Tr} f_i(Q)=\frac{1}{\pi}\int_{\mathbb{C}}\frac{\partial}{\partial \bar{z}}\tilde{f}_i(z)\mathbb{E}\operatorname{Tr} R_1(z)\mathrm{d}^2 z.
\end{equation}
Therefore, due to the regularity properties of $f_i,$ in order to study the quantity $\mathcal{M}_{\eta_0,E}$,  it suffices to study $\mathbb{E}\operatorname{Tr}{R_1}(z)-pm_p(z) $ and then apply a complex integral. The control of $\mathbb{E}\operatorname{Tr}{R_1}(z)-pm_p(z) $ is similar to the discussions in the proof of Lemma \ref{lem:resolventwick} using the cumulant expansion formula Lemma \ref{lem_cumulantexpansion}. Due to similarity, we only sketch the main ideas of the proof.  We start with $\mathbb{E}({R_1})_{ii}, 1 \leq i \leq p$.	 Using 
$$		\frac{\partial ({R_1})_{a b}}{\partial y_{ij}}=-({R_1})_{a i}\left(Y^* {R_1}\right)_{j b}-({R_1} Y)_{a j} ({R_1})_{i b}, \quad \frac{\partial ({R_1})_{a b}}{\partial x_{ij}}=\frac{\partial ({R_1})_{a b}}{\partial y_{ij}} \sqrt{\sigma_i}, 
$$
the relation $z{R_1}(z)=QR_1(z)-I$, Lemma \ref{lem_cumulantexpansion} and Theorem \ref{thm_locallaw}, we have that 
\begin{equation}\label{eq_biasdecomp}
	\begin{aligned}
		&z \mathbb{E} ({R_1})_{i i} 
		=\sqrt{\sigma_i} \sum_{j=1}^n \mathbb{E} X_{i j}({R_1} Y)_{i j}-1\\
		=&\frac{\sqrt{\sigma_i}}{\sqrt{np}} \mathbb{E} \sum_{j=1}^n  \frac{\partial({R_1} Y)_{i j}}{\partial x_{i j}}-1 +\frac{\sqrt{\sigma_i}}{2 ! (np)^{\frac{3}{4}}} \sum_{j=1}^n \kappa_3 \mathbb{E} \frac{\partial^2({R_1} Y)_{i j}}{\partial^2 x_{i j}}+\frac{\sqrt{\sigma_i}}{3 ! np} \sum_{j=1}^n \kappa_4 \mathbb{E} \frac{\partial^3({R_1} Y)_{i j}}{\partial^3 x_{i j}}+\mathrm{O}_{\prec}\left(n^{-\frac{1}{4}}p^{-\frac{5}{4}}\right) \\
		= & \frac{\sigma_i}{\sqrt{np}} \sum_{j=1}^n \mathbb{E} ({R_1})_{i i}-\frac{\sigma_i}{\sqrt{np}} \sum_{j=1}^n \mathbb{E}\left(Y^* ({R_1}) Y\right)_{j j} ({R_1})_{i i}-\frac{\sigma_i}{\sqrt{np}} \sum_{j=1}^n \mathbb{E}\left((({R_1}) Y)_{i j}\right)^2-1 \\ & +\frac{\sigma_i^{3 / 2}}{2 (np)^{\frac{3}{4}}} \sum_{j=1}^n \kappa_3 \mathbb{E}\left(-6 ({R_1})_{i i}(({R_1}) Y)_{i j}+6 ({R_1})_{i i}(({R_1}) Y)_{i j}\left(Y^* ({R_1}) Y\right)_{j j}+2\left((({R_1}) Y)_{i j}\right)^3\right) \\ & +\frac{\sigma_i^2}{6 np} \sum_{j=1}^n \kappa_4 \mathbb{E}\left(-6\left(({R_1})_{i i}\right)^2+12\left(({R_1})_{i i}\right)^2\left(Y^* ({R_1}) Y\right)_{j j}-6\left(({R_1})_{i i}\right)^2\left(\left(Y^* ({R_1}) Y\right)_{j j}\right)^2\right)+\mathrm{O}_{\prec}\left(p^{-1} \Psi(z)\right) \\
		:&=\mathsf{S}_{i,1}+\mathsf{S}_{i,2}+\mathsf{S}_{i,3}+\mathrm{O}_{\prec}\left(p^{-1}\Psi(z)\right).
		\end{aligned}
		\end{equation}
		
Then we deal with the terms $\mathsf{S}_{i,k}, k=1,2,3.$ First, by a discussion similar to the proof of Lemma \ref{lem_controlofLi}, using Theorem \ref{thm_locallaw}, we have that 
\begin{equation}\label{eq_bias2ndrough}
	\begin{aligned}
		\mathsf{S}_{i,1}	=& \phi^{-1/2}\sigma_i \mathbb{E} ({R_1})_{i i}-\frac{\sigma_i}{\sqrt{np}} \mathbb{E}\left[\operatorname{Tr}\left(Y^* {R_1} Y\right) ({R_1})_{i i}\right]-\frac{\sigma_i}{\sqrt{np}} \mathbb{E}(R_1Q R_1)_{i i}-1 \\
		=& \phi^{-1/2}\sigma_i \mathbb{E} ({R_1})_{i i}-\frac{\sigma_i}{\sqrt{np}} \mathbb{E}\left[\operatorname{Tr}\left((Q-z+z){R_1} \right) ({R_1})_{i i}\right]-\frac{\sigma_i}{\sqrt{np}} \mathbb{E}(R_1(Q-z+z) R_1)_{i i}-1 \\
		= & \sigma_i(\phi^{-1/2}-\phi^{1/2}) \mathbb{E} ({R_1})_{i i}-z \sigma_i \phi^{1/2} \mathbb{E}\left[p^{-1} \operatorname{Tr} ({R_1}) ({R_1})_{i i}\right]-\frac{\sigma_i}{\sqrt{np}}\mathbb{E}\left(\frac{\partial}{\partial z}z{R_1}(z)\right)_{ii}-1 \\
		=& \sigma_i(\phi^{-1/2}-\phi^{1/2}) \mathbb{E} ({R_1})_{i i}-z \sigma_i \phi^{1/2} m_{p} \mathbb{E}\left(({R_1})_{i i}+\frac{1}{z\left(1+\phi^{-1/2}m \sigma_i\right)}\right)\\ 
		&+\frac{\sigma_i \phi^{1/2}}{p\left(1+\phi^{-1/2}m \sigma_i\right)} \mathbb{E} \operatorname{Tr} ({R_1})-\frac{\sigma_i^2}{p} \frac{m^{\prime}}{\left(1+\phi^{-1/2}\sigma_i m\right)^2}-1+\tilde{\mathcal{E}}_{i,1},
	\end{aligned}
\end{equation}
where  we denote
 \begin{equation*}
 	\begin{aligned}
 		\tilde{\mathcal{E}}_{i,1}:=&\frac{z\sigma_i\phi^{1/2}}{p}\mathbb{E}\left[\left(\operatorname{Tr}({R_1(z)})-pm_p\right)\left(({R_1})_{ii}-\left(-\frac{1}{z(1+\phi^{-1/2}m\sigma_i)}\right)\right)\right]\\ &-\frac{\sigma_i}{\sqrt{np}}\left(\mathbb{E}\left( \frac{\partial}{\partial z}z{R_1}(z)\right)_{ii}-\frac{\phi^{-1/2}\sigma_im^{\prime}}{(1+\phi^{-1/2}m\sigma_i)^2}\right).
 	\end{aligned}
 \end{equation*} 
Here we used the identity $\frac{\partial(zR(z))}{\partial z}={R_1}(z)+z{R_1}^2(z)$ to estimate $\frac{\sigma_i}{\sqrt{np}}\mathbb{E}[({R_1}Q{R_1})_{ii}]$ in the second step, and used the identity $
	\mathbb{E}[AB]=\mathbb{E}[(A-a)(B-b)]+\mathbb{E}[aB-ab+Ab]
	$
	to control $-z \sigma_i \phi^{1/2} \mathbb{E}\left[p^{-1} \operatorname{Tr} ({R_1}) ({R_1})_{i i}\right]$ in the last step. Plugging \eqref{eq_bias2ndrough} into \eqref{eq_biasdecomp}, we have
\begin{equation*}
	\begin{aligned}
		z \mathbb{E} ({R_1})_{i i} =& \sigma_i(\phi^{-1/2}-\phi^{1/2}) \mathbb{E} ({R_1})_{i i}-z \sigma_i \phi^{1/2} m_{p} \mathbb{E}\left(({R_1})_{i i}+\frac{1}{z\left(1+\phi^{-1/2}m \sigma_i\right)}\right)\\ 
		&+\frac{\sigma_i \phi^{1/2}}{p\left(1+\phi^{-1/2}m \sigma_i\right)} \mathbb{E} \operatorname{Tr} ({R_1})-\frac{\sigma_i^2}{p} \frac{m^{\prime}}{\left(1+\phi^{-1/2}\sigma_i m\right)^2}-1+\tilde{\mathcal{E}}_{i,1}+\mathsf{S}_{i,2}+\mathsf{S}_{i,3}+\mathrm{O}_{\prec}\left(p^{-1}\Psi(z)\right).
	\end{aligned}
\end{equation*}
Using (\ref{eq_transformtransferdefinition}), we can rewrite the above equation as follows
\begin{equation}\label{eq_biaserrform}
	\begin{aligned}
		& z(1+\sigma_i\phi^{-1/2}m)\mathbb{E}({R_1})_{ii}\\
		& =-\frac{ \sigma_i \phi^{1/2} m_{p}} {\left(1+\phi^{-1/2}m \sigma_i\right)}+\frac{\sigma_i \phi^{1/2}}{p\left(1+\phi^{-1/2}m \sigma_i\right)} \mathbb{E} \operatorname{Tr} ({R_1})-\frac{\sigma_i^2}{p} \frac{m^{\prime}}{\left(1+\phi^{-1/2}\sigma_i m\right)^2}-1 \\
		& +\tilde{\mathcal{E}}_{i,1}+\mathsf{S}_{i,2}+\mathsf{S}_{i,3}+\mathrm{O}_{\prec}\left(p^{-1}\Psi(z)\right).
	\end{aligned}
\end{equation}
This yields that 
\begin{equation}\label{eq_biassumrough}
		\begin{aligned}
		&\mathbb{E}\operatorname{Tr}({R_1})\\
		=&\sum_{i=1}^p\frac{1}{z(1+\sigma_i\phi^{-1/2}m)}\left(\frac{\sigma_i \phi^{1/2}}{p\left(1+\phi^{-1/2}m \sigma_i\right)} \mathbb{E} [\operatorname{Tr} ({R_1})-pm_p]-\frac{\sigma_i^2}{p} \frac{m^{\prime}}{\left(1+\phi^{-1/2}\sigma_i m\right)^2}-1 \right. \\
	& \left.	+\tilde{\mathcal{E}}_{i,1}+\mathsf{S}_{i,2}+\mathsf{S}_{i,3}+\mathrm{O}_{\prec}\left(p^{-1}\Psi(z)\right)\right).
	\end{aligned}
\end{equation}
For the error term $\tilde{\mathcal{E}}_{i,1}$, using $\frac{1}{z(1+\sigma_i\phi^{-1/2}m)}\asymp \phi^{-1/2}$ and a fluctuation averaging argument as in the proof of (\ref{eq_averagelocallawtwo}) (c.f. Lemma \ref{lem:fa}), we can obtain that  
\begin{equation}\label{eq_biaserr1}
	\sum_{i=1}^p\frac{1}{z(1+\sigma_i\phi^{-1/2}m)}\tilde{\mathcal{E}}_{i,1}=\mathrm{O}_{\prec}\left(
	n^{-\frac{1}{2}}p^{-\frac{1}{2}}\eta^{-2}\right).
\end{equation}
For the summation involving $\mathsf{S}_{i,2},$ among all the terms, we discuss the following representative term  
\begin{equation}\label{eq_ffgggg}
(np)^{-3/4} \sum_{i=1}^p \sum_{j=1}^n  ({R_1})_{i i}(({R_1}) Y)_{i j}\left(Y^* ({R_1}) Y\right)_{j j} \prec \phi^{1/2} (np)^{-3/4} (np)^{1/2} \phi^{-1/2} \Psi(z)=(np)^{-1/4} \Psi(z),
\end{equation}
where we used the identities ${R_1}Y=Y{R_2}$, $Y^* {R_1}={R_2} Y^*$ and (\ref{eq:entrywiselaw}). The other terms can be handled analogously (and more easily). This results in   
\begin{equation}\label{eq_biaserr2} 
	\sum_{i=1}^p\frac{1}{z(1+\sigma_i\phi^{-1/2}m)} \mathsf{S}_{i,2}=\mathrm{O}_{\prec}\left(n^{\frac{1}{4}}p^{-\frac{3}{4}} \Psi(z) \right).
\end{equation}
Finally, using Theorem \ref{thm_locallaw} and similar arguments as in (\ref{eq_ffgggg}), we see that 
\begin{equation}\label{eq_biaserr3}
	\begin{aligned}
		\mathsf{S}_{i,3}=-\kappa_4 \frac{\sigma_i^2}{p}\left(\frac{m^2}{\left(1+\phi^{-1/2}m \sigma_i\right)^2}\right)+\mathrm{O}_{\prec}\left(p^{-1}\Psi(z)\right) .
	\end{aligned}
\end{equation}
Inserting \eqref{eq_biaserr1} \eqref{eq_biaserr2} \eqref{eq_biaserr3} back into \eqref{eq_biassumrough} and use $z\left(1+\phi^{-1/2}m \sigma_i\right) \asymp \sqrt{\phi}$ again, we get
\begin{equation}\label{eq_biassum1}
	\begin{aligned}
		&\mathbb{E} \operatorname{Tr}({R_1}) -\sum_{i=1}^p\left(-\frac{1}{z\left(1+\phi^{-1/2}m \sigma_i\right)}\right)\\ =&\frac{1}{z\left(1+\phi^{-1/2}m \sigma_i\right)}\left[\frac{\sigma_i \phi^{1/2}}{1+\phi^{-1/2}m \sigma_i}  \mathbb{E}\left[p^{-1} \operatorname{Tr} ({R_1})-m_{p}\right]-{\frac{1}{p} \frac{ \sigma_i^2 m^{\prime}}{\left(1+\phi^{-1/2}\sigma_i m\right)^2} } -\kappa_4 \frac{1}{p} \frac{ \sigma_i^2 m^2}{\left(1+\phi^{-1/2}m \sigma_i\right)^2}\right]\\ 
		&+\mathrm{O}_{\prec}\left(
		n^{-\frac{1}{2}}p^{-\frac{1}{2}}\eta^{-2}+n^{\frac{1}{4}}p^{-\frac{3}{4}} \Psi(z)+\phi^{-1/2}\Psi(z)\right),
	\end{aligned}
\end{equation}
Using Lemma \ref{lem_elementaryidentityes}, we can rewrite the above equation as follows
\begin{equation}\label{eq_biassum2}
	\begin{aligned}
		&\left(1-\frac{1}{p} \sum_{i=1}^p \frac{\phi^{1/2} \sigma_i}{z\left(1+\phi^{-1/2}m(z) \sigma_i\right)^2}\right) \mathbb{E}\left(\operatorname{Tr} ({R_1})-p m_{p}\right)\\ 
		=&{-\frac{1}{p} \sum_{i=1}^p \frac{\sigma_i^2 m^{\prime}}{z\left(1+\phi^{-1/2}m \sigma_i\right)^3} }-\kappa_4 \frac{1}{p} \sum_{i=1}^p \frac{ \sigma_i^2 m^2}{z\left(1+\phi^{-1/2}m \sigma_i\right)^3}+\mathrm{O}_{\prec}\left(
		n^{-\frac{1}{2}}p^{-\frac{1}{2}}\eta^{-2}+n^{\frac{1}{4}}p^{-\frac{3}{4}} \Psi(z)+\phi^{-1/2}\Psi(z)\right).
	\end{aligned}
\end{equation}
Combining \eqref{eq_mmprimeeq1} and \eqref{eq_mmprimeeq3} with a straightforward calculation, we obtain that 
\begin{equation}\label{eq_biaserrb4integral}
	\mathbb{E}(\operatorname{Tr}({R_1})-pm_p)=b(z)+\mathrm{O}_{\prec}\left( \frac{n^{-1}\eta^{-2}+n^{-1/4}p^{-1/4}\Psi(z)+\Psi(z)}{\sqrt{\kappa+\eta}} \right), 
\end{equation}
where we used the control $zm^{\prime}/m\asymp \phi^{1/2}/\sqrt{\kappa+\eta}.$ Recall (\ref{eq_omegaregion}) and let $\Omega$ be defined exactly in the same way as $\Omega_2^0$ for $z.$ Combining (\ref{eq_meanonereduce}), (\ref{eq_meanonereduceone}) and (\ref{eq_biaserrb4integral}), by a discussion similar to (\ref{eq:Zfpowers}),  we readily obtain that 
\begin{equation*}
\mathcal{M}_{\eta_0,E}=\frac{1}{\pi} \int_{\Omega} \frac{\partial}{\partial \bar{z}} \tilde{f}_i(z) b(z) \mathrm{d}^2 z+\mathrm{O}_{\prec}\left( \frac{1}{\sqrt{n \eta_0 \sqrt{\kappa+\eta_0}}} \right), 
\end{equation*}
where we used Lemma \ref{lem:mphizabs}.  By Lemma \ref{lem_Greenthm}, we have that  
\begin{equation*}
\mathcal{M}_{\eta_0,E}=\frac{1}{2\pi\mathrm{i}} \int_{\partial \Omega} \tilde{f}_i(z) b(z) \mathrm{d}z+\mathrm{O}_{\prec}\left( \frac{1}{\sqrt{n \eta_0 \sqrt{\kappa+\eta_0}}} \right).  
\end{equation*}
Finally, we can follow lines of the arguments between (\ref{eq_maincontourrough}) and (\ref{eq_controlcontrolsmall}) to conclude the proof of (\ref{eq_meanfunction}) and (\ref{eq_meanlocal}).

\vspace{0.1in}
\noindent{\bf Distribution of $\mathcal{Z}_{\eta_0, E}(f_i)$.} Using Lemma \ref{lem_CLTforgeneralfunction}, by a discussion similar to the proof of Theorem \ref{them_CLTresolvent}, we can use Lemma \ref{lem_wicktheorem} to conclude that $\{\mathcal{Z}_{\eta_0, E}(f_i)\}$ follows a multivariate Gaussian distribution asymptotically whose mean is $\mathbf{0}$ and covariance matrices satisfy (\ref{eq_varianceglobal}) for $\eta_0 \asymp 1$ and (\ref{eq_covariancelocal}) for $\eta_0=\mathrm{o}(1).$ This completes the proof of Theorems \ref{thm_mainone} and \ref{thm_maintwo} using (\ref{eq_decomposition}).

\section{Proof of the corollaries} \label{sec_subcorollaryproof}

\subsection{Proof of Corollary \ref{cor_globalexamples}}\label{proofofcorollary34}
In this subsection, we prove Corollary \ref{cor_globalexamples} using Theorem \ref{thm_mainone}. Note that when $\Sigma=I,$ \eqref{eq:lsd} becomes
\begin{equation}	\label{eq_MPlawrecudedone}		
z=-\frac{1}{m(z)}+\frac{\phi^{1/2}}{1+\phi^{-1/2}m(z)},\end{equation}
			or equivalently, 
			\begin{equation}\label{eq_MPlawrecudedtwo}
			m(z):=\frac{\phi^{1 / 2}-\phi^{-1 / 2}-z+\mathrm{i} \sqrt{\left(z-\gamma_{-}\right)\left(\gamma_{+}-z\right)}}{2 \phi^{-1 / 2} z}.
			\end{equation}

\noindent {\bf Proof of Part (1).} Following \cite[Section 3]{yao2015large}, for some $a,r$  and $\xi,$ we consider a change of variable that
\begin{equation}\label{eq_basictransform} 
m=-\frac{1}{a+r\xi}, \ |\xi|=1.
\end{equation}
			Together with (\ref{eq_MPlawrecudedone}), we see that  $$\frac{\phi^{1/2}}{1+\phi^{-1/2}m}=\frac{1}{1+\frac{-\phi^{-1/2}}{a+r\xi}}=\frac{\phi^{1/2}(a+r\xi)}{-\phi^{-1/2}+a+r\xi}.$$
Set $a=\phi^{-1/2}$, we further have
					$$\frac{\phi^{1/2}}{1+\phi^{-1/2}m}=r^{-1}\bar\xi+\phi^{1/2}.$$ 
This yields that 
\begin{equation}\label{eq_conformal}
					z=-\frac{1}{m}+\frac{\phi^{1/2}}{1+\phi^{-1/2}m}=\phi^{1/2}+\phi^{-1/2}+r\xi+r^{-1}\bar{\xi}, \ \mathrm{d}z=(r-r^{-1}\xi^{-2})\mathrm{d}\xi, 
\end{equation}
and
$$\begin{aligned}
						m'(z)=-\frac{m^2+\phi^{1/2}m}{2zm+z\phi^{1/2}+1-\phi}=\frac{r^2\xi^2}{(\phi^{-1/2}+r\xi)^2(r^2\xi^2-1)}
					\end{aligned}.$$ 
Together with the chain rule that $m^{\prime \prime}(z)\mathrm{d}z=\left(\frac{\partial m'}{\partial \xi}\right)\mathrm{d}\xi,$ it is easy to see that  
\begin{equation}\label{eq_changeofvariablebig}
\frac{m^{\prime\prime}(z)}{2m^{\prime}(z)}\mathrm{d}z=\left[\frac{2}{\xi}-\frac{2r}{\phi^{-1/2}+r\xi}-\frac{2r^2\xi}{r^2\xi^2-1}\right]\mathrm{d}\xi. 
\end{equation} 

We first with the mean part. According to Remark \ref{rem_contourform} and (\ref{eq_b(zdefinition)}), we decompose that 
\begin{equation}\label{eq_decompositionmean}
\mathsf{M}(f) \equiv \mathbb{E} \mathscr{G} :=\mathsf{M}_1(f)+\mathsf{M}_2(f), 
\end{equation}
where $\mathsf{M}_k, k=1,2,$ correspond to the $b_k$ part in (\ref{eq_b(zdefinition)}). On the one hand, using (\ref{eq_conformal}), (\ref{eq_changeofvariablebig}) and Remark \ref{rem_contourform} with calculations similar to  \cite[Section 3.2]{yao2015large}, we see that
			\begin{align}
				\mathsf{M}_1(f)=&\frac{-1}{2\pi\mathrm{i}}\oint_{\Gamma}f(z)\left(\frac{m^{\prime\prime}(z)}{2m^{\prime}(z)}-\frac{m^{\prime}(z)}{m(z)}\right)\mathrm{d} z \nonumber \\ 
				=&\frac{-1}{2\pi\mathrm{i}}\oint_{|\xi|=1}f({\phi^{1/2}+\phi^{-1/2} }+r\xi+\frac{1}{r\xi})\left(\frac{1}{\xi}-\frac{1}{2}\frac{1}{\xi+\frac{1}{r}}-\frac{1}{2}\frac{1}{\xi-\frac{1}{r}}\right)\mathrm{d} \xi  \nonumber \\
				=& \lim_{r \downarrow 1} \frac{-1}{2\pi\mathrm{i}}\oint_{|\xi|=1}f({\phi^{1/2}+\phi^{-1/2} }+\xi+\frac{1}{\xi})\left(\frac{1}{\xi}-\frac{1}{2}\frac{1}{\xi+\frac{1}{r}}-\frac{1}{2}\frac{1}{\xi-\frac{1}{r}}\right)\mathrm{d} \xi.
				 \label{m1fgeneralformula}
					\end{align}   
									
On the other hand, for $\mathsf{M}_2(f),$ with a discussion similar to (\ref{eq_changeofvariablebig}), we have that  
					$$
					\begin{aligned}
					\left[\frac{m^2(z)m''(z)}{2{m'(z)}^2}-m(z)\right]\mathrm{d}z=&\left[ \frac{1}{\xi}-\frac{1}{r^2\xi^3}-\frac{1}{\frac{\phi^{-1/2}}{r}+\xi} -\frac{1}{\xi}+\frac{1}{r^2\xi^2}\frac{1}{\xi+\frac{\phi^{-1/2}}{r}}+\frac{1}{\xi+\frac{\phi^{-1/2}}{r}}-\frac{1}{r^2\xi^2}\frac{1}{\xi+\frac{\phi^{-1/2}}{r}}\right]\mathrm{d}\xi\\ 
					=&-\frac{1}{r^2\xi^3}\mathrm{d}\xi.
					\end{aligned}$$					 	
Together with a discussion similar to the $\mathsf{M}_1(f)$ part, we see that   					
	\begin{align}\label{m2fgeneralformula}
						\mathsf{M}_2(f)
					& = \kappa_4 \frac{-1}{2\pi\mathrm{i}}\oint_{|\xi|=1}f({ \phi^{1/2}+\phi^{-1/2} }+r\xi+\frac{1}{r\xi})\left(-\frac{1}{r^2\xi^3}\right)\mathrm{d} \xi \nonumber \\
					& =  \kappa_4 \frac{-1}{2\pi\mathrm{i}}\oint_{|\xi|=1}f({ \phi^{1/2}+\phi^{-1/2} }+\xi+\frac{1}{\xi})\left(-\frac{1}{\xi^3}\right)\mathrm{d} \xi .
					\end{align}

We then work with the variance part. Similar to (\ref{eq_decompositionmean}), we also decompose that 
\begin{equation}\label{eq_vdecomposition}
\mathsf{V}(f) \equiv \operatorname{Var}(\mathscr{G}):=\mathsf{V}_1(f)+\mathsf{V}_2(f). 
\end{equation}
We again follow the ideas of \cite[Section 3]{yao2015large} and consider two non-overlapping contours defined by $z_j=\phi^{1/2}+\phi^{-1/2}+r_j\xi_j+r_j^{-1}\bar{\xi}_j, j=1,2,$ with  $r_2>r_1>1$. According to Remark \ref{rem_contourform} and (\ref{eq_basictransform}) that $m^{\prime}(z_j)\mathrm{d}z_j=\left(\frac{\partial }{\partial \xi_j}m(z_j)\right)\mathrm{d}\xi_j=\frac{r_j}{(\phi^{-1/2}+r_j\xi_j)^2}\mathrm{d}\xi_j,$ by a discussion similar to (\ref{m1fgeneralformula}),  we have that
				\begin{align}
			\mathsf{V}_1(f)=&-\frac{1}{2\pi^2}\oint_{|\xi_1|=1}\oint_{|\xi_2|=1}\frac{f(z_1)f(z_2)}{(m(z_1)-m(z_2))^2}\frac{r_1}{(\phi^{-1/2}+r_1\xi_1)^2}\frac{r_2}{(\phi^{-1/2}+r_2\xi_2)^2}\mathrm{d}\xi_1\mathrm{d}\xi_2 \nonumber  \\
			=&-\frac{1}{2\pi^2}\oint_{|\xi_1|=1}\oint_{|\xi_2|=1}\frac{f(z_1)f(z_2)}{(r_1\xi_1-r_2\xi_2)^2}r_1r_2\mathrm{d}\xi_1\mathrm{d}\xi_2 \label{eq_v1partnotice} \\ 
			=&-\frac{1}{2\pi^2}\oint_{|\xi_1|=1}\oint_{|\xi_2|=1}\frac{\frac{r_2}{r_1}f(\phi^{1/2}+\phi^{-1/2}+{r_1}\xi_1+\frac{1}{{r_1}}\xi_1^{-1})f(\phi^{1/2}+\phi^{-1/2}+{r_2}\xi_2+\frac{1}{{r_2}}\xi_2^{-1})}{(\xi_1-\frac{r_2}{r_1}\xi_2)^2}\mathrm{d}\xi_1\mathrm{d}\xi_2 \nonumber  \\
			=&-\frac{1}{2\pi^2} \lim_{\substack{r_2>r_1 \\ r_1, r_2 \downarrow 1}} \oint_{|\xi_1|=1}\oint_{|\xi_2|=1}\frac{f(\phi^{1/2}+\phi^{-1/2}+\xi_1+\xi_1^{-1})f(\phi^{1/2}+\phi^{-1/2}+\xi_2+\xi_2^{-1})}{(\xi_1-\frac{r_2}{r_1}\xi_2)^2}\mathrm{d}\xi_1\mathrm{d}\xi_2. \nonumber
						\end{align}
Similarly, using Remark \ref{rem_contourform} and the fact 
$$\frac{\partial}{\partial z_j}\left(\frac{1}{1+\phi^{-1/2}m(z_j)}\right)\mathrm{d}z_j=\frac{\partial}{\partial \xi_j}\left(\frac{1}{1+\phi^{-1/2}m(z_j)}(\xi_j)\right)\mathrm{d}\xi_j=-\frac{\phi^{-1/2}r_j^{-1}}{\xi_j^2}\mathrm{d}\xi_j, \quad j=1,2,$$
we readily obtain that 					
				\begin{align}\label{eq_v2fpart}
				\mathsf{V}_2(f)=&-\frac{\kappa_4}{4\pi^2}\oint_{|\xi_1|=1}f(\phi^{1/2}+\phi^{-1/2}+{r_1}\xi_1+\frac{1}{{r_1}}\xi_1^{-1})\frac{r_1^{-1}}{\xi_1^2}\mathrm{d}\xi_1\oint_{|\xi_2|=1}f(\phi^{1/2}+\phi^{-1/2}+{r_2}\xi_2+\frac{1}{{r_2}}\xi_2^{-1})\frac{r_2^{-1}}{\xi_2^2}\mathrm{d}\xi_2 \nonumber \\
				=& -\frac{\kappa_4}{4\pi^2}\left(\oint_{|\xi|=1}f(\phi^{1/2}+\phi^{-1/2}+\xi+\xi^{-1})\frac{1}{\xi^2}\mathrm{d}\xi\right)^2. 
						\end{align}
This completes the proof of part (1). \\

\noindent {\bf Proof of Part (2).} Armed with the results from Part (1), we now proceed to the calculations for the concrete examples. First, for $f_1(x)=x-(\phi^{1/2}+\phi^{-1/2}-c)$ and its mean part,  using  (\ref{m1fgeneralformula}) and (\ref{m2fgeneralformula}) with straightforward calculations and Cauchy's integral theorem, we see that 		
\begin{equation}\label{eq_M1fmeancontrol}
		\begin{aligned}
								\mathsf{M}_1(f_1)&=\frac{-1}{2\pi\mathrm{i}}\oint_{|\xi|=1}(c+r\xi+\frac{1}{r\xi})\left(\frac{1}{\xi}-\frac{1}{2}\frac{1}{\xi+\frac{1}{r}}-\frac{1}{2}\frac{1}{\xi-\frac{1}{r}}\right)\mathrm{d} \xi=0,\\ 
								\mathsf{M}_2(f_1)&=\frac{-1}{2\pi\mathrm{i}}\kappa_4\oint_{|\xi|=1}(c+r\xi+\frac{1}{r\xi})\left(-\frac{1}{r^2\xi^3}\right)\mathrm{d} \xi=0.
		\end{aligned}
		\end{equation}
For its variance part, according to (\ref{eq_v1partnotice}), we see that 
		\begin{equation*}
		\begin{aligned}
							\mathsf{V}_1(f_1)=-\frac{1}{2\pi^2}\oint_{|\xi_2|=1}(c+{r_2}\xi_2+\frac{1}{{r_2}}\xi_2^{-1})\cdot\oint_{|\xi_1|=1}\frac{c+{r_1}\xi_1+\frac{1}{{r_1}}\xi_1^{-1}}{(\xi_1-\frac{r_2}{r_1}\xi_2)^2}\mathrm{d}\xi_1\mathrm{d}\xi_2.
		\end{aligned}
	\end{equation*}
Moreover, it is easy to see that 
$$\begin{aligned}
			&\frac{1}{2 \pi\mathrm{i}}\oint_{|\xi_1|=1}\frac{c+{r_1}\xi_1+\frac{1}{{r_1}}\xi_1^{-1}}{(\xi_1-\frac{r_2}{r_1}\xi_2)^2}\mathrm{d}\xi_1\\ 
			=&\frac{1}{2 \pi\mathrm{i}}\oint_{|\xi_1|=1}\frac{c\xi_1+{r_1}\xi_1^2+\frac{1}{{r_1}}}{\xi_1(\xi_1-\frac{r_2}{r_1}\xi_2)^2}\mathrm{d}\xi_1
			=&\frac{1}{2 \pi\mathrm{i}}\left[\oint_{|\xi_1|=1}\frac{c}{(\xi_1-\frac{r_2}{r_1}\xi_2)^2}\mathrm{d}\xi_1+\oint_{|\xi_1|=1}\frac{{r_1}\xi_1}{(\xi_1-\frac{r_2}{r_1}\xi_2)^2}\mathrm{d}\xi_1+\oint_{|\xi_1|=1}\frac{\frac{1}{{r_1}}}{\xi_1(\xi_1-\frac{r_2}{r_1}\xi_2)^2}\mathrm{d}\xi_1 \right]\\ 
				=&\frac{1}{{r_1}(\frac{r_2}{r_1})^2\xi_2^2}.
			\end{aligned}$$
This yields that 
\begin{equation}\label{eq_v1f1}
	\begin{aligned}
			\mathsf{V}_1(f_1)=&\frac{2\frac{r_2}{r_1}}{2 \pi\mathrm{i} (\frac{r_2}{r_1})^2r_1}\oint_{|\xi_2|=1}\frac{c\xi_2+{r_2}\xi_2^2+\frac{1}{{r_2}}}{\xi_2^{3}}\mathrm{d}\xi_2={2}.
			\end{aligned}
\end{equation}					
Similarly, using (\ref{eq_v2fpart}), we can obtain that  		
\begin{equation*}
							\begin{aligned}
								&\mathsf{V}_2(f_1)\\ 
								=&-\frac{\kappa_4}{4\pi^2}\oint_{|\xi_1|=1}f(\phi^{1/2}+\phi^{-1/2}+{r_1}\xi_1+\frac{1}{{r_1}}\xi_1^{-1})\frac{r_1^{-1}}{\xi_1^2}\mathrm{d}\xi_1\oint_{|\xi_2|=1}f(\phi^{1/2}+\phi^{-1/2}+{r_2}\xi_2+\frac{1}{{r_2}}\xi_2^{-1})\frac{r_2^{-1}}{\xi_2^2}\mathrm{d}\xi_2\\ 
								=&\kappa_4.
							\end{aligned}
\end{equation*}
This concludes the proof for $f_1(x).$

Second, for $f_2(x)=[x-(\phi^{1/2}+\phi^{-1/2}-c)]^2,$ by a discussion similar to (\ref{eq_M1fmeancontrol}) using residual theorem, we see that 
\begin{equation*}
\mathsf{M}_1(f_2)=1, \ \mathsf{M}_2(f_2)=\kappa_4.
\end{equation*}
For its variance part, using an argument similar to (\ref{eq_v1f1}), it is not hard to see that 
		\begin{equation*}
							\begin{aligned}
								\mathsf{V}_1(f_2)&=-\frac{\frac{r_2}{r_1}}{2\pi^2}\oint_{|\xi_2|=1}((c+{r_2}\xi_2+\frac{1}{{r_2}}\xi_2^{-1})^2)\cdot\oint_{|\xi_1|=1}\frac{(c+{r_1}\xi_1+\frac{1}{{r_1}}\xi_1^{-1})^2}{(\xi_1-\frac{r_2}{r_1}\xi_2)^2}\mathrm{d}\xi_1\mathrm{d}\xi_2 \\
								&=2\frac{r_2}{r_1}\left[\frac{2c}{r_1(\frac{r_2}{r_1})^2}{2cr_2}+\frac{2}{r_1^2(\frac{r_2}{r_1})^{3}}{r_2^2}\right]=8c^2+4.
							\end{aligned}
						\end{equation*}
Analogously, we have that  						
\begin{equation*}
							\begin{aligned}
								\mathsf{V}_2(f_2)=&-\frac{\kappa_4}{4\pi^2}\oint_{|\xi_1|=1}f(\phi^{1/2}+\phi^{-1/2}+{r_1}\xi_1+\frac{1}{{r_1}}\xi_1^{-1})\frac{r_1^{-1}}{\xi_1^2}\mathrm{d}\xi_1\oint_{|\xi_2|=1}f(\phi^{1/2}+\phi^{-1/2}+{r_2}\xi_2+\frac{1}{{r_2}}\xi_2^{-1})\frac{r_2^{-1}}{\xi_2^2}\mathrm{d}\xi_1\\ 
								=&{4c^2\kappa_4}.
							\end{aligned}
\end{equation*}
This concludes the proof for $f_2(x).$				

Finally, for $f_3(x)=\log(x-\phi^{1/2}-\phi^{-1/2}+t+t^{-1}),$ for the mean part, we first decompose that	$$	\begin{aligned}
							\mathsf{M}_1(f_3)=&-\frac{1}{2\pi\mathrm{i}}\oint_{|\xi|=1}	\log \left( (t+{r} \xi)(1+\frac{1}{{r}t\xi}) \right)\left[\frac{1}{\xi}-\frac{1}{2}\frac{1}{\xi+\frac{1}{r}}-\frac{1}{2}\frac{1}{\xi-\frac{1}{r}}\right]\mathrm{d}\xi\\
							=&-\frac{1}{2\pi\mathrm{i}}\oint_{|\xi|=1}	\left[\log  (t+{r} \xi)+\log(1+\frac{1}{{r}t\xi}) \right]\left[\frac{1}{\xi}-\frac{1}{2}\frac{1}{\xi+\frac{1}{r}}-\frac{1}{2}\frac{1}{\xi-\frac{1}{r}}\right]\mathrm{d}\xi\\ 
							:=&\mathsf{M}_{11}(f_3)+\mathsf{M}_{12}(f_3),
						\end{aligned}$$
where 
						$$	\begin{aligned}
							\mathsf{M}_{11}(f_3)
							=&-\frac{1}{2\pi\mathrm{i}}\oint_{|\xi|=1}	\log  (t+{r} \xi)\left[\frac{1}{\xi}-\frac{1}{2}\frac{1}{\xi+\frac{1}{r}}-\frac{1}{2}\frac{1}{\xi-\frac{1}{r}}\right]\mathrm{d}\xi\\ 
							=&-\left[\log(t)-\frac{1}{2}\log(t-1)-\frac{1}{2}\log(t+1)\right]=\frac{1}{2}\log(1-\frac{1}{t^2}),
						\end{aligned}\\ 
						$$
						and 
						$$	\begin{aligned}
							\mathsf{M}_{12}(f_3)
							=&-\frac{1}{2\pi\mathrm{i}}\oint_{|\xi|=1}	\log(1+\frac{1}{{r}t\xi}) \left[\frac{1}{\xi}-\frac{1}{2}\frac{1}{\xi+\frac{1}{r}}-\frac{1}{2}\frac{1}{\xi-\frac{1}{r}}\right]\mathrm{d}\xi\\ 
							=&-\frac{1}{2\pi\mathrm{i}}\oint_{|z|=1}	\log(1+\frac{z}{{r}t}) \left[z-\frac{1}{2}\frac{1}{\frac{1}{z}+\frac{1}{r}}-\frac{1}{2}\frac{1}{\frac{1}{z}-\frac{1}{r}}\right]\frac{1}{z^2}\mathrm{d}z=0.
						\end{aligned}$$
This results in 	$\mathsf{M}_1(f_3)=\frac{1}{2}\log(1-\frac{1}{t^2}).$ Similarly, one can show that 	$$	\begin{aligned}
							\mathsf{M}_{2}(f_3)
							=-\frac{\kappa_4}{2\pi\mathrm{i}}\oint_{|\xi|=1}	\log  (t+{r} \xi)\left[-\frac{1}{r^2\xi^3}\right]\mathrm{d}\xi-\frac{1}{2\pi\mathrm{i}}\oint_{|\xi|=1}	\log(1+\frac{1}{{r}t\xi}) \left[-\frac{1}{r^2\xi^3}\right]\mathrm{d}\xi=-\frac{\kappa_4}{2t^2}.
						\end{aligned} 
						$$
For its variance part, using a straightforward calculation with Cauchy's differentiation formula and residual theorem, we have that  						
		$$			\begin{aligned}
							\mathsf{V}_1(f_3)& =2 \left[\frac{\frac{r_2}{r_1}}{2  \pi\mathrm{i} (\frac{r_2}{r_1})^2} \oint_{\left|\xi_2\right|=1} \frac{\log \left(t+{r_2} \xi_2\right)}{\xi_2\left(\frac{1}{\frac{r_2}{r_1}}+{r_1}t\xi_2\right)} \mathrm{d} \xi_2+\frac{\frac{r_2}{r_1}}{2  \pi\mathrm{i} \frac{r_2}{r_1}^2} \oint_{\left|\xi_2\right|=1} \frac{\log \left(1+ \xi_2^{-1}/({r_2}t)\right)}{\xi_2\left(\frac{1}{\frac{r_2}{r_1}}+{r_1}t\xi_2\right)} \mathrm{d} \xi_2 \right]  \\
							&=2 \left( \log t- \log (t-t^{-1}) \right). 
						\end{aligned}
						$$
Similarly, we have that 						
						\begin{equation*}
							\begin{aligned}
								\mathsf{V}_2(f_3)=&-\frac{\kappa_4}{4\pi^2}\oint_{|\xi_1|=1}f(\phi^{1/2}+\phi^{-1/2}+{r_1}\xi_1+\frac{1}{{r_1}}\xi_1^{-1})\frac{r_1^{-1}}{\xi_1^2}\mathrm{d}\xi_1\oint_{|\xi_2|=1}f(\phi^{1/2}+\phi^{-1/2}+{r_2}\xi_2+\frac{1}{{r_2}}\xi_2^{-1})\frac{r_2^{-1}}{\xi_2^2}\mathrm{d}\xi_1\\ 
								=&\kappa_4/t^2.
							\end{aligned}
						\end{equation*}
This completes the proof. 

\subsection{Proof of Corollary \ref{cor_localexample}}

As in Section \ref{sec_goodandbadregion}, we focus our discussion on a single function. The multiple functions setting can be handled in the same ways as in the end of Section \ref{sec_goodandbadregion}. \vspace*{0.1in}

\noindent{\bf Simplification of variance.} Similar to (\ref{eq_reducedreducedreducedform}), we rewrite 
\begin{equation*}
\operatorname{Var}(\mathcal{Z}_{\mathrm{l}})=(\mathcal{K}_2)_{++}+(\mathcal{K}_2)_{--}-(\mathcal{K}_2)_{+-}-(\mathcal{K}_2)_{-+}.
\end{equation*}
Here we use $\mathcal{Z}_{\mathrm{l}} \equiv \mathcal{Z}_{\eta_0, E}$ when $\eta_0=\mathrm{o}(1).$ In view of (\ref{eq_mainalphabeta}) and (\ref{eq_controlcontrolsmall}), we can write 
  \begin{align}\label{eq)fafafafa}
  	(\mathcal{K}_2)_{++}	  &= -\frac{\eta_0^2}{4 \pi^2 } \int\int  {g}\left(\tilde{x}_1\right)  {g}\left(\tilde{x}_2\right)\widehat{\beta}\left(E+\tilde{x}_1 {\eta_0}+\mathrm{i}\underline{\eta}_1, E+\tilde{x}_2 {\eta_0}+\mathrm{i}\underline{\eta}_2\right) \mathrm{d} \tilde{x}_1 \mathrm{d} \tilde{x}_2  \nonumber \\ 
  &	=	-\frac{{\eta_0}^2}{2 \pi^2} \iint_{1 \le \tilde{y}_{1,2} \le 2} g\left(\tilde{x}_1\right) g\left(\tilde{x}_2\right) \left[ \frac{m^{\prime}\left(\left(E+\tilde{x}_1 {\eta_0}\right)+\mathrm{i}  \underline{\eta}_1\right)m^{\prime}\left(\left(E+\tilde{x}_2 {\eta_0}\right)+\mathrm{i}  \underline{\eta}_2\right)}{\left(m\left(\left(E+\tilde{x}_1 {\eta_0}\right)+\mathrm{i}  \underline{\eta}_1\right)-m\left(\left(E+\tilde{x}_2 {\eta_0}\right)+\mathrm{i}  \underline{\eta}_2\right)\right)^2}\right. \nonumber \\
  			& \left.  -\frac{1}{\left(\left(\tilde{x}_1-\tilde{x}_2\right){\eta_0}+\mathrm{i}\frac{1}{2}{\eta_0} n^{-\varepsilon_1}\right)^2} \right] \mathrm{d} \tilde{x}_1  \mathrm{d} \tilde{x}_2 \nonumber \\
  			&:= \left(\mathcal{K}_{2,1}\right)_{++}+\left(\mathcal{K}_{2,2}\right)_{++},
  \end{align}
  where we used the definition of $\widehat{\beta}$ as in (\ref{eq_moreconvention}). 
  
 We first reformulate the case in the bulk $E \in\left(\gamma_{-}+\tau^{\prime}, \gamma_{+}-\tau^{\prime}\right)$ for some fixed small constant $\tau^{\prime}>0$. Note that $m$ is holomorphic in $\mathbb{C}\backslash [\gamma_-,\gamma_+]$. Since we are dealing with $z_1,z_2$ in the same upper half plane, we can use the taylor expansion at $z_1$ that
 \begin{equation}\label{eq_kerneltaylor}
  	\begin{aligned}
  		\frac{m'(z_1)m'(z_2)}{(m(z_1)-m(z_2))^2}-\frac{1}{(z_1-z_2)^2}=&\frac{m'(z_1)(m'(z_1)+(z_2-z_1)m''(z_1)+\mathrm{O}(|z_1-z_2|^2))}{((z_2-z_1)m'(z_1)+\mathrm{O}(|z_1-z_2|^2))^2}-\frac{1}{(z_1-z_2)^2}\\
  		=&\frac{m''(z_1)}{(z_2-z_1)m'(z_1)+\mathrm{O}(|z_1-z_2|^2)}+\mathrm{O}(1).
  	\end{aligned}
  \end{equation}
  Similarly we have
  $$\begin{aligned}
  	\frac{m'(z_1)m'(z_2)}{(m(z_1)-m(z_2))^2}-\frac{1}{(z_1-z_2)^2}
  	=&\frac{m'(z_2)(m'(z_2)+(z_1-z_2)m''(z_2)+\mathrm{O}(|z_1-z_2|^2))}{((z_1-z_2)m'(z_2)+\mathrm{O}(|z_1-z_2|^2))^2}-\frac{1}{(z_1-z_2)^2}\\
  	=&\frac{m''(z_2)}{(z_1-z_2)m'(z_2)+\mathrm{O}(|z_1-z_2|^2)}+\mathrm{O}(1).
  \end{aligned}$$
 This implies that 
 \begin{equation}\label{eq_samehalfplanelocalcontrol}
  	\widehat{\beta}(z_1,z_2)\sim \frac{1}{z_2-z_1}\left(\frac{m''(z_1)}{m'(z_1)}-\frac{m''(z_2)}{m'(z_2)}\right).
  \end{equation}
Using the above control,  we can conclude that $${\eta_0}^2\left[ \frac{m^{\prime}\left(\left(E+\tilde{x}_1 {\eta_0}\right)+\mathrm{i}  \underline{\eta}_1\right)m^{\prime}\left(\left(E+\tilde{x}_2 {\eta_0}\right)+\mathrm{i}  \underline{\eta}_2\right)}{\left(m\left(\left(E+\tilde{x}_1 {\eta_0}\right)+\mathrm{i}  \underline{\eta}_1\right)-m\left(\left(E+\tilde{x}_2 {\eta_0}\right)+\mathrm{i}  \underline{\eta}_2\right)\right)^2} -\frac{1}{\left(\left(\tilde{x}_1-\tilde{x}_2\right){\eta_0}+\mathrm{i}\frac{1}{2}{\eta_0} n^{-\varepsilon_1}\right)^2} \right]=\mathrm{O}({\eta_0}^2).$$  
This shows that $(\mathcal{K}_2)_{++}$ is negligible and similar discussions and results apply to $(\mathcal{K}_2)_{--}.$ 

When $z_1$ and $z_2$ are in different half planes, we have that $$\left|m^{\prime}\left(\left(E+\tilde{x}_1 {\eta_0}\right)+\mathrm{i}  \underline{\eta}_1\right)\right|, \left|m^{\prime}\left(\left(E+\tilde{x}_2 {\eta_0}\right)+\mathrm{i}  \underline{\eta}_2\right) \right|\lesssim \frac{1}{\sqrt{\tau^{\prime}}},$$ and $$\left| \left(m\left(\left(E+\tilde{x}_1 {\eta_0}\right)+\mathrm{i}  \underline{\eta}_1\right)-m\left(\left(E+\tilde{x}_2 {\eta_0}\right)+\mathrm{i}  \underline{\eta}_2\right)\right) \right|\sim \rho(E+\tilde{x}_1\eta_0)+\rho(E+\tilde{x}_2\eta_0)\sim 1.$$ 
We can therefore conclude that $(\mathcal{K}_{2,1})_{+-}=\mathrm{O}(\eta_0^2/\tau^{\prime})$ and $(\mathcal{K}_{2,1})_{-+}=\mathrm{O}(\eta_0^2/\tau^{\prime}).$   For the remaining term $(\mathcal{K}_{2,2})_{+-}+(\mathcal{K}_{2,2})_{-+}$, we have  $$-\frac{1}{\pi^2}\iint g\left(\tilde{x}_1\right) g\left(\tilde{x}_2\right)  \frac{1}{\left(\tilde{x}_1-\tilde{x}_2\pm \mathrm{i}\frac{3}{2}\eta_0n^{-\varepsilon_1}\right)^2} \mathrm{d} \tilde{x}_1 \mathrm{d} \tilde{x}_2.$$
Note that we can add additional $\frac{1}{2\pi^2}\iint   \frac{g\left(\tilde{x}_1\right)^2}{\left(\tilde{x}_1-\tilde{x}_2\pm \mathrm{i}\frac{3}{2}n^{-\varepsilon_1}\right)^2} \mathrm{d} \tilde{x}_1 \mathrm{d} \tilde{x}_2$ and $\frac{1}{2\pi^2}\iint   \frac{g\left(\tilde{x}_2\right)^2}{\left(\tilde{x}_1-\tilde{x}_2\pm \mathrm{i}\frac{3}{2}n^{-\varepsilon_1}\right)^2} \mathrm{d} \tilde{x}_1 \mathrm{d} \tilde{x}_2$ to the above terms since they are negligible. Then we can conclude the proof using dominated convergence theorem. 

Then we reformulate the case at the edge. Due to similarity, we focus on the right edge $\gamma_+.$  Note that near the edge, for some constants $c_1, c_2$ and $c_3=c_1/2,$ we have the following approximation
\begin{equation}\label{eq_m(z)elementarybound}
  	m(z)=m(\gamma_+)+c_1\sqrt{z-\gamma_+}+\mathrm{O} \left( |z-\gamma_+|\right), \  m'(z)=c_2+\frac{c_3}{\sqrt{z-\gamma_+}}+\mathrm{O}(\sqrt{|z-\gamma_+|}). 
\end{equation}
For the terms in (\ref{eq)fafafafa}) and counterparts in $(\mathcal{K}_2)_{--}, (\mathcal{K}_2)_{-+}$ and $(\mathcal{K}_2)_{+-}$, it is not hard to see that 
\begin{equation*}
(\mathcal{K}_{2,2})_{++}+(\mathcal{K}_{2,2})_{--}+(\mathcal{K}_{2,2})_{+-}+(\mathcal{K}_{2,2})_{-+}=0.
\end{equation*}
Then it suffices to control $(\mathcal{K}_{2,1})_{++}, (\mathcal{K}_{2,1})_{--}, (\mathcal{K}_{2,1})_{-+}$ and $(\mathcal{K}_{2,1})_{+-}.$ Note that
  	\begin{equation*}
  		\begin{aligned}
  			&(\mathcal{K}_{2,1})_{++}=-\frac{1}{8\pi^2}\iint\frac{g(\tilde{x_1})g(\tilde{x}_2)\left(\frac{1}{\sqrt{\tilde{x}_1+\mathrm{i}n^{-\varepsilon_1}}}+\mathrm{O}(\sqrt{{\eta_0}})\right) \left(\frac{1}{\sqrt{\tilde{x}_2+\mathrm{i}n^{-\varepsilon_1}/2}}+\mathrm{O}(\sqrt{{\eta_0}})\right)}{\left(\sqrt{\tilde{x}_1+\mathrm{i}n^{-\varepsilon_1}}-\sqrt{\tilde{x}_2+\mathrm{i}n^{-\varepsilon_1}/2}+\mathrm{O}(\sqrt{{\eta_0}})\right)^2}\mathrm{d}\tilde{x}_1\mathrm{d}\tilde{x}_2\\ 
  			&=-\frac{1}{8\pi^2}\iint\frac{g(\tilde{x_1})g(\tilde{x}_2)}{\sqrt{\tilde{x}_1+\mathrm{i}n^{-\varepsilon_1}}\sqrt{\tilde{x}_2+\mathrm{i}n^{-\varepsilon_1}/2}\left(\sqrt{\tilde{x}_1+\mathrm{i}n^{-\varepsilon_1}}-\sqrt{\tilde{x}_2+\mathrm{i}n^{-\varepsilon_1}}\right)^2}\mathrm{d}\tilde{x}_1\mathrm{d}\tilde{x}_2+\mathrm{O}(\sqrt{{\eta_0}}n^{3\varepsilon_1}).
  		\end{aligned}
  	\end{equation*}
  	Similar to the discussion in the bulk, we could add vanishing $g(\tilde{x}_1)^2$ and $g(\tilde{x}_1)^2$ terms by consider the corresponding complex integral and use dominated convergence theorem to show that 
  	\begin{equation}
  		\begin{aligned}
  		(\mathcal{K}_{2,1})_{++}=&\frac{1}{16\pi^2}\iint\frac{(g(\tilde{x}_1)-g(\tilde{x}_2))^2}{\sqrt{\tilde{x_1}+\mathrm{i}0}\sqrt{\tilde{x}_2+\mathrm{i}0}(\sqrt{\tilde{x_1}+\mathrm{i}0}-\sqrt{\tilde{x_2}+\mathrm{i}0})^2}\mathrm{d}\tilde{x}_1\mathrm{d}\tilde{x}_2+\mathrm{o}(1)\\ 
  			=&\frac{1}{4 \pi^2} \int_{\psi(\mathbb{R}+\mathrm{i} 0)} \int_{\psi(\mathbb{R}+\mathrm{i} 0)} \frac{\left(g\left(w_1^2\right)-g\left(w_2^2\right)\right)^2}{\left(w_1-w_2\right)^2} \mathrm{~d} w_1 \mathrm{~d} w_2+\mathrm{o}(1),
  		\end{aligned}
  	\end{equation}
  	where we used the change the variable $\psi(z):=\sqrt{z} $ with branch cut such that $\psi: \mathbb{C}^{+} \rightarrow \mathbb{C}^{+}$. Similarly, we also have
  	$$
  	\begin{aligned}
  		(\mathcal{K}_{2,1})_{--} & =\frac{1}{4 \pi^2} \int_{\psi(\mathbb{R}-\mathrm{i} 0)} \int_{\psi(\mathbb{R}-\mathrm{i} 0)} \frac{\left(g\left(w_1^2\right)-g\left(w_2^2\right)\right)^2}{\left(w_1-w_2\right)^2} \mathrm{d} w_1 \mathrm{d} w_2+\mathrm{o}(1), \\
  		(\mathcal{K}_{2,1})_{+-} & =\frac{1}{4  \pi^2} \int_{\psi(\mathbb{R}+\mathrm{i} 0)} \int_{\psi(\mathbb{R}-\mathrm{i} 0)} \frac{\left(g\left(w_1^2\right)-g\left(w_2^2\right)\right)^2}{\left(w_1-w_2\right)^2} \mathrm{d} w_1 \mathrm{d} w_2+\mathrm{o}(1), \\
  		(\mathcal{K}_{2,1})_{-+} &=\frac{1}{4 \pi^2} \int_{\psi(\mathbb{R}-\mathrm{i} 0)} \int_{\psi(\mathbb{R}+\mathrm{i} 0)} \frac{\left(g\left(w_1^2\right)-g\left(w_2^2\right)\right)^2}{\left(w_1-w_2\right)^2} \mathrm{d} w_1 \mathrm{d} w_2 +\mathrm{o}(1).
  	\end{aligned}
  	$$
This yields that $$ \mathcal{K}_{2,1}=\frac{1}{4 \pi^2} \int_{-\mathrm{i}\infty}^{\mathrm{i}\infty} \int_{-\mathrm{i}\infty}^{\mathrm{i}\infty} \frac{\left(g\left(w_1^2\right)-g\left(w_2^2\right)\right)^2}{\left(w_1-w_2\right)^2} \mathrm{d} w_1 \mathrm{d} w_2 +\mathrm{o}(1)=\frac{1}{4 \pi^2} \int_{\mathbb{R}}\int_{\mathbb{R}}  \frac{\left(g\left(-w_1^2\right)-g\left(-w_2^2\right)\right)^2}{\left(w_1-w_2\right)^2} \mathrm{d} w_1 \mathrm{d} w_2+\mathrm{o}(1). $$

\noindent {\bf Simplification of mean.}  When $\operatorname{Re}z$ is in the bulk, we have that $|m(z)|\sim 1,\ {m}^{\prime}(z) \sim (\kappa+{\eta})^{-1/2}.$  Using the change of variable as in (\ref{eq_changeofvariable}) and a discussion similar to (\ref{eq_controlcontrolsmall}), we have that   
	$$\frac{1}{2\pi\mathrm{i}}\left(\int_{\mathbb{R}} f(x)b_1^{\mathfrak{+}}(x) \mathrm{d} x-\int_{\mathbb{R}} f(x)b_1^{\mathfrak{-}}(x) \mathrm{d} x\right)=\frac{\eta_0}{2\pi\mathrm{i}}\left(\int_{\mathbb{R}} g(\tilde{x})b_1^{\mathfrak{+}}(E+\tilde{x}\eta_0) \mathrm{d} \tilde{x}-\int_{\mathbb{R}} g(\tilde{x})b_1^{\mathfrak{-}}(E+\tilde{x}\eta_0) \mathrm{d} \tilde{x}\right).$$
Furthermore, by definition and Lemma \ref{lem:mphizabs}, we have that $b_1(x)=\mathrm{O}\left(1+n^{\varepsilon_1/2}\eta_0^{-1/2}\right).$ This shows that the mean is asymptotically $0$ in the bulk. 

Next, at the  edge, using (\ref{eq_m(z)elementarybound}) and  
$$m^{\prime \prime}(z)=-\frac{c_{3}}{2\left(\sqrt{z-\gamma_+}\right)^3}+\mathrm{O}\left(\left|z-\gamma_+\right|^{-1/2}\right).$$
Consequently, we have that 
			$$\left(\frac{m^{\prime\prime}(z)}{2m^{\prime}(z)}-\frac{m^{\prime}(z)}{m(z)}\right)+\kappa_4 \left(\frac{m^2(z)m''(z)}{2({m'}(z))^2}-m(z)\right)=\frac{1}{4(z-\gamma_+)}+\mathrm{O}\left(\left|z-\gamma_+\right|^{-1/2}\right).$$
Inserting the above estimate into (\ref{eq_meanlocal}), by Lemma \ref{lem_SPformula}, we have that the mean is asymptotically $g(0)/4.$


\subsection{Proof of Corollary \ref{coro_statisticalapplications}}

 For the global statistics, (\ref{eq_globalnullone}) and (\ref{eq_globalnulltwo}) follow directly from (2) of Corollary \ref{cor_globalexamples} and (\ref{eq_generalmeanform}). For (\ref{eq_globalnullthree}), we can prove it using Delta method, i.e., if $\mathfrak{g}^{\prime}(\mathbf{a})$ exists and $n^b\left(\bm{x}_n-\mathbf{a}\right) \Rightarrow \bm{x}$ for some fixed constant $b>0$, then $n^b\left( \mathfrak{g} \left(\bm{x}_n\right)-\mathfrak{g}(\mathbf{a})\right) \Rightarrow [\mathfrak{g}^{\prime}(\mathbf{a})]^* \bm{x}.$  In light of (\ref{eq_globalnullone}) and the definition of $\mathcal{T}_4^{\mathrm{g}},$ we set
$$\mathfrak{g}(y_1,y_2)=\frac{y_2}{y_1^2}, \ \bm{x}_n=(\mathcal{T}_1^{\mathrm{g}}, \mathcal{T}_2^{\mathrm{g}})'.$$ 
Therefore, in order to apply Delta method, it suffices to find the distribution of $\bm{x}_n$ using Theorem \ref{thm_mainone} under $\Sigma=I.$ Recall (\ref{eq_vdecomposition}). Following lines of the arguments in Section \ref{proofofcorollary34}, we have that 
					$$\begin{aligned}
						\mathsf{V}_1(f_1,f_2)=&4c,\quad
						\mathsf{V}_2(f_1,f_2)=2c\kappa_4,
					\end{aligned}$$ 
so that $\mathsf{V}(f_1,f_2)=4c+2c \kappa_4.$ Then together with Theorem \ref{thm_mainone} and (2) of Corollary \ref{cor_globalexamples}, we see from (\ref{eq_globalnullone}) that the distribution of $\bm{x}_n$ follows
\begin{equation*}
\begin{pmatrix}
\mathcal{T}_1^{\mathrm{g}}-n(c-\phi^{-1/2}) \\
\mathcal{T}_2^{\mathrm{g}}-n (1+c^2-2c \phi^{-1/2}+\phi^{-1})
\end{pmatrix}
\Rightarrow 
\mathcal{N}_2 \left(
\begin{pmatrix}
0 \\
1+\kappa_4
\end{pmatrix},  
\begin{pmatrix}
2+\kappa_4 & 4c+2c \kappa_4 \\
4c+2c \kappa_4 &4+4c^2(\kappa_4+2)
\end{pmatrix}
\right). 
\end{equation*}						
Then the proof follows from Delta method with straightforward calculations.

For the local statistics, (\ref{eq_locallnullone}) follows directly from (2) of Corollary \ref{cor_localexample}. The proof of (\ref{eq_locallnulltwo}) follows from the Delta method as for (\ref{eq_globalnullthree}). This completes the proof.   					

\subsection{Proof of Corollary  \ref{cor_poweranalysis}}

Note that (\ref{eq_realevalue}), the centering depends on $\Sigma$ under $\mathbf{H}_a$ in (\ref{alternative}) so that the definitions in (\ref{eq_gxdefinition}) need to be modified accordingly. In the proof, we focus on a single test function and let $f_0$ be defined as in (\ref{eq_gxdefinition}) under $\mathbf{H}_0$ and $f_a$ be that under $\mathbf{H}_a.$ Moreover, under $\mathbf{H}_a,$ we denote 
\begin{equation*}
\mathrm{m}_{\mathfrak{d}}:=\mathbb{E}\left( Z_{\eta_0,E}(f_\mathfrak{d})  \right), \ \mathrm{v}_\mathfrak{d}=\operatorname{Var} \left( Z_{\eta_0,E}(f_\mathfrak{d}) \right),  \ \mathfrak{d}=0, \  a.  
\end{equation*}
We also denote $\mathrm{m}':=\mathbb{E}\left( Z_{\eta_0,E}(f_0)  \right)$ under $\mathbf{H}_0.$
\begin{proof}[\bf Proof]
Note that under $\mathbf{H}_a,$ we have that 
\begin{equation*}
\frac{Z_{\eta_0,E}(f_a)-\mathrm{m}_a}{\sqrt{\mathrm{v}_a}} \Rightarrow \mathcal{N}(0,1). 
\end{equation*}
Moreover, under the nominal level $\alpha,$ we shall reject $\mathbf{H}_a$ if 
\begin{equation*}
\frac{Z_{\eta_0,E}(f_0)-\mathrm{m}_0}{\sqrt{\mathrm{v}_0}}>z_{1-\alpha/2} \ \text{or} \ \frac{Z_{\eta_0,E}(f_0)-\mathrm{m}_0}{\sqrt{\mathrm{v}_0}}<z_{\alpha/2}.
\end{equation*}
Due to similarity, we focus on the left tail and the power can be written as 
\begin{equation*}
\mathsf{Po}:=\mathbb{P} \left(  \frac{Z_{\eta_0,E}(f_0)-\mathrm{m}_0}{\sqrt{\mathrm{v}_0}}<z_{\alpha/2}| \mathbf{H}_a \right).
\end{equation*}
After a straightforward calculation, we can rewrite 
\begin{equation*}
\mathsf{Po}=\mathbb{P}\left( \frac{Z_{\eta_0,E}(f_a)-\mathrm{m}_a}{\sqrt{\mathrm{v}_a}}  < T_1+T_2+T_3+T_4 | \mathbf{H}_a \right),
\end{equation*}
where $T_i, 1 \leq i \leq 4$ are defined as 
\begin{equation*}
T_1:=z_{\alpha/2}\frac{\sqrt{\mathrm{v}_0}}{\sqrt{\mathrm{v}_a}}, \ T_2:=\frac{\mathrm{m}'-\mathrm{m}_0}{\sqrt{\mathrm{v}_a}}, \ T_3:=-\frac{ \sum_{i=1}^n(f_{0}(\lambda_i)-f_{a}(\lambda_i)) +     n(\int f_{a}\mathrm{d}\varrho_{a}-\int f_{0}\mathrm{d}\varrho_{a})+\mathrm{m}_a-\mathrm{m}_0}{\sqrt{\mathrm{v}_a}},
\end{equation*}
\begin{equation*}
T_4:= -\frac{n(\int f_{0}\mathrm{d}\varrho_{a}-\int f_{0}\mathrm{d}\varrho_{0})}{\sqrt{\mathrm{v}_a}}.
\end{equation*}
It is not hard to see that $T_i=\mathrm{O}(1), 1 \leq i \leq 3.$ Consequently, to have $\mathsf{Po} \rightarrow 1,$ we require $T_4 \rightarrow \infty$ for various test functions. Due to similarity, we only focus on the setting when $\mathsf{sy}=\mathrm{g}$ and $\ell=1,2.$ 

When $\ell=1,$ if $(\mathfrak{m}_1(\pi_{a})-1)\sqrt{\phi}\gg1,$ we a=have that $\int f_{0}\mathrm{d}\varrho_{a}-\int f_{0}\mathrm{d}\varrho_{0}=\int (f_{0}-f_{a})\mathrm{d}\varrho_{a}+(\int f_{a}\mathrm{d}\varrho_{a}-\int f_{0}\mathrm{d}\varrho_{0})=(\mathfrak{m}_1(\pi_{a})-1)\sqrt{\phi}+\mathrm{O}(1).$ This yields that 
\begin{equation*}
T_4 \asymp n \phi^{1/2} (\mathfrak{m}_1(\pi_{a})-1). 
\end{equation*}
Furthermore, $(\mathfrak{m}_1(\pi_{a})-1)\sqrt{\phi}=\mathrm{O}(1),$ we still have $\int f_{0}\mathrm{d}\varrho_{a}-\int f_{0}\mathrm{d}\varrho_{0}=\int (f_{0}-f_{a})\mathrm{d}\varrho_{a}+(\int f_{a}\mathrm{d}\varrho_{a}-\int f_{0}\mathrm{d}\varrho_{0})=(\mathfrak{m}_1(\pi_{a})-1)\sqrt{\phi}+\mathrm{O}((\mathfrak{m}_1(\pi_{a})-1)).$ This verifies (\ref{eq_conditiononeoeoneonenenen}) for $\ell=1$. 
		
When $\ell=2,$ we decompose that  $\int f_{0}\mathrm{d}\varrho_{a}-\int f_{0}\mathrm{d}\varrho_{0}=\int (f_{0}-f_{a})\mathrm{d}\varrho_{a}+(\int f_{a}\mathrm{d}\varrho_{a}-\int f_{0}\mathrm{d}\varrho_{0}).$ Using 
		$$ \mathrm{d}\varrho_{a}(x)= \left( \frac{\sqrt{4\mathfrak{m}_2(\pi_{a})^2-(x-\mathfrak{m}_1(\pi_{a})\sqrt{\phi})^2}}{2\pi\mathfrak{m}_2(\pi_{a})}+\mathrm{O}(1/\sqrt{\phi})\right) \mathrm{d}x,$$
		we have that 
		$$\int x\mathrm{d}\varrho_{a}=\mathfrak{m}_1(\pi_{a})\sqrt{\phi}+\mathrm{O}(\mathfrak{m}_1(\pi_{a})),$$
		$$\int x^2\mathrm{d}\varrho_{a}=\mathfrak{m}_2(\pi_{a})^2+2(\mathfrak{m}_1(\pi_{a})\sqrt{\phi})^2-(\mathfrak{m}_1(\pi_{a})\sqrt{\phi})^2+\mathrm{O}(\mathfrak{m}_1(\pi_{a})^2\sqrt{\phi})=\mathfrak{m}_2(\pi_{a})^2+\mathfrak{m}_1(\pi_{a})^2\phi+\mathrm{O}(\mathfrak{m}_1(\pi_{a})^2\sqrt{\phi}),$$
		$$\begin{aligned}
			&\int(x-\sqrt{\phi}-1/\sqrt{\phi}+c)^2\mathrm{d}\varrho_{a}\\ =&\mathfrak{m}_1(\pi_{a})^2\phi-2\sqrt{\phi}\mathfrak{m}_1(\pi_{a})\sqrt{\phi}+\phi+(c-1/\sqrt{\phi})^2+{ \mathfrak{m}_2(\pi_{a})^2}+\mathrm{O}((\mathfrak{m}_1(\pi_{a})-1)^2\sqrt{\phi})\\ 
			=&(\mathfrak{m}_1(\pi_{a})-1)^2\phi+(c-1/\sqrt{\phi})^2+{ \mathfrak{m}_2(\pi_{a})^2}+2c(\mathfrak{m}_1(\pi_{a})-1)\sqrt{\phi}+\mathrm{O}((\mathfrak{m}_1(\pi_{a})-1)^2\sqrt{\phi}).
		\end{aligned} $$  
		Meanwhile, we have for $\ell=2,$ 
		$$\int f_{0}\mathrm{d}\varrho_{0}=1+c^2-2c\phi^{-1/2}+\phi^{-1}.$$
Combining all the above controls, we have that 
		$$\begin{aligned}
			T_4 \asymp n \left((\mathfrak{m}_1(\pi_{a})-1)^2{\phi}+2c(\mathfrak{m}_1(\pi_{a})-1)\sqrt{\phi}+(\mathfrak{m}_2(\pi_{a})^2-1)+\mathrm{O}((\mathfrak{m}_1(\pi_{a})-1)^2\sqrt{\phi})\right).
		\end{aligned} $$
This verifies (\ref{eq_conditiontwo}) for $\ell=2.$ The other cases can be handled similarly. This completes our proof.

\end{proof}

\appendix

\section{Proof of Theorem \ref{thm_locallaw}}\label{appendix_proof54}
In this section, we prove the main ingredient, Theorem \ref{thm_locallaw} following \cite{LL, AILL} with necessary modifications. In Section \ref{app_sec_tool}, we summarize the basic tools. In Section \ref{app_sec_entrywiselocalaw}, we prove the entrywise local law when $\mathbf{u}_k, \mathbf{v}_k$ in (\ref{eq_locallawequationone}) and (\ref{eq_locallawequationtwo}) are standard basis. We also prove (\ref{eq_averagelocallawone}), (\ref{eq_averagelocallawtwo}) and (\ref{eq:entrywiselaw}).  In Section \ref{app_sec_prooflocallaw}, armed with the entrywise local law, we prove (\ref{eq_locallawequationone}) and (\ref{eq_locallawequationtwo}) under the general setting.

\subsection{Some tools}\label{app_sec_tool}
In this section, we provide some technical preparation. Recall (\ref{eq_twoblockmatrices}). We denote 
\begin{equation}\label{eq_G1G2}
G_1(z):=z \Sigma^{1/2} R_1(z) \Sigma^{1/2}, \ G_2(z):=R_2(z). 
\end{equation}
It is not hard to see that by Schur's complement, 
	\begin{equation}\label{eq:Gblock}
		G=\left(\begin{array}{cc}
			\Sigma X G_2 X^* \Sigma-\Sigma & \Sigma X G_2 \\
			G_2 X^* \Sigma & G_2
		\end{array}\right)=\left(\begin{array}{cc}
			G_1 & z^{-1} G_1 X \\
			z^{-1} X^* G_1 & z^{-2} X^* G_1 X-z^{-1}
		\end{array}\right) .
	\end{equation}
Following Definition 4.1 of \cite{AILL}, for matrices of the form $A=(A_{st}: s \in l(A), t \in r(A)),$ whose entries are indexed by arbitrary finite subsets of $l(A), r(A) \subset \mathbb{N},$ then the matrix multiplication $AB$ is defined for $s \in l(A)$ and $t \in r(B)$ by 
\begin{equation*}
(AB)_{st}:=\sum_{q \in l(A) \cap r(B)} A_{sq} B_{qt}. 
\end{equation*}  
Recall (\ref{eq_BIGG}). Let $H \equiv H(z):=G^{-1}(z).$ For $S \subset \mathcal{I}$ we define the minor $H^{(S)}:=\left(H_{s t}: s, t \in \mathcal{I} \backslash S\right).$ Note that $H^{(S)}$ is also an $(p+n) \times (p+n)$ matrix where the undefined entries are zeros. We also write $G^{(S)}:=\left(H^{(S)}\right)^{-1}.$ The matrices $G_1^{(S)}$ and $G_2^{(S)}$ can be defined similarly.  Throughout the paper, we abbreviate $(\{s\}) \equiv(s)$ and $(\{s, t\}) \equiv(s t)$. The following identities will be frequently used in our calculations. 
{
\begin{lem}\label{lem_schurcomplement} Recall Definition \ref{eq_indexset}, (\ref{eq_BIGG}), (\ref{eq_G1G2}) and the assumption that $\Sigma$ is diagonal.  We have that 
\begin{enumerate}
\item[(1).] For $\mu \in \mathcal{I}_2$ we have
	\begin{equation*}
		\frac{1}{G_{\mu \mu}}=-z-\left(X^* G^{(\mu)} X\right)_{\mu \mu},
	\end{equation*}
	and for $\mu \neq \nu \in \mathcal{I}_2$
	\begin{equation*}
		G_{\mu \nu}=-G_{\mu \mu}\left(X^* G^{(\mu)}\right)_{\mu \nu}=-G_{\nu \nu}\left(G^{(\nu)} X\right)_{\mu \nu}=G_{\mu \mu} G_{\nu \nu}^{(\mu)}\left(X^* G^{(\mu \nu)} X\right)_{\mu \nu} .
	\end{equation*}
\item[(2).] For $i \in \mathcal{I}_1$ we have
	\begin{equation}\label{eq:3.8}
		\frac{1}{G_{i i}}=-\frac{1}{\sigma_i}-\left(X G^{(i)} X^*\right)_{i i},
	\end{equation}
	and for $i \neq j \in \mathcal{I}_1$
	\begin{equation}\label{eq:3.9}
		G_{i j}=-G_{i i}\left(X G^{(i)}\right)_{i j}=-G_{j j}\left(G^{(j)} X^*\right)_{i j}=G_{i i} G_{j j}^{(i)}\left(X G^{(i j)} X^*\right)_{i j} .
	\end{equation}
	\item[(3).] For $i \in \mathcal{I}_1$ and $\mu \in \mathcal{I}_2$ we have,
	\begin{equation*}
		G_{i \mu}=-G_{\mu \mu}\left(G^{(\mu)} X\right)_{i \mu}, \quad G_{\mu i}=-G_{\mu \mu}\left(X^* G^{(\mu)}\right)_{\mu i},
	\end{equation*}
and
	\begin{equation*}
		\begin{aligned}
			& G_{i \mu}=-G_{i i}\left(X G^{(i)}\right)_{i \mu}=G_{i i} G_{\mu \mu}^{(i)}\left(-x_{i \mu}+\left(X G^{(i \mu)} X\right)_{i \mu}\right) \\
			& G_{\mu i}=-G_{i i}\left(G^{(i)} X\right)_{\mu i}=G_{\mu \mu} G_{i i}^{(\mu)}\left(-x_{\mu i}+\left(X^* G^{(\mu i)} X^*\right)_{\mu i}\right) .
		\end{aligned}
	\end{equation*}
\item[(4).]	 For $r \in \mathcal{I}$ and $s, t \in \mathcal{I} \backslash\{r\}$ we have
	\begin{equation}		\label{eq:3.7}
		G_{s t}^{(r)}=G_{s t}-\frac{G_{s r} G_{r t}}{G_{r r}},\quad
\frac{1}{G_{ss}}=\frac{1}{G_{ss}^{( r)}}-\frac{G_{sr} G_{rs}}{G_{ss} G_{ss}^{( r)} G_{r r}}.
	\end{equation}
	\item[(5).] All of the identities from (1)-(4) hold for $G^{(S)}$ if $S \subset \mathcal{I}_2$ or $S \subset \mathcal{I}.$
\end{enumerate}
\end{lem}
\begin{proof}
See Lemma 4.4 of \cite{AILL}. 
\end{proof}
Next, we introduce the spectral decomposition of $G$. Denote the singular value decomposition of $\Sigma^{1/2}X$ as $
\Sigma^{1 / 2} X=\sum_{k=1}^{n} \sqrt{\lambda_k} \boldsymbol{\xi}_k \boldsymbol{\zeta}_k^*,$ where$
\sqrt{\lambda_1} \geqslant \sqrt{\lambda_2} \geqslant \cdots \geqslant \sqrt{\lambda_{n}}>0$ are the singular values and  $\left\{\boldsymbol{\xi}_k\right\} $ and $\left\{\boldsymbol{\zeta}_k\right\}$ the associated singular vectors. Then according to (\ref{eq:Gblock}), for $\mu, \nu \in \mathcal{I}_2,$ we have
\begin{equation}\label{eq:Gspectraldecom}
	G_{\mu \nu}(z)=\sum_{k=1}^n \frac{\boldsymbol{\zeta}_k(\mu) \overline{\boldsymbol{\zeta}_k(\nu)}}{\lambda_k-z}.
\end{equation}
Similarly, by defining 
$$
\mathbf{u}_k:=\left(\begin{array}{c}
	\mathbf{1}(k \leqslant p) \sqrt{\lambda_k} \boldsymbol{\xi}_k \\
	\mathbf{1}(k \leqslant n) \zeta_k
\end{array}\right) \in \mathbb{R}^{\mathcal{I}}, \ 
\overline{\Sigma}:=\left(\begin{array}{cc}
	\Sigma & 0 \\
	0 & I_n
\end{array}\right) \in \mathbb{R}^{\mathcal{I} \times \mathcal{I}},
$$
we have (see equation (4.15) of \cite{AILL})
\begin{equation}\label{eq:G+Sigma}
	G=-\overline{\Sigma}+\overline{\Sigma}^{1 / 2} \sum_{k=1}^{p} \frac{\mathbf{u}_k \mathbf{u}_k^*}{\lambda_k-z} \overline{\Sigma}^{1 / 2} .
\end{equation}
Next, we provide some preliminary controls for the resolvent $G.$ For the deterministic vectors $\bv \in \mathbb{R}^{\mathcal{I}_1}$ and $\bw \in \mathbb{R}^{\mathcal{I}_2},$ we identity them with their natural embeddings $(\bv \ \ 0)^*$ and $(0 \ \ \bw)^*$ in $\mathbb{R}^{\mathcal{I}}.$ Moreover, for an $\mathcal{I} \times \mathcal{I}$ matrix $A,$ we use the following notations for their generalized entries
	$$
	A_{\mathbf{v w}}:=\bv^*A\bw, \quad A_{\mathbf{v} s}:=\bv^* A \mathbf{e}_s, \quad A_{s \mathbf{v}}:=\mathbf{e}_s^*A \mathbf{v},
	$$
	where $\mathbf{e}_s$ denotes the standard unit vector in the coordinate direction $s$ in $\mathbb{R}^{\mathcal{I}}$. The following deterministic controls will be frequently used in our proofs.  
\begin{lem}\label{lem_someelementarybound}
For $z \in \mathbf{S}$ uniformly in (\ref{eq_spectralparameterset}), we have that for some constant $C>0,$ with high probability 
	\begin{equation}\label{eq:GpartialzG}
		\left\|\overline{\Sigma}^{-1 / 2} G \overline{\Sigma}^{-1 / 2}\right\| \leqslant C{ (\phi^{1/2}+\phi^{-1/2})}  \eta^{-1}, \quad\left\|\overline{\Sigma}^{-1 / 2} \partial_z G \overline{\Sigma}^{-1 / 2}\right\| \leqslant C{ (\phi^{1/2}+\phi^{-1/2})} \eta^{-2} .
	\end{equation}
	Furthermore, for unit vectors $\mathbf{v} \in \mathbb{R}^{\mathcal{I}_1}$ and $\mathbf{w} \in \mathbb{R}^{\mathcal{I}_2}$, then for some constant $C>0$
	\begin{equation}\label{eq:waldid1}
		\sum_{\mu \in \mathcal{I}_2}\left|G_{\mathbf{w} \mu}\right|^2 =\frac{\operatorname{Im} G_{\mathbf{w} \mathbf{w}}}{\eta}, \\
	\end{equation}
	\begin{equation}\label{eq:waldid2}
		\sum_{i \in \mathcal{I}_1}\left|G_{\mathbf{v} i}\right|^2  \leqslant \frac{C(\phi^{1/2}+\phi^{-1/2})}{\eta} \operatorname{Im} G_{\mathbf{v} \mathbf{v}}+2\left(\Sigma^2\right)_{\mathbf{v} \mathbf{v}}, \\
	\end{equation}
	\begin{equation}\label{eq:waldid3}
		\sum_{i \in \mathcal{I}_1}\left|G_{\mathbf{w} i}\right|^2  \leqslant C(\phi^{1/2}+\phi^{-1/2}) \sum_{\mu \in \mathcal{I}_N}\left|G_{\mathbf{w} \mu}\right|^2, \\
	\end{equation}
	\begin{equation}\label{eq:waldid4}
		\sum_{\mu \in \mathcal{I}_2}\left|G_{\mathbf{v} \mu}\right|^2 \leqslant C(\phi^{1/2}+\phi^{-1/2}) \sum_{i \in \mathcal{I}_1}\left|G_{\mathbf{v} i}\right|^2 .
	\end{equation}
	Finally, the estimates  \eqref{eq:waldid1}-\eqref{eq:waldid4} remain true for $G^{(S)}$ if $S \subset \mathcal{I}_2$ or $S \subset \mathcal{I}$.
\end{lem}
\begin{proof}
Note that according to \cite{bai2010spectral} and the scaling (\ref{eq:normalization}), we find that with high probability, for some constant $C>0,$
\begin{equation*}
 \left\| X^* X \right\| \leq C(\phi^{1/2}+\phi^{-1/2}).  
\end{equation*} 
The estimates \eqref{eq:GpartialzG} follow from \eqref{eq:G+Sigma}, using $\|A\|=\sup \{|\langle\mathbf{x}, A \mathbf{y}\rangle|:|\mathbf{x}|,|\mathbf{y}| \leqslant 1\}$, and since {we are in the spectral domain $\mathbf{S}$ so $|\operatorname{Re}z-\lambda_k|\asymp1$},  $\left|\lambda_k\right| /\left|\lambda_k-z\right| \leqslant C{(\phi^{1/2}+\phi^{-1/2})} \eta^{-1}$. Moreover, \eqref{eq:waldid1} directly follows from \eqref{eq:Gblock} and \eqref{eq:Gspectraldecom}.
In order to prove \eqref{eq:waldid2}, we use \eqref{eq:Gblock} to write
$$
\begin{aligned}
	&\sum_{i \in \mathcal{I}_1}\left|G_{\mathbf{v} i}\right|^2=\sum_{i \in \mathcal{I}_1}\left|\left(\Sigma X G X^* \Sigma\right)_{\mathbf{v} i}-\Sigma_{\mathbf{v} i}\right|^2 \leqslant 2\left(\Sigma X G X^* \Sigma^2 X G^* X^* \Sigma\right)_{\mathbf{v v}}+2\left(\Sigma^2\right)_{\mathbf{v v}} \\
	&\leqslant C\left\|X^* X\right\|\left(\Sigma X G_N G_N^* X^* \Sigma\right)_{\mathbf{v v}}+2\left(\Sigma^2\right)_{\mathbf{v v}}=\frac{C\left\|X^* X\right\|}{\eta} \operatorname{Im}\left(\Sigma X G X^* \Sigma\right)_{\mathbf{v v}}+2\left(\Sigma^2\right)_{\mathbf{v v}} \\
	&=\frac{C\left\|X^* X\right\|}{\eta} \operatorname{Im} G_{\mathbf{v v}}+2\left(\Sigma^2\right)_{\mathbf{v v}}.
\end{aligned}$$
In order to prove \eqref{eq:waldid3}, we use \eqref{eq:Gblock} and \eqref{eq:Gspectraldecom} to get
$$
\sum_{i \in \mathcal{I}_1}\left|G_{\mathbf{w} i}\right|^2=\left(G X^* \Sigma^2 X G^*\right)_{\mathbf{w w}} \leqslant C\left(G X^* X G^*\right)_{\mathbf{w w}} \leqslant C\left\|X^* X\right\|\left(G_2 G_2^*\right)_{\mathbf{w w}} .
$$
The estimate \eqref{eq:waldid4} can be proved similarly in the sense that
$$
\sum_{\mu \in \mathcal{I}_N}\left|G_{\mathbf{v} \mu}\right|^2=|z|^{-2}\left(G X X^* G^*\right)_{\mathrm{vv}} \leqslant C\left\|X^* X\right\|\left(G_1 G_1^*\right)_{\mathbf{v v}}.
$$
Finally, the same estimates for $G^{(S)}$ instead of $G$ using minor modifications of the above arguments.
\end{proof}

Finally, we provide some large  deviation estimates. 
\begin{lem}\label{lem_largedervation}
Let $(\xi_i)$ and $(\zeta_i)$ be two independent families satisfying that $\mathbb{E} \xi_i=\mathbb{E} \zeta_i=0, \ \mathbb{E}|\xi_i|^2=\mathbb{E} |\zeta_i|^2=1$, $(\mathbb{E}|\xi_i|^q)^{1/q} \leq C$ and $(\mathbb{E}|\zeta_i|^q)^{1/q} \leq C$ for all $q \in \mathbb{N}$ and some constant $C>0.$ Then for deterministic sequences $(a_{ij})$ and $(b_j),$ we have that 
\begin{align*}
\sum_i b_i \xi_i  \prec \left( \sum_i |b_i|^2 \right)^{1/2}, \  \sum_{i,j} a_{ij} \xi_i \zeta_j \prec \left( \sum_{i,j}|a_{ij}|^2 \right)^{1/2}, \ \sum_{i \neq j} a_i a_j \xi_i \xi_j \prec  \left( \sum_{i \neq j}|a_{ij}|^2 \right)^{1/2}. 
\end{align*}
\end{lem}
\begin{proof}
See Lemma 3.1 of \cite{LL}. 
\end{proof}

\subsection{Entrywise local law and proof of (\ref{eq_averagelocallawone}), (\ref{eq_averagelocallawtwo}) and (\ref{eq:entrywiselaw})}\label{app_sec_entrywiselocalaw}
In this section, we prove the entrywise local law and (\ref{eq_averagelocallawone}), (\ref{eq_averagelocallawtwo}) and (\ref{eq:entrywiselaw}). More specifically, we will prove the following proposition. 
\begin{prop}\label{prop_reducedresult} Suppose the assumptions of Theorem \ref{thm_locallaw} hold. Then we have that (1) of Theorem \ref{thm_locallaw} holds when $\mathbf{u}_k$ and $\mathbf{v}_k$ are the standard basis in $\mathbb{R}^p$ and $\mathbb{R}^n,$ respectively, and (2) and (3) of Theorem \ref{thm_locallaw} hold. 
\end{prop}

We first prepare some notations. For each $i \in \mathcal{I}_1,$ we define
\begin{equation}\label{eq_miform}
	m_i:=\frac{-\sigma_i}{1+\phi^{-1/2} m \sigma_i} .
\end{equation}
Recall \eqref{eq:lsd}. We find that the functions $m$ and $m_i$ satisfy
$$
\frac{1}{m}=-z-{\frac{\phi^{1/2} }{p}} \sum_{i \in \mathcal{I}_1} m_i, \quad \frac{1}{m_i}=-\frac{1}{\sigma_i}-{\phi^{-1/2} }m.
$$
Using (\ref{eq_ratioassumption}), (\ref{eq_constantbound}) and (\ref{eq_miform}), we readily obtain that 
\begin{equation}\label{eq_normestimate}
\left|m_i\right| \asymp \sigma_i \quad \text { for } z \in \mathbf{S} \text { and } i \in \mathcal{I}_1 \text {. }
\end{equation}
For notational simplicity, we extend the definitions of $\sigma_i$ and $m_i$ by setting $\sigma_\mu:=1$ and   $m_\mu:=m$ for $\mu \in \mathcal{I}_2$. Recall (\ref{eq_BIGG}). We define the blockwise random control parameters as follows 
$$
\Lambda_{2}:=\max_{\mu, \nu \in \mathcal{I}_2} \frac{\left|G_{\mu \nu}-\delta_{\mu \nu} m_\mu\right|}{\sigma_\mu \sigma_\nu}, \quad \Lambda_{o,2}=\max _{\mu \neq \nu \in \mathcal{I}_2} \frac{\left|G_{\mu\nu}\right|}{\sigma_\mu \sigma_\nu}, \ 
\Lambda_{1}=\max_{i, j \in \mathcal{I}_1} \frac{\left|G_{i j}-\delta_{i j} m_i\right|}{\sigma_i \sigma_j}, \quad \Lambda_{o,1}=\max _{i \neq j \in \mathcal{I}_1} \frac{\left|G_{ij}\right|}{\sigma_i \sigma_j},
$$
and
$$
\Lambda_{o,12}=\max _{i\in \mathcal{I}_1, \mu \in \mathcal{I}_2} \frac{\left|G_{i\mu }\right|}{\sigma_i \sigma_\mu}.
$$
Armed with the above notations, we further denote 
$$
\Theta_1:=\left|\frac{1}{p} \sum_{i \in \mathcal{I}_1}\left(G_{i i}-m_i\right)\right|, \quad \Theta_2:=\left|\frac{1}{n} \sum_{\mu \in \mathcal{I}_2}\left(G_{\mu \mu}-m\right)\right|=\left|m_{2n}-m\right|.
$$
Moreover, we define { 
	\begin{equation}\label{eq_parameterkeydefinition}
	\Lambda:=\phi^{1/2}\Lambda_{1}+\phi^{1/4}\Lambda_{o,12}+\Lambda_{2}, \ \Lambda_o:=\phi^{1/2}\Lambda_{o,1}+\phi^{1/4}\Lambda_{o,12}+\Lambda_{o,2}, \ \Theta:=\phi^{1/2} \Theta_1+\Theta_2.
	\end{equation}
}
Note that all the above parameters depend on $z$ implicitly and we have the trivial bound $\Theta=\rO\left(\Lambda \right).$ With (\ref{eq_parameterkeydefinition}), we further denote the parameters   
{$$
	\Psi_{\Theta,2}:=\sqrt{\frac{\operatorname{Im} m+\Theta}{n \eta}}, \
	\Psi_{\Theta,1}:=\sqrt{\frac{\operatorname{Im} m+\Theta}{p \eta}}, \
	\Psi_{\Theta,12}:=\sqrt{\frac{\operatorname{Im} m+\Theta}{\sqrt{p n} \eta}},
	$$ }
and the $z$-dependent probability event $\Xi \equiv \Xi(z)$
\begin{equation}\label{eq_probabilityevent}
\Xi=\left\{\Lambda(z) \leqslant(\log n)^{-1}\right\}.
\end{equation}

To prove Proposition \ref{prop_reducedresult}, we basically follow the approach of \cite[Section 5]{AILL} or Section A.2 of \cite{DY} and focus on the details that departs significantly from that of \cite{DY,AILL}.

\subsubsection{Weak entrywise local law}
In this section, we prove a weaker version of entrywise local law. It is an analog of \cite[Proposition 5.1]{AILL} or \cite[Lemma A.12]{DY}. 
\begin{prop}\label{weaklocallaw}
Suppose that the assumptions of Theorem \ref{thm_locallaw} hold. Then $\Lambda \prec(n \eta)^{-1 / 4}$ uniformly in $z \in \mathbf{S}$.
\end{prop}
The rest of this subsection is devoted to the proof of Proposition \ref{weaklocallaw}. For $G$ in (\ref{eq_BIGG}), recall that we set $H=G^{-1}.$ For $s \in \mathcal{I},$ we introduce the conditional expectation
$$
\mathbb{E}_s[\cdot]:=\mathbb{E}\left[\cdot | H^{(s)}\right].
$$
Recall (\ref{eq_G1G2}). According to Lemma \ref{lem_schurcomplement},  we have that for $i \in \mathcal{I}_1$
\begin{equation}\label{eq_firstcontrol}
\frac{1}{G_{i i}}=-\frac{1}{\sigma_i}-{\frac{\phi^{-1/2}}{n}} \operatorname{Tr} G_2^{(i)}-Z_i, \quad Z_i:=\left(1-\mathbb{E}_i\right)\left(X G^{(i)} X^{*}\right)_{i i},
\end{equation}
and  for $\mu \in \mathcal{I}_2$
\begin{equation}\label{eq_secondcontrol}
\frac{1}{G_{\mu \mu}}=-z-{\frac{\phi^{1/2}}{p}} \operatorname{Tr} G_1^{(\mu)}-Z_\mu, \quad Z_\mu:=\left(1-\mathbb{E}_\mu\right)\left(X^{*} G^{(\mu)} X\right)_{\mu \mu}.
\end{equation}

In order to prove Proposition \ref{weaklocallaw}, we start with the following estimate which is analogous to Lemma 5.2 of \cite{AILL} or \cite[Lemma A.9]{DY}. Recall (\ref{eq_probabilityevent}).  
\begin{lem} \label{lem_priorcontrol}
Suppose the assumptions of Proposition \ref{weaklocallaw} hold.	Then for $s \in \mathcal{I}$ and $z \in \mathbf{S},$ we have
	\begin{equation}\label{eq:seq}
		\mathbf{1}(\Xi)\left(\left|Z_s\right|+\Lambda_o\right) \prec \Psi_{\Theta,2},
	\end{equation}	
	and
	\begin{equation}\label{eq:inil}
		\mathbf{1}(\eta \geqslant 1)\left(\left|Z_s\right|+\Lambda_o\right) \prec \Psi_{\Theta,2}.	
	\end{equation}
\end{lem}
\begin{proof}
Due to similarity, we mainly focus on the proof of (\ref{eq:seq}). Using (\ref{eq_normestimate}), (\ref{eq:3.7}) and a simple induction argument, it is not hard to see that 
\begin{equation}\label{eq)nativevound}
\mathbf{1}(\Xi)\left|G_{t t}^{(S)}\right| \asymp \sigma_t,\end{equation}
for any $S \subset \mathcal{I}$ and $t \in \mathcal{I} \backslash S$ satisfying $|S| \leqslant C$. We first work with $\Lambda_o$ in (\ref{eq_parameterkeydefinition}). It suffices to prove that
\begin{equation}\label{eq:priorbound_blockwise}
	\begin{aligned}
		&\mathbf{1}(\Xi)\left|G_{ij}\right| \prec \sigma_i \sigma_j\left(\frac{\operatorname{Im} m+\Theta_2+\Lambda_{o,12}^2}{  p \eta}\right)^{1 / 2},\\ 
		&\mathbf{1}(\Xi)\left|G_{\mu \nu}\right| \prec \sigma_\mu \sigma_\nu\left(\frac{\operatorname{Im} m+\sqrt{\phi}\Theta_1+\Lambda_{o,12}^2}{  n \eta}\right)^{1 / 2},\\ 
		&\mathbf{1}(\Xi)\left|G_{i\mu}\right| \prec \sigma_i\sigma_\mu \left(\frac{\operatorname{Im} m+\Theta_2+\Lambda_{o,2}^2}{  \sqrt{pn} \eta}\right)^{1 / 2}.\\ 
	\end{aligned}
\end{equation}
In fact, using the second estimate in (\ref{eq:priorbound_blockwise}), it is not hard to see that on $\Xi$
$$
\Lambda_{o,2} \prec \Psi_{\Theta,2}+(n \eta)^{-1 / 2} \Lambda_{o,12}.
$$
Moreover, using the first estimate in (\ref{eq:priorbound_blockwise}), we see that on $\Xi$ 
\begin{equation*}
\phi^{1/2} \Lambda_{o,1} \prec \Psi_{\Theta,2}+(n \eta)^{-1/2}\Lambda_{o,12}.   
\end{equation*}
Similarly, using the last estimate in (\ref{eq:priorbound_blockwise}), we have that on $\Xi$ 
\begin{equation*}
\phi^{1/4} \Lambda_{o,12} \prec \Psi_{\Theta,2}+(n \eta)^{-1/2}\Lambda_{o,2}.   
\end{equation*}
Combining the above bounds with the definition $\Lambda_o$ in (\ref{eq_parameterkeydefinition}), we see that  
$\mathbf{1}(\Xi) \Lambda_o \prec \Psi_{\Theta,2}.$ Now we prove (\ref{eq:priorbound_blockwise}).  Let us start with $G_{i j}$ for $i \neq j \in \mathcal{I}_1$. By (\ref{eq:3.9}), (\ref{eq)nativevound}) and Lemma \ref{lem_largedervation}, we find on $\Xi$
\begin{align*}
\left|G_{i j}\right| & \leqslant \left|G_{i i} G_{j j}^{(i)}\right|\left|\sum_{\mu, \nu \in \mathcal{I}_2} x_{i \mu} G_{\mu \nu}^{(i j)} x_{\nu j}\right| \prec  \sigma_i \sigma_j\left(\frac{1}{  np    } \sum_{\mu, \nu \in \mathcal{I}_2}\left|G_{\mu \nu}^{(i j)}\right|^2\right)^{1 / 2} \\
& =\sigma_i \sigma_j \left( \frac{1}{  np  \eta} \sum_{\mu \in \mathcal{I}_2} \operatorname{Im} G_{\mu \mu}^{(i j)} \right)^{1/2} \prec  \left(\frac{1}{   np \eta} \sum_{\mu \in \mathcal{I}_2} \operatorname{Im} G_{\mu \mu}+\frac{\Lambda_{o,12}^2}{  p \eta}\right)^{1/2} \prec \sqrt{\frac{\operatorname{Im} m+\Theta_2+\Lambda_{o,12}^2}{   p \eta}},
\end{align*}
where in the third step we used (\ref{eq:waldid1}) and the fourth step we used (\ref{eq:3.7}). This completes the proof for the first estimate of (\ref{eq:priorbound_blockwise}). Similarly, for 
$ G_{\mu \nu}$ with $\mu \neq \nu \in \mathcal{I}_2,$ using (1) of Lemma \ref{lem_schurcomplement} and (\ref{eq:waldid2}), as well as the bound  $\operatorname{Im} m_i =\rO({\phi^{-1/2}}\sigma_i^2 \operatorname{Im} m)$ for all $i \in \mathcal{I}_1$ which follows from (\ref{eq_miform}), we have that on $\Xi$
\begin{align*}
\left|G_{\mu \nu}\right| & \leqslant \left|G_{\mu \mu} G_{\nu \nu }^{(\mu)}\right|\left|\sum_{i,j\in \mathcal{I}_1} x_{\mu i} G_{ij}^{(\mu \nu)} x_{j \nu}\right| \prec  \sigma_{\mu} \sigma_{\nu}\left(\frac{1}{   np    } \sum_{i, j \in \mathcal{I}_1}\left|G_{ij}^{(\mu\nu)}\right|^2\right)^{1 / 2} \\
& \prec \left( \frac{1}{np} \frac{\phi^{1/2}}{\eta} \sum_{k\in \mathcal{I}_1} \operatorname{Im} G_{kk}^{(\mu\nu)}+\frac{1}{np}\sum_{k\in \mathcal{I}_1}((\Sigma^{(\mu \nu)})^2)_{kk} \right)^{1/2}  \prec \sqrt{\frac{\operatorname{Im} m+\phi^{1/2}\Theta_1+\phi^{1/2}\Lambda_{o,12}^2}{ n \eta}}. 
\end{align*}
This completes the proof for the second estimate of (\ref{eq:priorbound_blockwise}). The last estimate on $\mathbf{1}(\Xi) G_{i \mu}$ with $i \in \mathcal{I}_1$ and $\mu \in \mathcal{I}_N$ can be estimated similarly using (3) of Lemma \ref{lem_schurcomplement} and (\ref{eq:waldid3}) in the sense that 
\begin{align*}
	\left|G_{ i\mu}\right| & =|G_{i i} G_{\mu \mu}^{(i)}|\left|-x_{i \mu}+\left(X G^{(i \mu)} X\right)_{i \mu}\right|\prec \left(\frac{1}{   np   } \sum_{\nu\in \mathcal{I}_2,j\in\mathcal{I}_1}\left|G_{\nu j}^{(i\mu)}\right|^2\right)^{1 / 2}.\\
	\prec & \left( \frac{\phi^{1/2}}{np} \sum_{\nu\in\mathcal{I}_2}\frac{\operatorname{Im} G_{\nu\nu}^{(i\mu)}}{\eta}\right)^{1/2}\prec{\frac{\phi^{1/2}\operatorname{Im}m +\phi^{1/2}\Theta_2+\phi^{1/2}\Lambda_{o,2}^2}{p\eta}}=\frac{\operatorname{Im}m +\Theta_2+\Lambda_{o,2}^2}{\sqrt{np}\eta}. 
\end{align*}
This completes the proof of (\ref{eq:priorbound_blockwise}). 

Then we proceed to the control of $Z_s.$ On the one hand, when $s=i \in \mathcal{I}_1,$ we decompose that 
$$
\left|Z_i\right| \leqslant \left|\sum_{\mu\in\mathcal{I}_2}^{(i)}\left(x_{i \mu}^2-\frac{1}{\sqrt{np}}\right) G_{\mu\mu}^{(i)}\right|+\left|\sum_{\mu \neq \nu}^{(i)} x_{i \mu} G_{\mu \nu}^{(i)} x_{\nu i}\right| .
$$
By Lemma \ref{lem_largedervation} and (\ref{eq:waldid1}), on $\Xi,$ we have that the first term of the above equation can be stochastically dominated by $\left(\sum_\mu^{(i)} (pn)^{-1}\left|G_{\mu \mu}^{(i)}\right|^2\right)^{1 / 2} \prec \frac{1}{\sqrt{p\eta}}$. For the second term, we have
$$
\sum_{\mu, \nu}^{(i)} \frac{1}{np}\left|G_{\mu \nu}^{(i)}\right|^2  \leqslant  \frac{1}{np} \sum_{\mu, \nu}^{(i)} \left|G_{\mu \nu}^{(i)}\right|^2= \frac{1}{np\eta} \sum_\mu^{(i)} \operatorname{Im} G_{\mu\mu}^{(i)} \prec \frac{\operatorname{Im} m+\Theta_2+\Lambda_{o}^2}{p \eta}.
$$
Since $p \gg n,$ this results in 
$$
\left|Z_i\right| \prec \sqrt{\frac{\operatorname{Im} m+\Theta+\Lambda_o^2}{p \eta}} \prec \Psi_{\Theta,2}.
$$
Similarly, for $s=\mu \in \mathcal{I}_2,$ we can show that  
%
%
$$|Z_\mu|\prec \sqrt{\frac{\operatorname{Im} m+\phi^{1/2}\Theta_1+\Lambda_o^2}{n\eta}} \prec \Psi_{\Theta,2}.$$
This completes the proof of (\ref{eq:seq}).

Finally, for the proof of  (\ref{eq:inil}), the argument is similar except we need to use the bound $\| G^{(S)} \|=\rO(1)$ which follows from Lemma \ref{lem_someelementarybound} and $\eta \geq 1$ whenever it is necessary instead of (\ref{eq)nativevound}). This completes our proof. 

\end{proof}

The following estimate is an analog of \cite[Lemma A.10]{DY} or \cite[Lemma 5.3]{AILL}. 
\begin{lem}
Suppose the assumptions of Proposition \ref{weaklocallaw} hold. 
	 Denote
	\begin{equation}\label{eq_Z1Z2definition}
	[Z]_2:=\frac{1}{n} \sum_{\mu \in \mathcal{I}_2} Z_\mu, \quad[Z]_1:=\frac{\phi^{1/2}}{p} \sum_{i \in \mathcal{I}_1} \frac{\sigma_i^2}{\left(1+\phi^{-1/2}m_{2n} \sigma_i\right)^2} Z_i.
	\end{equation}
	Then for $z \in \mathbf{S},$ we have for $f$ defined in \eqref{eq:discrete}
	\begin{equation}\label{eq_bound2xi}
	\mathbf{1}(\Xi)\left(f\left(m_{2n}\right)-z\right)=\mathbf{1}(\Xi)\left([Z]_2+[Z]_1+\mathrm{O}_{\prec}\left(\Psi_{\Theta,2}^2\right)\right),
	\end{equation}
	and 
		\begin{equation}\label{eq_bound2xi2}
	\mathbf{1}(\eta \geq 1)\left(f\left(m_{2n}\right)-z\right)=\mathbf{1}(\eta \geq 1)\left([Z]_2+[Z]_1+\mathrm{O}_{\prec}\left(n^{-1}\right)\right).
	\end{equation}
\end{lem}
\begin{proof}
As before, due to similarity, we focus our proof on (\ref{eq_bound2xi}). Recall that $m_{2n}=n^{-1} \sum_{\mu \in \mathcal{I}_2}G_{\mu \mu}(z).$ Using $G_{\mu \mu}=m_{2n}+\left(G_{\mu \mu}-m_{2n}\right),$ we see that
\begin{equation}\label{eq_spectraldecomposition} 
\frac{1}{G_{\mu \mu}}=\frac{1}{m_{2n}}-\frac{1}{m_{2n}^2}\left(G_{\mu \mu}-m_{2n}\right)+\frac{1}{m_{2n}^2}\left(G_{\mu \mu}-m_{2n}\right)^2 \frac{1}{G_{\mu \mu}} .
\end{equation} 
Moreover, according to (\ref{eq_secondcontrol}), using Lemma \ref{lem_priorcontrol} and  (\ref{eq:3.7}), we see that  
\begin{equation}\label{eq:1Gmu}
	\mathbf{1}(\Xi) \frac{1}{G_{\mu \mu}}=\mathbf{1}(\Xi)\left(-z-\frac{\phi^{1/2}}{p} \sum_{i \in \mathcal{I}_1} G_{i i}-Z_\mu+\mathrm{O}_{\prec}\left({\phi^{1/2}}\Psi_{\Theta,{12}}^2\right)\right),
\end{equation}
and 
\begin{equation}\label{eq_ccccccccccccc}
\mathbf{1}(\Xi)\left|G_{\mu \mu}-m_{2n}\right| \prec \Psi_{\Theta,2}. 
\end{equation}
Similarly, according to (\ref{eq_firstcontrol}), we have that 
\begin{equation}\label{eq:1Gmi}
	\mathbf{1}(\Xi) G_{i i}=\mathbf{1}(\Xi) \frac{-\sigma_i}{1+\phi^{-1/2}m_{2n} \sigma_i+\sigma_i Z_i+\mathrm{O}_{\prec}\left(\sigma_i \phi^{-1/2}\Psi_{\Theta,{12}}^2\right)} .
\end{equation}

On the one hand, combining (\ref{eq_spectraldecomposition}) and (\ref{eq_ccccccccccccc}), we see that 
\begin{equation}\label{eq_formone}
\mathbf{1}(\Xi) \frac{1}{n} \sum_{\mu \in \mathcal{I}_2} \frac{1}{G_{\mu \mu}}=\mathbf{1}(\Xi) \frac{1}{m_{2n}}+\mathrm{O}_{\prec}\left(\Psi_{\Theta,2}^2\right) .
\end{equation}
On the other hand,  plugging \eqref{eq:1Gmi} into \eqref{eq:1Gmu}, we see that 
$$\begin{aligned}
	&\mathbf{1}(\Xi)\frac{1}{n} \sum_{\mu \in \mathcal{I}_2} \frac{1}{G_{\mu \mu}} \\ 
	=&\mathbf{1}(\Xi)\left(-z+\frac{\phi^{1/2}}{p} \sum_{i \in \mathcal{I}_1} \frac{\sigma_i}{1+\phi^{-1/2}m_{2n} \sigma_i+\sigma_i Z_i+\mathrm{O}_{\prec}\left(\sigma_i \phi^{1/2}\Psi_{\Theta,12}^2\right)}-\frac{1}{n} \sum_{\mu \in \mathcal{I}_2} Z_\mu+\mathrm{O}_{\prec}\left(\Psi_{\Theta,2}^2\right)\right) .
\end{aligned}
$$
Moreover, using (\ref{eq)nativevound}) and the assumption $p \gg n$, we see that $1+\phi^{-1/2} m_{2n} \sigma_i \asymp 1.$
Therefore, together with (\ref{eq_formone}), we can conclude the proof of (\ref{eq_bound2xi}). 

The proof of (\ref{eq_bound2xi2}) is similar as discussed in the end of the proof of Lemma \ref{lem_priorcontrol}. We omit further details.  

\end{proof}

The following lemma provides the stability condition for  \eqref{eq:lsd}. Roughly, it says that if $f(m_{2n})-z$ is small and $m_{2n}(\tilde{z})-m(\tilde{z})$ is small for $\tilde{z}:=z+\mathrm{i}{n^{-5}} $, then $m_{2n}(z)-m(z)$ is small. For $z \in \mathbf{S},$ we introduce the discrete set
$$
L(z):=\{z\} \cup\left\{w \in \mathbf{S}: \operatorname{Re} w=\operatorname{Re} z, \operatorname{Im} w \in[\operatorname{Im} z, 1] \cap\left(n^{-5}{ \phi^{-1/2}} \mathbb{N}\right)\right\} .
$$
Thus, if $\operatorname{Im} z \geqslant 1$ then $L(z)=\{z\}$ and if $\operatorname{Im} z \leqslant 1$ then $L(z)$ is a one-dimensional lattice with spacing $n^{-5}{\phi^{-1/2}}$ plus the point $z$. Clearly, we have the bound $|L(z)| \leqslant n^5{\phi^{1/2}}$.
\begin{lem}\label{lem_stabilitybound}
Suppose that $\delta: \mathbf{S} \rightarrow(0, \infty)$ satisfies $n^{-2} \leqslant \delta(z) \leqslant(\log n)^{-1}$ for $z \in \mathbf{S}$ and that $\delta$ is Lipschitz continuous with Lipschitz constant $n^2{\phi^{1/2} }$. Suppose moreover that for each fixed $E$, the function $\eta \mapsto \delta(E+\mathrm{i} \eta)$ is nonincreasing for $\eta>0$. Suppose that $u: \mathbf{S} \rightarrow \mathbb{C}$ is the Stieltjes transform of a probability measure supported in $[0, C]$ for some large enough constant $C>0$. Let $w \in L(z)$ and suppose that
	$$
	|f(u(w))-w| \leqslant \delta(w) .
	$$
Then we have that
	\begin{equation}\label{eq:stability}
		|u(z)-m(z)| \leqslant \frac{C' \delta(z)}{\operatorname{Im} m(z)+\sqrt{\delta(z)}},
	\end{equation}
	for some constant $C'>0$ independent of $z$ and $n.$ 
\end{lem}
\begin{proof}
This lemma can proved with the same method as in e.g. \cite[Lemma 4.5]{LL} and \cite[Appendix A.2]{AILL}. The main inputs are Lemmas \ref{lem_proofgloballawresult} and \ref{lem:mphizabs} and some estimates in their proofs. We omit further details. 
\end{proof}

Armed with the above results, we can follow \cite{LL, DY,AILL} to complete the proof.

\begin{proof}[\bf Proof of Proposition \ref{weaklocallaw}]
When $\eta \geq 1,$ using the above controls and Lemma \ref{lem:mphizabs}, following the same arguments as in \cite[Lemma 4.6]{LL} or \cite[(A.44)-(A.46)]{DY}, we find that  $\Lambda \prec n^{-1/4}. $ It remains to deal with the small $\eta$ case. One can prove this lemma using a continuity argument as in e.g. \cite[Section 4.1]{LL}. The key inputs are Lemmas \ref{lem_priorcontrol}--\ref{lem_stabilitybound}, and the estimates in the $\eta \geq 1$ case. All
the other parts of the proof are essentially the same. We omit the details. 
\end{proof}

\subsubsection{Proof of Proposition \ref{prop_reducedresult}}

To get strong laws as in Proposition \ref{prop_reducedresult}, we need stronger bounds on $[Z]_2$ and $[Z]_1$ in (\ref{eq_Z1Z2definition}). The following lemma is a counterpart of Lemma A.13 of \cite{DY} or Lemma 5.6 of \cite{AILL} or Lemma 4.9 of \cite{LL}. Due to similarity, we only sketch the proof. 

\begin{lem}\label{lem:fa}
Suppose the assumptions of Proposition \ref{prop_reducedresult} hold. Moreover, if $\Upsilon$ is a positive, $n$-dependent, deterministic function on $\mathbf{S}$ satisfying $n^{-1 / 2} \leqslant \Upsilon \leqslant n^{-c}$ for some constant $c>0$. In addition, if we assume that that $\Lambda_{2} \prec n^{-c}$ and $\Lambda_{o,2} \prec \Upsilon$ on $\mathbf{S}$. Then on $\mathbf{S},$ we have
	$$
	\frac{1}{n} \sum_{\mu \in \mathcal{I}_2}\left(1-\mathbb{E}_\mu\right) \frac{1}{G_{\mu \mu}}=\mathrm{O}_{\prec}\left(\Upsilon^2\right).
	$$
	Similarly, for some other positive deterministic function on $\mathbf{S}$ satisfying  $p^{-1 / 2} \leqslant \widetilde{\Upsilon} \leqslant p^{-c}$, suppose that $\Lambda_{1} \prec p^{-c}$ and $\Lambda_{o,1} \prec \widetilde{\Upsilon}$ on $\mathbf{S}$. Then we have 
	$$
	\frac{1}{p} \sum_{i \in \mathcal{I}_1} \frac{\sigma_i^2}{\left(1+\phi^{-1/2}m_{2n} \sigma_i\right)^2}\left(1-\mathbb{E}_i\right) \frac{1}{G_{i i}}=\mathrm{O}_{\prec}\left(\widetilde{\Upsilon}^2\right) .
	$$
\end{lem}
\begin{proof}
As discussed in the proof of Lemma 5.6 of \cite{AILL}, the proof of the first estimate follows from a straightforward extension of Lemma 4.9 of \cite{LL}. For the second estimate, it can be proved similarly using the discussions as explained in Appendix B of \cite{LL}. The only complication is the term $$\frac{\sigma_i^2}{\left(1+\phi^{-1/2}m_{2n} \sigma_i\right)^2}= \frac{\sigma_i^2}{\left(1+\phi^{-1/2}m \sigma_i\right)^2}-\frac{\sigma_i^2[\phi^{-1}(m_{2n}^2-m^2)\sigma_i^2+2\phi^{-1/2}(m_{2n}-m)\sigma_i]}{(1+\phi^{-1/2}m_{2n}\sigma_i)^2(1+\phi^{-1/2}m\sigma_i)^2}.$$
On the one hand, running the argument on $
	\frac{1}{p}\sum_{i\in\mathcal{I}_1}\frac{\sigma_i^2}{\left(1+\phi^{-1/2}m \sigma_i\right)^2}\left(1-\mathbb{E}_i\right) \frac{1}{G_{i i}}=\mathrm{O}_{\prec}\left(\widetilde{\Upsilon}^2\right)
	$ is a straightforward extension of the proof of the first estimate. On the other hand, for the second term, since we have $|m_{2n}-m|\prec (n\eta)^{-1/4}$ from Proposition \ref{weaklocallaw},  we  see that  $$\frac{\sigma_i^2[\phi^{-1}(m_{2n}^2-m^2)\sigma_i^2+2\phi^{-1/2}(m_{2n}-m)\sigma_i]}{(1+\phi^{-1/2}m_{2n}\sigma_i)^2(1+\phi^{-1/2}m\sigma_i)^2}\prec \phi^{-1/2}(n\eta)^{-1/4}.$$ Using Proposition \ref{weaklocallaw}, we see that  
	$$ \frac{1}{p} \sum_{i \in \mathcal{I}_1} \frac{\sigma_i^2[\phi^{-1}(m_{2n}^2-m^2)\sigma_i^2+2\phi^{-1/2}(m_{2n}-m)\sigma_i]}{(1+\phi^{-1/2}m_{2n}\sigma_i)^2(1+\phi^{-1/2}m\sigma_i)^2} (1-\mathbb{E}_i) \frac{1}{G_{ii}} \prec  \phi^{-1/2} (n \eta)^{-1/2}.$$ This completes the proof. For more details, we refer the readers to \cite[Appendix B]{LL}, \cite[Lemma 5.6]{AILL} and \cite[Lemma A.13]{DY}. 
\end{proof}
Armed with Lemma \ref{lem:fa}, we can prove the following iterative estimates which is an analog of Lemma 4.8 of \cite{LL}.

\begin{lem}\label{lem_recursiveestimate}
	 Let $t \in(0,1)$ and suppose that $	\Theta \prec(n \eta)^{-t}	$
	uniformly in $z \in \mathbf{S}$. Then we have that uniformly in $z \in \mathbf{S}$
	$$
	|[Z]_1|+|[Z]_2| \prec \frac{\operatorname{Im} m+(n \eta)^{-t}}{n \eta}. 
	$$
\end{lem}
\begin{proof}
Due to similarity, we focus on $[Z]_2.$ According to (\ref{eq_secondcontrol}) and (1) of Lemma \ref{lem_schurcomplement}, we see that 
\begin{equation}\label{eq_basicidentity}
-Z_u=(1-\mathbb{E}_\mu) \frac{1}{G_{\mu \mu}}.
\end{equation}
We now apply Lemma \ref{lem:fa} to the above identity. According to Proposition \ref{weaklocallaw} and its proof, we conclude that 
$$\Lambda_{o,2} \prec \Upsilon, \quad \Lambda \prec \Upsilon', \quad \Upsilon:=\sqrt{\frac{\operatorname{Im} m+(n \eta)^{-t}}{n \eta}}, \quad \Upsilon':=(n \eta)^{-1 / 4}.$$
Using Lemma \ref{lem:mphizabs}, it is easy to see that the conditions of $\Upsilon$ and $\Upsilon'$ satisfy the conditions of Lemma \ref{lem:fa} so that the result follows from Lemma \ref{lem:fa} and (\ref{eq_basicidentity}). Similar discussions apply to $[Z]_1$ using 
$$\widetilde{\Upsilon}:=\sqrt{\frac{\operatorname{Im} m+ (n\eta)^{-t}}{p\eta}},\quad \widetilde{\Upsilon}'=(p\eta)^{-1/4}.$$
This completes the proof. 
\end{proof}

Armed with Lemma \ref{lem_recursiveestimate} and Proposition \ref{weaklocallaw}, we can prove Proposition \ref{prop_reducedresult} following the arguments in \cite[Section 4.2]{LL} verbatim. Due to similarity, we only sketch the proof.

\begin{proof}[\bf Proof of Proposition \ref{prop_reducedresult}]
We notice that Lemma \ref{lem_recursiveestimate} implies that if $\Theta \prec (n \eta)^{-\tau},$ one may invoke 
\eqref{eq:stability} to obtain that  for some constant $C>0$
$$\Theta_2 \prec \frac{\operatorname{Im} m}{n \eta} \frac{1}{\sqrt{\kappa+\eta}}+\sqrt{\frac{(n \eta)^{-\tau}}{n \eta}}  \leqslant C(n \eta)^{-1 / 2-\tau / 2}.$$
According to Proposition \ref{weaklocallaw}, we can iterate the above results and  push the rate all the way down to 
$\Theta_2 \prec (n\eta)^{-1}$ on $\mathbf{S}$. Similarly, we can obtain the control for $\phi^{1/2} \Theta_1$ and show that $\Theta \prec (n \eta)^{-1}$ using the definition (\ref{eq_parameterkeydefinition}). This completes the proof of (\ref{eq_averagelocallawone}) and (\ref{eq_averagelocallawtwo}). Once the averaged local law is proved, we can prove the entrywise local law (\ref{eq:entrywiselaw}) and (\ref{eq_locallawequationone}) and (\ref{eq_locallawequationtwo}) with standard basis using Proposition \ref{weaklocallaw} and Lemma \ref{lem_schurcomplement}; see Section Section 4.2 of \cite{LL} for more details. This concludes our proof.  
\end{proof}

{
	\subsection{Proof of (\ref{eq_locallawequationone}) and (\ref{eq_locallawequationtwo})}\label{app_sec_prooflocallaw} 
	
It remains to prove (\ref{eq_locallawequationone}) and (\ref{eq_locallawequationtwo}) for general projections $\mathbf{u}$ and $\mathbf{v}.$ Our argument basically follows that of Section 5 of \cite{LL}. We only sketch the proof and refer the readers to \cite{LL} for more details. Due to similarity, we only provide the arguments for the proof of (\ref{eq_locallawequationone}). By linearity and polarization, we see that it suffices to prove that for any fix deterministic unit vector $\mathbf{v} \in \mathbb{R}^p,$ $\mathbf{v}^* R_1 \mathbf{v}=\mathbf{v}^* \Pi \mathbf{v}+\rO_{\prec}(\phi^{-1} \Psi).$ Furthermore, using the definitions  in (\ref{eq:pidefi}), (\ref{eq_twoblockmatrices}), (\ref{eq_G1G2}) and (\ref{eq_miform}), it suffices to prove 
\begin{equation}\label{eq_mainproofquadratic}
\mathbf{v}^* G_1(z) \mathbf{v}=\mathbf{v}^* M \mathbf{v}+\rO_{\prec} \left( \phi^{-1/2} \Psi \right), 
\end{equation}
where for $m_i$ in (\ref{eq_miform}) we denote
\begin{equation*}
M=\operatorname{diag}\{m_1, m_2, \cdots, m_p\}. 
\end{equation*}
The rest of this section is devoted to the proof of (\ref{eq_mainproofquadratic}). Throughout the proof of this section, to make more explicit connections with Section 5 of \cite{LL}, with a little abuse of notations, for (\ref{eq_G1G2}), we use
\begin{equation}\label{eq_lazyconvention}
G \equiv G_1 \in \mathbb{R}^{p \times p}, \ R \equiv G_2 \in \mathbb{R}^{n \times n}. 
\end{equation}
Moreover, we will use $a,b$ for indices valued in $\{1,2,\ldots,p\}$ and $\mathbf{v}=(v_a).$ We decompose that 
\begin{equation*}
\mathbf{v}^* G \mathbf{v}-\mathbf{v}^* M \mathbf{v}=\sum_{a} v_a^2(G_{aa}-m_i)+\mathcal{Z},
\end{equation*} 
where 
\begin{equation}\label{eq_defnZ}
\mathcal{Z}:=\sum_{a \neq b} v_a G_{ab} v_b. 
\end{equation}
Note that the first part can be can be controlled using Proposition \ref{prop_reducedresult} in the following way
\begin{equation}\label{eq_diagonalcontrol}
|\sum_{a} v_a^2(G_{aa}-m_a)| \leq \max_a |G_{aa}-m_a| \prec \phi^{-1/2} \Psi. 
\end{equation} 
It suffices to control $\mathcal{Z}.$ Throughout this section, we will follow Section 5 of \cite{LL} to prove that the following result. Note that due to our linearization as in (\ref{eq_G1G2}) and (\ref{eq:Gblock}), we do not need to rescale the resolvent as in Section 5.1 of \cite{LL}. 
\begin{lem}\label{lem_momentcontrol} Suppose the assumptions of Theorem \ref{thm_locallaw} hold. Then we have that for all $q \in \mathbb{N},$
$$
	\mathbb{E}|\mathcal{Z}|^q \prec \left(\phi^{-1/2} \Psi\right)^q .
	$$
\end{lem}
\begin{proof}
The proof is almost the same as Sections 5.5-5.15 of \cite{LL}. We omit the details. 
\end{proof}
Armed with Lemma \ref{lem_momentcontrol}, we see from Chebyshev's inequality that $\mathcal{Z} \prec \phi^{-1/2} \Psi.$ Combining with (\ref{eq_diagonalcontrol}), we can complete the proof of (\ref{eq_mainproofquadratic}) and hence (\ref{eq_locallawequationone}). 

{

\section{Additional technical proofs}\label{sec_appedixb}

In this section, we prove Lemmas \ref{lem_proofgloballawresult} and \ref{lem:mphizabs}. 

\subsection{Proof of (\ref{eq_constantbound})}\label{appendix_section_priorbound}

Taking imaginary part and real part on both sides of \eqref{eq:lsd}, we have
\begin{align}\label{eq_initialboundargument}
&	\frac{\operatorname{Re} m}{|m|^2}=-E+\int \frac{\phi(\phi^{1/2}x^{-1}+\operatorname{Re} m)}{|\phi^{1/2}x^{-1}+m|^2}\pi(\mathrm{d}x), \nonumber \\ 
&	\frac{-\operatorname{Im} m }{|m|^2}=-\eta+	\int \frac{\phi(-\operatorname{Im} m)}{|\phi^{1/2}x^{-1}+m|^2}\pi(\mathrm{d}x).
\end{align}
For simplicity, we denote $R:=\operatorname{Re}m$ and $I:=\operatorname{Im}m$. We can further rewrite the above equations as 
\begin{equation}\label{eq:realimag}
	\begin{aligned}
		R \int\frac{(\phi^{1/2}x^{-1}+R_{})^2+I_{}^2-\phi(R_{}^2+I_{}^2)}{(R_{}^2+I_{}^2)[(\phi^{1/2}x^{-1}+R_{})^2+I_{}^2]}\pi(\mathrm{d}x)&=-E+\int \frac{\phi^{3/2}x^{-1}}{(\phi^{1/2}x^{-1}+R_{})^2+I_{}^2}\pi(\mathrm{d}x),\\ 
		I \int\frac{(\phi^{1/2}x^{-1}+R_{})^2+I_{}^2-\phi(R_{}^2+I_{}^2)}{(R_{}^2+I_{}^2)[(\phi^{1/2}x^{-1}+R_{})^2+I_{}^2]}\pi(\mathrm{d}x)&=\eta.
	\end{aligned}
\end{equation}
Recall that the Stieltjes transform should map $\mathbb{C}_+$ to $\mathbb{C}_+.$ Therefore, for all $z \in \mathbb{C}_+,$ we see from the second equation of (\ref{eq:realimag}) that 
$$\int\frac{(\phi^{1/2}x^{-1}+R_{})^2+I_{}^2-\phi(R_{}^2+I_{}^2)}{(R_{}^2+I_{}^2)[(\phi^{1/2}x^{-1}+R_{})^2+I_{}^2]}\pi(\mathrm{d}x)>0.$$
This implies that  there must exist $x_0 \in \operatorname{supp}(\pi)$ so that 
$$(\phi^{1/2}x_0^{-1}+R_{})^2+I_{}^2-\phi(R_{}^2+I_{}^2)>0.$$
Set  $t:=|m_{}|=\sqrt{R_{}^2+I_{}^2}$. We see that for all $\phi>0$
$$\begin{aligned}
	(\phi-1)t^2-2\phi^{1/2}x_0^{-1}t-\phi x_0^{-2}&<0.\\ 
\end{aligned}$$ 
Consequently, for $\phi \gg 1,$ we shall have that $\forall z\in\mathbb{C}_+$ 
$$   t<\frac{2\phi^{1/2}x_0^{-1}+2\sqrt{\phi x_0^{-2}+\phi(\phi-1)x_0^{-2}}}{2(\phi-1)}\sim x_0^{-1}(1+\phi^{-1/2})=\mathrm{O}(1).$$
This proves the upper bound for $|m|.$ For the lower bound, equipped with the upper bound, according to (2) of Assumption \ref{assum:XSigma},  it is easy to see from  \eqref{eq:realimag} that when $\phi \gg 1$ 
$$\begin{aligned}
	R_{}\left(\frac{1}{R_{}^2+I_{}^2}-\int x^{2}\pi(\mathrm{d}x)+\mathrm{O}(\phi^{-1/2})\right)&=-E+\int \frac{\phi^{3/2}x^{-1}}{(\phi^{1/2}x^{-1}+R_{})^2+I_{}^2}\pi(\mathrm{d}x),\\ 
	I_{}\left(\frac{1}{R_{}^2+I_{}^2}-\int x^{2}\pi(\mathrm{d}x)+\mathrm{O}(\phi^{-1/2})\right)&=\eta.\\ 
\end{aligned}$$
%
This yields that 
$$\begin{aligned}
(R_{}^2+I_{}^2)\left(\frac{1}{R_{}^2+I_{}^2}-\int x^{2}\pi(\mathrm{d} x)+\mathrm{O}(\phi^{-1/2})\right)^2=\left(-E+\int \frac{\phi^{3/2}x^{-1}}{(\phi^{1/2}x^{-1}+R_{})^2+I_{}^2}\pi(\mathrm{d}x)\right)^2+\eta^2.
\end{aligned}$$
Using the upper bound for $t,$ when $z \in \mathbf{S}$ in (\ref{eq_spectralparameterset}), we see that the right-hand side of the above equation can be bounded by $\mathrm{O}(1).$ Moreover, the left-hand side of the above equation can be rewritten as $\frac{(t^{-2}-\int x^2\pi( \mathrm{d}x))^2}{t^{-2}}.$ It  shows that $t$ is bounded from both above and below so that $m(z) \asymp 1$.  Together with (\ref{eq_initialboundargument}), we see that $\operatorname{Im} m$ is bounded from below by $c\eta.$ This concludes the proof.

\subsection{Proof of Lemma \ref{lem_proofgloballawresult} and (\ref{eq_squarerootbehavior})}\label{appendix_sectionlemma25} 
The proofs are similar to those in \cite{AILL}. Recall $f$ defined in (\ref{eq:discrete}) and the intervals $I_k, 0 \leq k \leq 2p,$ defined below it. We abbreviate $r_i:=\phi\pi(\{s_i\})$. Then we can rewrite $$f(x)=-\frac{1}{x}+\sum_{i=1}^p \frac{r_i}{x+s_i^{-1}}.$$
%
By definition of $f$ we have
\begin{equation}\label{eq:derivative1}
	f^{\prime}(x)=\frac{1}{x^2}-\sum_{i=1}^p \frac{r_i}{\left(x+s_i^{-1}\right)^2}, \quad f^{\prime \prime}(x)=\frac{-2}{x^3}+2 \sum_{i=1}^p \frac{r_i}{\left(x+s_i^{-1}\right)^3} .
\end{equation}

Denote the multiset $\mathcal{C} \subset \mathbb{R}$ of the critical points of $f,$ using the conventions that a nondegenerate
critical point is counted once and a degenerate critical point twice. According to Lemma 2.4 of \cite{AILL}, we have that $|\mathcal{C} \cap I_0|=|\mathcal{C} \cap I_1|=1$ and $|\mathcal{C} \cap I_i| \in \{0,2\}$ for $i=2,\cdots,p.$ This implies that $|\mathcal{C}|=2q$ is even and we denote by $x_1 \geq x_2 \geq \cdots \geq x_{2q-1} \geq x_{2q}$ be these critical points. Moreover, when $\phi \gg 1,$ we find that  for $0<t \leq 1,$
$$\partial_t\partial_xf^t(x)=-\sum_{i=1}^p\frac{\phi t \pi(\{\sigma_i\})x}{(x+(\phi t)^{1/2}\sigma_i^{-1})^3}<0, \ x \in I_2 \cup I_3 \cup \cdots \cup I_p. $$ 
Following lines of the proof of Lemma 2.5 of \cite{AILL}, we have that for $a_k=f(x_k),$ $a_1 \geq a_2 \geq \cdots \geq a_{2q}$ and $x_k=m(a_k).$ 

Note that by definition, $x_{k}\asymp\phi^{1/2}$ for $k\in\{2,\cdots,2q-1\}.$ However, according to (\ref{eq_constantbound}) that $m(z)$ is bounded above and the relation that $f(m(z))=z$, this shows that $q=1$ and we only have two critical points in $I_1$ and $I_0,$ denoted as $x_1$ and $x_2.$ Following lines of the proof of Lemma 2.6 of \cite{AILL}, we can complete the first part of Lemma \ref{lem_proofgloballawresult}. To prove the second part of Lemma \ref{lem_proofgloballawresult} and  (\ref{eq_squarerootbehavior}), we need the following lemma which is an analog of Lemma A.3 of \cite{AILL}.

\begin{lem}
There exists a $\tau>0,$ such that for $k=1,2$ and some constant $C>0$
	$$
	\left|x_{k}\right| \asymp 1, \quad\left|f^{\prime \prime}\left(x_{k}\right)\right| \asymp 1, \quad\left|f^{\prime \prime \prime}(\zeta)\right| \leqslant C \quad \text { for } \quad\left|\zeta-x_{k}\right| \leqslant \tau^{\prime}. 
	$$
\end{lem}
\begin{proof}
First, the above bounds for $x_k$ are clear. The lower bound $\left|x_{k}\right| \geqslant {c} $ follows from $f^{\prime}\left(x_{k}\right)=0$ and \eqref{eq:derivative1}. Second, using 
$$
f^{\prime \prime}\left(x_{k}\right)=-\frac{2}{x_{k}} \sum_{i} \frac{s_{i}^{-1} r_{i}}{\left(x_{k}+s_{i}^{-1}\right)^{3}},
$$
and $|x_k| \asymp 1,$ we see that $|f^{\prime \prime}(x_k)| \asymp 1. $ Finally, the last control follows from 
$$f^{\prime\prime\prime}(x)=\frac{6}{x^4}-6\sum_{i=1}^p\frac{\phi \pi(\{\sigma_i\})}{(x+\phi^{1/2}\sigma_i^{-1})^{4}}.$$
\end{proof}

Using (\ref{eq:discrete}), $x_1, x_2 \asymp 1$ and $\gamma_+=f(x_1), \gamma_-=f(x_2),$ when $\phi \geq 1,$ it is easy to see that $\gamma_+, \gamma_- \asymp \phi^{1/2}.$ Similarly, we can show that $\gamma_+-\gamma_-=f(x_1)-f(x_2)=\mathrm{O}(1).$ This completes the second part of Lemma \ref{lem_proofgloballawresult}. Finally, (\ref{eq_squarerootbehavior}) follows from the following expansion
$$
z-a_{k}=\frac{f^{\prime \prime}\left(x_{k}\right)}{2}\left(m(z)-x_{k}\right)^{2}+\mathrm{O}\left(\left|m(z)-x_{k}\right|^{3}\right) \quad \text { for } \quad\left|z-a_{k}\right| \leqslant \tau^{} .
$$
Due to similarity, we omit more details and refer the readers to Section A.2 of \cite{AILL} or Lemma C.5 of \cite{DYAOS} for more details. 

Finally, we prove (\ref{eq_derivativecontrol}) following that of Lemma 8.5 of \cite{Li2021}.} Using (\ref{eq_eeeeeee}), we obtain that 
	$$E\operatorname{Im}m +\eta\operatorname{Re}m=\frac{\phi^{1/2}}{p}\sum_{i=1}^p\frac{\sigma_i\operatorname{Im} m}{|1+\phi^{-1/2}m\sigma_i|^2},$$
	and
		$$E\operatorname{Re}m -\eta\operatorname{Im}m=\phi-1-\frac{\phi}{p}\sum_{i=1}^p\frac{1+\phi^{-1/2}\sigma_i\operatorname{Im} m}{|1+\phi^{-1/2}m\sigma_i|^2}.$$
This yields that 
		$$\left|z-\frac{\phi^{1/2}}{p}\sum_{i=1}^p\frac{\sigma_i}{|1+\phi^{-1/2}m\sigma_i|^2}\right|=\left| \mathrm{i}\eta-\frac{\operatorname{Re}m}{\operatorname{Im} m}\eta\right|=\frac{|m|\eta}{\operatorname{Im}m}.$$
	Together with (\ref{eq_constantbound}) and \eqref{eq_squarerootbehavior}, we obtain an uppper bound $C\sqrt{\kappa+\eta}$ for the above term.
		Note that
		\begin{align*}
			&\left| \frac{1}{p}\sum_{i=1}^p\frac{\sigma_i}{|1+\phi^{-1/2}m\sigma_i|^2}-\frac{1}{p}\sum_{i=1}^p\frac{\sigma_i}{(1+\phi^{-1/2}m\sigma_i)^2} \right|\\
			=&2\left|  \frac{1}{p}\sum_{i=1}^p \frac{ \sigma_i[(\operatorname{Im}(1+\phi^{-1/2}m\sigma_i))^2 +  \mathrm{i}(\operatorname{Im}(1+\phi^{-1/2}m\sigma_i))(\operatorname{Re}(1+\phi^{-1/2}m\sigma_i))  ]}{|1+\phi^{-1/2}m\sigma|^4}  \right|\le C\phi^{-1/2}\sqrt{\kappa+\eta}.
		\end{align*}
	Combining the above bounds  and (\ref{eq_mmprimeeq1}), we get the upper bound  $C\sqrt{\kappa+\eta}$ for $|\frac{m}{m^{\prime}}|$, it remains to derive a lower bound. 
	When $E\in [\gamma_-,\gamma_+]$, we have
	\begin{align*}
		&\left|z-\frac{\phi^{1/2}}{p}\sum_{i=1}^p\frac{\sigma_i}{(1+\phi^{-1/2}m\sigma_i)^2}\right|\ge  \left|\operatorname{Im}\left(z-\frac{\phi^{1/2}}{p}\sum_{i=1}^p\frac{\sigma_i}{(1+\phi^{-1/2}m\sigma_i)^2}\right)\right|\\ 
		\ge& \left|\eta+\frac{2}{p}\sum_{i=1}^p\frac{\sigma_i(1+\phi^{-1/2}\operatorname{Re}m \sigma_i)\sigma_i\operatorname{Im}m}{|1+\phi^{-1/2}m\sigma_i|^4} \right|\ge \frac{2}{p}\sum_{i=1}^p\frac{\sigma_i^2(1+\phi^{-1/2}\operatorname{Re}m \sigma_i)\operatorname{Im}m}{|1+\phi^{-1/2}m\sigma_i|^4}\ge C{\sqrt{\kappa+\eta}}.
	\end{align*}
	When $E\notin [\gamma_-,\gamma_+]$, we have
	\begin{align*}
		&\left| z-\frac{\phi^{1/2}}{p}\sum_{i=1}^p\frac{\sigma_i}{(1+\phi^{-1/2}m\sigma_i)^2}\right|\ge \left||z|-\frac{\phi^{1/2}}{p}\sum_{i=1}^p\frac{\sigma_i}{|1+\phi^{-1/2}m\sigma_i|^2} \right|\\
		=&\left|\sqrt{E^2+\eta^2}-E-\eta\frac{\operatorname{Re}m}{\operatorname{Im}m}  \right|=\left|\eta\frac{\operatorname{Re}m}{\operatorname{Im}m}-\frac{\eta}{E+\sqrt{E^2+\eta^2}} \right|\ge C\sqrt{\kappa+\eta}.
	\end{align*} 
This completes our proof. 	
\subsection{Proof of (\ref{eq_keyboundbound})}\label{imporveconcentration}
Let $z=E+\mathrm{i} \eta$ and $E \in \mathbf{R}.$ First, when  $\eta \geq \eta_c:=n^{-1+\tau},$ we see that the result follows from Theorem \ref{thm_locallaw}. Then we handle the case when $\eta \leq \eta_c.$ Denote $z_c:=E+\mathrm{i} {\eta}_c$. Since $\operatorname{Tr}R_1(z)=(p-n)/z+\operatorname{Tr}R_2(z),$ it is equivalent to establishing a bound for $\operatorname{Tr}R_2(z)$. In view of \eqref{eq:Gblock}, it suffices to study $(1-\mathbb{E})\sum_{\mu\in\mathcal{I}_2}G_{\mu\mu}$. Using \eqref{eq:Gspectraldecom}, Lemma \ref{lem:mphizabs} and Theorem \ref{thm_locallaw},  we have that 
		
		$$
		\begin{aligned}
			& \left|{G}_{\mu\mu}(z)-{G}_{\mu\mu}\left(z_c\right)\right| \lesssim \sum_{j=1}^n \frac{{\eta}_c\left|\left\langle \mathbf{e}_\mu, \boldsymbol{\zeta}_j\right\rangle\right|\left|\left\langle \mathbf{e}_\mu, \boldsymbol{\zeta}_j\right\rangle\right|}{\left|\left(\lambda_j-E-\mathrm{i} {\eta}\right)\left(\lambda_j-E-\mathrm{i} {\eta}_c\right)\right|} \\
		\leqslant	&  {\eta}_c\left(\sum_{j=1}^n \frac{\left|\left\langle \mathbf{e}_\mu, \boldsymbol{\zeta}_j\right\rangle\right|^2}{\left|\left(\lambda_j-E-\mathrm{i} {\eta}\right)\right|^2}\right)^{1 / 2}\left(\sum_{j=1}^n \frac{\left|\left\langle \mathbf{e}_\mu, \boldsymbol{\zeta}_j\right\rangle\right|^2}{\left|\left(\lambda_j-E-\mathrm{i} {\eta}_c\right)\right|^2}\right)^{1 / 2} \\
		 \leqslant	& {\eta}_c\left(\frac{{\eta}_c^2}{{\eta}^2} \sum_{j=1}^n \frac{\left|\left\langle \mathbf{e}_\mu, \boldsymbol{\zeta}_j\right\rangle\right|^2}{\left|\left(\lambda_j-E-\mathrm{i} {\eta}_c\right)\right|^2}\right)^{1 / 2} \sqrt{\frac{\operatorname{Im}\left[ G_{\mu\mu}\left(z_c\right)\right]}{{\eta}_c}} \\
		=	& \frac{{\eta}_c}{{\eta}} \sqrt{\operatorname{Im}\left[ G_{\mu\mu}\left(z_c\right)\right] \operatorname{Im}\left[G_{\mu\mu}\left(z_c\right)\right]}\prec \frac{1}{n {\eta}}. 
		\end{aligned}
		$$
We readily obtain that		
		$$\begin{aligned}
		&|(1-\mathbb{E})\operatorname{Tr}{R_1}(z)|
		=		|(1-\mathbb{E})\operatorname{Tr}{R_2}(z)|\\  =&|(1-\mathbb{E})\operatorname{Tr}({R_2}(z)-{R_2}(z_c))+ (1-\mathbb{E})\operatorname{Tr}({R_2}(z_c))|\\ 
		\le &\sum_{\mu\in\mathcal{I}_2}
		\left|{G}_{\mu\mu}(z)-{G}_{\mu\mu}\left(z_c\right)\right| +|(1-\mathbb{E})\operatorname{Tr}({R_2}(z_c))| \prec \frac{1}{{\eta_c}}+\frac{1}{|\operatorname{Im}(z)|}\prec\frac{1}{|\operatorname{Im}(z)|}.	
		\end{aligned} $$
This completes our proof.

\bibliographystyle{abbrv}
\bibliography{boot,boot1}

%

\end{document}